\documentclass[journal]{IEEEtran}

\ifCLASSINFOpdf
\usepackage[pdftex]{graphicx}

\graphicspath{{../pdf/}{../jpeg/}}

\DeclareGraphicsExtensions{.pdf,.jpeg,.png}
\else
\usepackage{stfloats}
\usepackage{graphicx}
\usepackage[cmex10]{amsmath}
\graphicspath{{../eps/}}
\allowdisplaybreaks

\fi

\usepackage{color}
\usepackage{amssymb}
\usepackage{amsmath}
\usepackage{amsthm}
\usepackage{bm}
 \usepackage{mathptmx}
\newtheorem{assum}{Assumption}

\newtheorem{remark}{Remark}
\newtheorem{lem}{Lemma}
\newtheorem{thm}{Theorem}

\begin{document}
\title{Vibration Suppression for Coupled Wave PDEs\\ in Deep-sea Construction}

\author{Ji~Wang,~\IEEEmembership{Member,~IEEE},
       and Miroslav~Krstic,~\IEEEmembership{Fellow,~IEEE},
\thanks{
J. Wang and M. Krstic are with Department of Mechanical and Aerospace Engineering, University of California, San Diego, La Jolla, CA 92093-0411, USA(e-mail: jiwang9024@gmail.com;krstic@ucsd.edu).
}}

\maketitle
\begin{abstract}
A deep-sea construction vessel is used to install underwater parts of an off-shore oil drilling platform at the designated locations on the seafloor. By using extended Hamilton's principle, a nonlinear PDE system governing the lateral-longitudinal coupled vibration dynamics of the deep-sea construction vessel consisting of a time-varying-length cable with an attached item is derived, and it is linearized at the steady state generating a linear PDE model, which is extended to a more general system including two coupled wave PDEs connected with two interacting ODEs at the uncontrolled boundaries. Through a preliminary transformation, an equivalent reformulated plant is generated as a $4\times 4$ coupled heterodirectional
hyperbolic PDE-ODE system characterized by spatially-varying coefficients on a time-varying domain.  To stabilize such a system, an observer-based output-feedback control design is proposed, where the measurements are only placed at the actuated boundary of the PDE, namely, at the platform at the sea surface. The exponential stability of the closed-loop system, boundedness and exponential convergence of the control inputs, are proved via Lyapunov analysis.  The obtained theoretical result is tested on a nonlinear model with ocean disturbances, even though the design is developed in the absence of such real-world effects.
\end{abstract}
\begin{IEEEkeywords}
wave PDE, distributed parameter system, boundary control, backstepping, vibration control.
\end{IEEEkeywords}
\section{Introduction}
\paragraph{Deep-sea construction vessels}
In deep-sea oil exploration, a deep-sea construction vessel (DCV)  is an important device used  to install equipment such as a subsea manifold, a subsea pump station, flowlines and so on, at the designated locations around the drill center on the seafloor \cite{Stensgaard2010Subsea,Standing2002Enhancing}.  A dominant component in the DCV is a long cable with time-varying length, of which the top
 is attached to a ship-mounted crane and the bottom is attached to the equipment (referred to as payloads hereafter). Excessive vibration of the cable due to its compliant and lightweight properties is a major problem affecting the payload positioning precision of the DCV. Besides, the excessive
vibration also may cause premature fatigue fracture of the cable especially at the connection point \cite{He2014Modeling}, which would raise the cost of part replacement or maintenance of the DCV. Therefore, vibration suppression for the cable in the DCV is well motivated for improving  performance of the DCV.
\paragraph{Vibration control of string/cable}
Vibration control of a string/cable has received much attention in the past several decades. An active vibration control strategy was designed for  a vibrational string in \cite{zhu2001}, where the actuator and sensor are required to be placed at the interior point of the cable. \cite{He2014Modeling,He2017Boundary} suppressed the undesired vibrations of a moving vibrational string  by applying control inputs at both boundaries.
In \cite{zhang2012modeling}, robust adaptive vibration control was proposed for a disturbed vibrating string of which one boundary is fixed and the opposite one is connected with a payload regulated by the control input.  However, the required actuator layouts of the aforementioned control systems are unsuitable for the DCV where the actuator is only available at the ship-mounted crane, i.e., the top end of the cable.
\paragraph{Boundary control of wave PDE-ODE systems}
The mathematical model of a vibrational string/cable is a wave partial differential equation
(PDE), and the attached payload is modeled as an ordinary differential equation (ODE) at the PDE uncontrolled boundary. Boundary control of a wave PDE-ODE systems where the control input and the unstable ODE are anti-collocated is a more challenging
task than the classical collocated ``boundary damper'' feedback
control. Stabilization of such a wave PDE-ODE system on a fixed domain was presented in \cite{2009compensating,Susto2010Control}. For a wave PDE-nonlinear ODE system, the boundary control problem is also studied in \cite{bekiaris2014compensation,cai2016nonlinear}. References \cite{J2017Axial,J2017Exponential} presented boundary control of a wave PDE-ODE coupled system on a time-varying domain, which physically describes the axial vibration of a mining cable elevator.  Adaptive boundary control of a wave PDE-ODE coupled system with unknown parameters was also considered in \cite{wang2019Adaptive}. The aforementioned works only concern a single wave PDE describing vibrations in one-direction. Challenges appear in the suppression of two-dimensional string vibrations, because there are in-domain couplings between two wave PDEs.
\paragraph{Boundary control of coupled heterodirectional transport PDEs}
A feasible approach to solving the boundary control problem of in-domain coupled wave PDEs is by introducing Riemann transformations to convert the plant to coupled transport PDEs, for which the boundary control problem has been a research focus for the past ten years, with many authors contributing to this topic.
Basic boundary stabilization of a $2\times2$ coupled linear
transport PDEs, i.e., two heterodirectional coupled transport PDEs, by backstepping was  proposed in \cite{Vazquez2011Backstepping,Coron2013Local}. It was further extended to boundary control of a $n+1$ system in \cite{Meglio2013Stabilization}. For a more general coupled linear
transport PDE system where the
number of PDEs in either direction is arbitrary, the boundary stabilization problem
was first addressed in \cite{Hu2016Control} by backstepping, which leads to a systematic framework
for the backstepping-based control of this type of system. Moreover, adaptive control with unknown system parameters or disturbance rejection for external periodic disturbances in coupled heterodirectional transport PDEs, have appeared in \cite{Anfinsen2017Adaptive,Anfinsen2018} and \cite{Deutscher2017Finite-time,Deutscher2017Output,Anfinsen2017Disturbance,Aamo2013Disturbance} respectively. Considering the attached massive payload at the bottom of the cable, boundary control of coupled linear
transport PDEs connected with an ODE at the uncontrolled boundary can also be found in \cite{Meglio2017Stabilization,J2017Control,Deutscher2018Output1}. The aforementioned works on coupled transport PDEs focus on the constant spatial domain rather than a time-varying domain introduced by the time-varying length of the cable.
\paragraph{Contribution}
\begin{itemize}
    \item To the best of our knowledge, our result is the first for boundary control of a two-dimensional coupled vibrating cable of time-varying length, where only one control input is applied at one boundary without requirements of energy dissipation or another controller on  another boundary.
        \item A similar physical problem was considered in \cite{Bohm2014} which presented state-feedback control design while neglecting the couplings between two directions of vibration and assuming the cable length as constant. An observer-based output-feedback  is proposed here, considering the coupled vibrations of a time-varying-length cable.
    \item \cite{Meglio2017Stabilization} first proposed a
full-state feedback law needing measurements of the distributed states for a general coupled linear heterodirectional hyperbolic
PDE-ODE system on a fixed domain. We develop a design for such systems on a time-varying domain, and using only  measurements at the actuated boundary.
\item As compared to our previous results about longitudinal vibration suppression control of cables
 in PDE-modeled mining cable elevators \cite{J2017Axial,J2017Exponential,J2018Balancing}, this paper achieves suppression of longitudinal-lateral coupled vibrations in cables, where the in-domain couplings between wave PDEs, i.e., the couplings between two direction vibrations, make the control design more challenging.
\end{itemize}
For complete clarity, the comparisons with our previous results and other related results from the theoretical and application aspects are summarized in Tab. \ref{tab:com}-\ref{tab:com1} respectively.
\begin{table}[!ht]
\centering
\caption{Comparisons with  application results on boundary vibration control of cables.}\label{tab:com}
\begin{tabular}{lccccccc}
  \hline
& Multi-dimensional  & Time-varying & Number of \\
  & coupled vibrations & cable & controlled/fixed/\\
    &in cables  &lengths  & damped boundaries\\
  \hline
  \cite{He2014Modeling,Liu2017Modeling} & $\times$&$\surd$& 2 \\
  \cite{He2017Boundary} & $\surd$&$\surd$& 2\\
  \cite{Bohm2014}&$\surd$&$\times$&1\\
  \cite{J2017Axial}-\cite{J2018Balancing}&$\times$& $\surd$&1 \\
  This paper  &$\surd$& $\surd$ &1 \\
  \hline
\end{tabular}
\end{table}
\begin{table*}
\centering
\caption{Comparisons with recent theoretical results on boundary control of linear coupled transport PDEs-ODE systems.}\label{tab:com1}
\begin{tabular}{lccccccc}
  \hline
& Types of  & Spatially-varying & Time-varying &Measurements \\
  & hyperbolic PDEs & coefficients & domain& \\
  \hline
  \cite{Meglio2017Stabilization}& $n$ coupled transport PDEs & $\surd$ & $\times$ &Full distributed states  \\
  \cite{Deutscher2018Output1} & $n$ coupled transport PDEs & $\surd$ & $\times$ &States at the uncontrolled boundary\\
  \cite{J2020delay} & $2\times 2$ coupled transport PDEs & $\times$ &$\times$ &Output states of ODE at the uncontrolled boundary\\
   \cite{J2018Balancing} & $2\times 2$ coupled transport PDEs & $\times$ & $\surd$ & States at two boundaries \\
  This paper & $4\times 4$ coupled transport PDEs & $\surd$ & $\surd$ & States at the actuated boundary \\
  \hline
\end{tabular}
\end{table*}
\paragraph{{Organization}}This paper is organized as
follows. In Section \ref{sec:problem}, a nonlinear distributed parameter model governing the longitudinal-lateral vibration dynamics of the DCV is derived by Hamilton's principle, which is then linearized around the steady state and extend to a general plant, based on which the control design would be conducted in the following sections. In Section \ref{sec:Con}, the state-feedback control design is presented and the exponential stability
result of the state-feedback closed-loop system is proved. In Section \ref{sec:observer}, we design a state observer and prove the exponential convergence to zero of observer errors. In Section \ref{sec:output}, we propose an observer-based output-feedback controller and establish the main result in the output-feedback closed-loop system. In Section \ref{sec:Sim},  the obtained theoretical result is tested on a nonlinear model with ocean disturbances, even though the design is developed in the absence of such real-world effects. In Section \ref{sec:Conclusion}, the conclusion and future work are provided.
\paragraph{Notation} Throughout this paper, the partial derivatives and total derivatives are denoted as:
${f_x}(x,t) = \frac{{\partial f}}{{\partial x}}(x,t)$, ${f_t}(x,t) = \frac{{\partial f}}{{\partial t}}(x,t)$,
$f '(x) = \frac{{df (x)}}{{dx}}$, $\dot f(t) = \frac{{df(t)}}{{dt}}$.
\section{Problem Formulation}\label{sec:problem}
\subsection{Modeling}
\begin{figure}
\centering
\includegraphics[width=7.9cm]{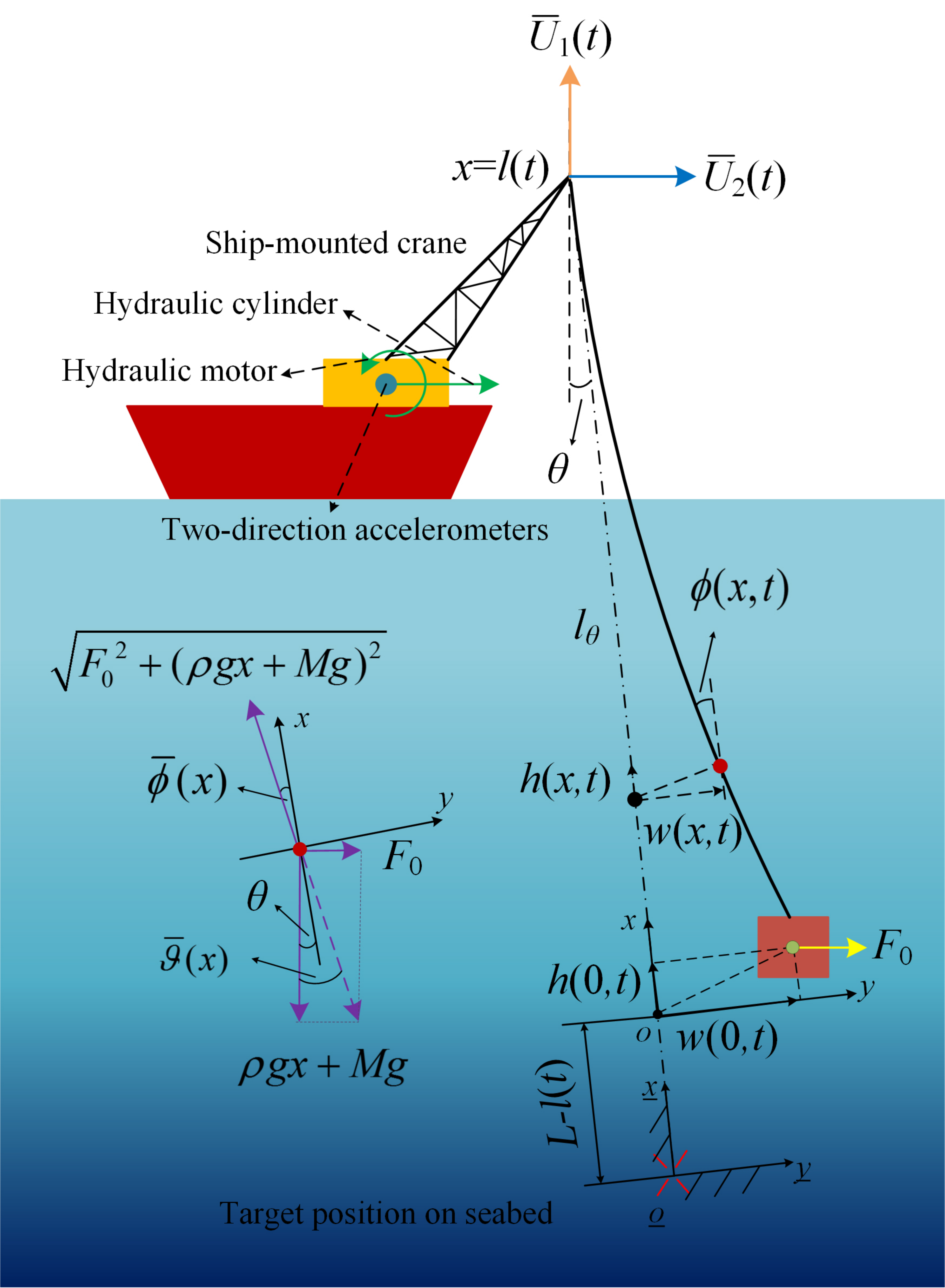}
\caption{Diagram of a deep-sea construction vessel.}
\label{fig:boat}
\end{figure}
DCVs are often used for installation
of underwater parts for an offshore drilling platform, such as a subsea manifold, a subsea pump station, a subsea distribution
unit along with associated foundations, flowlines and
umbilicals \cite{Stensgaard2010Subsea}-\cite{Standing2002Enhancing}. A DCV is depicted in Fig. \ref{fig:boat}, where a crane mounted on a ship regulates a cable to install an equipment (attached payload) at the target position on the seabed. The attached payload is subject to a constant drag force
caused by a constant water stream velocity. We only consider the water-stream caused drag force at the payload because the
diameter of the cable is much smaller than that of the payload. For suppression of cable oscillation/vibration, two-directional control forces implemented by two actuators (hydraulic  cylinder and hydraulic  motor) and measurements by accelerometers are applied/placed at the ship-mounted crane.
Note that cable motion is provided by an additional winch winded by cable on the ship, which can be considered as a  pre-defined time-varying function $l(t)$ of cable length in the cable vibration dynamics regulated by the cranes.

Hamilton's principle \cite{Kaczmarczyk2003} is applied to formulate the mathematical model of the two-dimensional vibrations of the DCV in Fig. \ref{fig:boat}, where we neglect the dynamics of the ship because it can be kept
at the desired position by the ship dynamic positioning
system. The physical parameters of the DCV are shown in Table \ref{table1} and the given values are used in the simulation. To describe the vibrations of the cable, the classical moving frame approach is used \cite{Jalon1994}. In Fig. \ref{fig:boat}, the $xoy$
frame is a moving coordinate associated with cable motion, i.e., the time-varying cable length $l(t)$, moving along line $l_{\theta}$, where $x$ ranges from $x=0$ at the cable bottom to $x=l(t)$ at the top end. The other coordinate
frame $\underline x\underline o\underline y$  is earth fixed. Buoyancy on the cable can easily be included
by adjusting the cable linear density $m_c$ as $\rho=m_c-\rho_w A_a$, and changing the payload mass $M_L$ as $M=M_L-\rho_w V_p$, when calculate the gravity of the payload.

Denote the longitudinal and (in-plane) lateral dynamic deflections in the distributed position $x$ in the cable as $h(x,t)$ and $w(x,t)$, respectively. The deformed position vector, i.e., vibration displacement vector, in the earth fixed coordinate frame $\underline x\underline o\underline y$ is
\begin{align}
D(x,t)=[L-l(t)+x+h(x,t),w(x,t)]^T.
\end{align}
Thus, the velocity vector is
\begin{align}
D_t(x,t)=[h_t(x,t)-\dot l(t),w_t(x,t)]^T.
\end{align}
According to the large displacement
approach \cite{Geradin1994}, the strain at the distributed positions $x$ in the cable is given by
\begin{align}
\varepsilon(x,t)={h_x}(x,t) + \frac{1}{2}{w_x}{{(x,t)}^2}.
\end{align}
Therefore, the kinetic energy $E_k$ and
the potential energy $E_p$ of the system are represented as
\begin{align}
E_k &= \frac{\rho }{2}\int_0^{l(t)} {{D _t}{{(x,t)}^2}} dx + \frac{M_L}{2}{D _t}{(0,t)^2},\label{eq:Ek}\\
E_p &= \frac{1}{2}\int_0^{l(t)} {EA_a\varepsilon {{(x,t)}^2}} dx + \int_0^{l(t)} {T(x,t)\varepsilon (x,t)} dx,\label{eq:Ep}
\end{align}
where $T(x)$ is  static tension given by
\begin{align}
T(x) = \rho gx + Mg.\label{eq:statictension}
\end{align}
Because the long steel cable and the equipment at the bottom of the cable are very heavy and the resulting gravity is much larger than the water-stream caused drag force $F_0=\frac{\rho_w}{2}C_dV_s^2$ \cite{Bohm2014}, the angle $\theta=\arctan\frac{F_0}{\sqrt{\rho gL+Mg}}$ in Fig. \ref{fig:boat} is thus small. Therefore, it is reasonable to assume that the static tension $T(x)$ only results from the mass of the payload and cable.
\begin{table}
\centering
\caption{Physical parameters of the DCV.}
\begin{tabular}{lccc}
\hline
Parameters (units)&values\\ \hline
Final cable length ${L}$ (m) &1210\\
Initial cable length $l(0)$ (m) &250\\
Maximum descending velocity $\bar M$ (m/s) & 10\\
Operation time $t_f$ (s) &120\\
Cable cross-sectional area ${A_a}$ (m$^{2}$) &0.47${\times}$${10^{-3}}$\\
Cable effective Young¡¯s Modulus $E$ (N/m$^{2}$) &7.03$\times$${10^{10}}$\\
Cable linear density $m_c$ (kg/m) &8. 95\\
Payload mass ${M_L}$ (kg) &8000\\
Payload volume $V_p$ ($m^3$)& 5\\
Gravitational acceleration \emph{g} (m/s$^{2}$) &9.8\\
Drag coefficient $C_d$ & 1\\
Stream velocity $V_s$ (m/s) &2\\
Seawater density $\rho_w$ ($kgm^{-3}$) &1024\\
Longitudinal damping coefficient in cable $c_u$ & 0.5\\
Lateral damping coefficient in cable $c_v$ & 0.3\\
Longitudinal damping coefficient at attached payload $c_h$ & 0.5\\
Lateral damping coefficient at attached payload $c_w$ & 0.3\\\hline
\end{tabular}
\label{table1}
\end{table}
The virtual work is
\begin{align}
\delta W =& {\bar U_1}\delta h(l(t),t) + {\bar U_2}\delta w(l(t),t) + c_h({h_t}(0,t) - \dot l(t))\delta h(0,t)\notag\\
 &+ {c_w}{w_t}(0,t)\delta w(0,t)+ \int_0^{l(t)} {{c_v}{w_t}(x,t)\delta } w(x,t)dx\notag\\
 & + \int_0^{l(t)} {c_u}({h_t}(x,t)- \dot l(t))\delta  h(x,t)dx + F_0\delta w(0,t)\notag\\
  & - Mg\delta h(0,t)- \int_0^{l(t)} {\rho g\delta h(x,t)} dx. \label{eq:vwork}
\end{align}
Note that an approximation is adopted in the virtual work \eqref{eq:vwork}, that the virtual work done by the gravity and $U_1$ are parallel to $x$-direction, and the virtual work done by $F_0$ and $U_2$ are parallel to $y$-direction, due to the small $\theta$.

Apply variations of \eqref{eq:Ek}-\eqref{eq:Ep} into extended Hamilton's principle
\begin{align}
\int_{t_1}^{t_2}(\delta E_k-\delta E_p+\delta W)dt=0,\label{eq:Ham}
\end{align}
where a difference from the standard procedure due to the time-varying integration domain should be noted as \eqref{eq:321}-\eqref{eq:123},
\begin{align}
&\rho \int_0^{l(t)} {{D_t}(x,t)\delta {D _t}(x,t)} dx\notag\\
& = \frac{{\rho \int_0^{l(t)} {{D _t}(x,t)\delta D (x,t)} dx}}{{dt}} - \dot l(t)\rho {D_t}(l(t),t)\delta D (l(t),t)\notag\\
 &- \rho \int_0^{l(t)} {{D _{tt}}(x,t)\delta D (x,t)} dx,\label{eq:321}
\end{align}
integrating \eqref{eq:321} from $t_1$ to $t_2$ then yields
\begin{align}
&\int_{{t_1}}^{{t_2}}  \rho \int_0^{l(t)} {{D_t}(x,t)\delta {D_t}(x,t)} dxdt\notag\\
 &=  - \int_{{t_1}}^{{t_2}}  \dot l(t)\rho {D_t}(l(t),t)\delta D (l(t),t)dt \notag\\
 &- \int_{{t_1}}^{{t_2}}  \rho \int_0^{l(t)} {{D_{tt}}(x,t)\delta D (x,t)} dxdt.\label{eq:123}
\end{align}
Note that $l(t)$ is treated as a
pre-defined function and thus it is not necessary
to consider its variation in $\delta {D}(x,t),\delta {D_t}(x,t)$, i.e., $\delta l(t)=\delta\dot{ l}(t)=0$.

Through a lengthy calculation for \eqref{eq:Ham}, the following governing equations are then obtained as
\begin{align}
 &- m_c ({h_{tt}}(x,t) - \ddot l(t)) + EA_a{h_{xx}}(x,t)\notag\\
 &+ {c_u}({h_t}(x,t) - \dot l(t)) + EA_a{w_x}(x,t){w_{xx}}(x,t) =0,\label{eq:n1}\\
&- m_c {w_{tt}}(x,t)+ \frac{3}{2}EA_a{w_x}{(x,t)^2}{w_{xx}}(x,t) + T(x){w_{xx}}(x,t) \notag\\
& +EA_a{h_x}(x,t){w_{xx}}(x,t)+ EA_a{h_{xx}}(x,t){w_x}(x,t) \notag\\
 &+ {c_v}{w_t}(x,t) + T'(x){w_x}(x,t)  =0,\label{eq:n2}\\
&M_L({h_{tt}}(0,t) - \ddot l(t)) +c_h({h_t}(0,t) - \dot l(t)) + EA_a{h_x}(0,t)\notag\\
 &+ \frac{1}{2}EA_a{w_x}{(0,t)^2}=0,\label{eq:n3}\\
&M_L{w_{tt}}(0,t) +{c_w}{w_t}(0,t)+ \frac{1}{2}EA_a{w_x}{(0,t)^3}\notag\\
&   + EA_a{h_x}(0,t){w_x}(0,t)+ T(0){w_x}(0,t) + F_0 = 0,\label{eq:n4}\\
 &- \left(EA_a({h_x}(l(t),t) + \frac{1}{2}{w_x}{(l(t),t)^2}) + T(l(t))\right){w_x}(l(t),t)\notag\\
& - \dot l(t)\rho {w_t}(l(t),t) + {\bar U_2}(t)=0,\label{eq:n5}\\
& - EA_a\left({h_x}(l(t),t) + \frac{1}{2}{w_x}{(l(t),t)^2}\right) - T(l(t))\notag\\
& - \dot l(t)\rho ({h_t}(l(t),t) - \dot l{(t)})  + {\bar U_1}(t)=0.\label{eq:n6}
\end{align}
Note that crane dynamics are neglected and the control input is considered to act on the top end of the cable directly. Incorporating the crane dynamics into the actuation path of the cable would generate an ODE-coupled hyperbolic PDEs-ODE system. Control problem of such a sandwiched system was addressed in \cite{J2020delay} which, however, only dealt with one-dimensional vibrations of DCV with a constant-length cable.

\eqref{eq:n1}-\eqref{eq:n6} is a strongly nonlinear system, so an approximated linear model that is suitable for control design should be built. The nonlinear PDE system \eqref{eq:n1}-\eqref{eq:n6} is linearized in the next subsection around a steady state as in the procedure in \cite{Bohm2014} .
\subsection{Linearization}\label{sec:linear}
The steady states of the distributed strain and pivot angle $\varepsilon(x,t), \phi(x,t)$ can be calculated analytically and expressed as
\begin{align}
\bar\varepsilon(x)&=\frac{1}{EA_a}\sqrt{(\rho gx+Mg)^2+F_0^2},\label{eq:steady1}\\
\bar \phi(x)&=\bar\vartheta(x)-\theta=\arctan\left(\frac{F_0}{\rho gx + Mg}\right)-\theta.\label{eq:steady2}
\end{align}
In nonlinear terms in \eqref{eq:n1}-\eqref{eq:n6}, replacing $h_x(x,t)$ which approximately describes the distributed strain in the cable by $\bar\varepsilon(x)$ , and replacing  $-w_x(x,t)$ which is approximately equal to the pivot angle $\phi(x,t)$ in Fig. \ref{fig:boat} by $\bar \phi(x)$, a linear model around the steady state can be obtained as
\begin{align}
 &- m_c {u_{tt}}(x,t) + EA_a{u_{xx}}(x,t) -EA_a{w_x}(x,t)\bar\phi'(x)\notag\\
 &+ {c_u}{u_t}(x,t) =0,\label{eq:l1}\\
&- m_c {w_{tt}}(x,t)+ \left(\frac{3}{2}EA_a\bar\phi(x)^2 + T(x)\right){w_{xx}}(x,t)\notag\\
& + (EA_a\bar\varepsilon'(x)+\rho g){w_x}(x,t) \notag\\
 &+ {c_v}{w_t}(x,t) -EA_a\bar\phi'(x) {u_x}(x,t)=0,\label{eq:l2}\\
&M_L{u_{tt}}(0,t)  + c_hu_{t}(0,t) + EA_a{u_x}(0,t)\notag\\
&- \frac{EA_a}{2}\bar\phi(0){w_x}{(0,t)}=0,\label{eq:l3}\\
&M_L{w_{tt}}(0,t) + {c_w}{w_t}(0,t)+ \frac{1}{2}EA_a\bar\phi(0)^2{w_x}(0,t)\notag\\
&   + EA_a\bar\phi(0){u_x}(0,t) = 0,\label{eq:l4}\\
 & {w_x}(l(t),t)=\frac{1}{EA_a\bar\varepsilon(l(t))+ \frac{EA_a}{2}\bar\phi(l(t))^2 + T(l(t))}{U_2}(t),\label{eq:l5}\\
& {u_x}(l(t),t) =\frac{1}{EA_a} {U_1}(t),\label{eq:l6}
\end{align}
with defining $u(x,t)=h(x,t) - l(t)$ and
\begin{align}
{U_1}(t)=& {\bar U_1}(t)+ \frac{EA_a}{2}\bar\phi(l(t))^2\notag\\
& - T(l(t))- \dot l(t)\rho {u_t}(l(t),t),\label{eq:bU1}\\
{U_2}(t)=& {\bar U_2}(t)- \dot l(t)\rho {w_t}(l(t),t).\label{eq:bU2}
\end{align}
Note that $Mg{v_x}(0,t) + F_0\approx 0$ in \eqref{eq:n4} via replacing ${v_x}(0,t)$ by the steady state $-\bar\varepsilon(0)\approx-\arctan(\frac{F_0}{Mg})\approx\frac{-F_0}{Mg}$ considering the small angles.

From a practical point
of view, available measurements are acceleration signals $u_{tt}(l(t),t),w_{tt}(l(t),t)$ obtained by accelerometers placed at the crane, because measuring vibrational acceleration rather than velocity/displacement is a more convenient way in vibrational mechanical system. The velocity signals $u_t(l(t),t),w_t(l(t),t)$ can then be obtained by integrations of the measured acceleration signals under known initial conditions.

Therefore, the vibration control design of the DCV corresponds to boundary control of the above coupled wave PDEs \eqref{eq:l1}-\eqref{eq:l6}, characterized by spatially-varying coefficients related to the steady states, the time-varying domain introduced by the time-varying length of the cable and second order boundary conditions describing dynamics of the attached payload. Note that using the known signals $u_t(l(t),t),w_t(l(t),t)$, the designed control laws $U_1(t),U_2(t)$ can be converted to the physical control forces $\bar U_1(t),\bar U_2(t)$ at the ship-mounted crane via \eqref{eq:statictension},\eqref{eq:steady2}.
\subsection{General plant}
In this paper, we represent \eqref{eq:l1}-\eqref{eq:l6} in a more general form which includes more couplings between two wave PDEs, in both the domain and the dynamic boundary, and consider the boundary control problem for this general model. The obtained theoretical result is then applied back to the specific application problem of the DCV, i.e., \eqref{eq:l1}-\eqref{eq:l6} in the simulation.

The concerned plant is
\begin{align}
{w_{tt}}(x,t) &= {d_1}(x){w_{xx}}(x,t) + {d_2}(x){w_x}(x,t) + {d_3}(x){u_x}(x,t)\notag\\
&\quad + {d_{4}}(x){w_t}(x,t) + {d_{5}}(x){u_t}(x,t),\label{eq:o1}\\
{u_{tt}}(x,t) &= {d_6}(x){u_{xx}}(x,t) + {d_7}(x){w_x}(x,t) + {d_{8}}(x){u_x}(x,t)\notag\\
&\quad + {d_{9}}(x){w_t}(x,t) + {d_{10}}(x){u_t}(x,t),\label{eq:o2}\\
{w_{tt}}(0,t) &=  d_{11}{w_t}(0,t)+{d_{12}}{w_x}(0,t)+d_{13}{u_t}(0,t) \notag\\
&\quad + {d_{14}}{u_x}(0,t),\label{eq:o3}\\
{u_{tt}}(0,t) &= d_{15}{u_t}(0,t)+{d_{16}}{u_x}(0,t) +d_{17}{w_t}(0,t)\notag\\
&\quad  + {d_{18}}{w_x}(0,t),\label{eq:o4}\\
{u_x}(l(t),t) &= d_{19}(l(t)){U_1}(t),\label{eq:o5}\\
{w_x}(l(t),t) &= d_{20}(l(t)){U_2}(t),\label{eq:o6}
\end{align}
$\forall (x,t) \in [0,l(t)]\times[0,\infty)$, and assumed measurements are ${u_t}(l(t),t),{w_t}(l(t),t)$ according to the available measurements in the DCV mentioned in Section \ref{sec:linear}. Wave PDEs $w$ and $u$ are coupled with each other both in the domain and at the dynamic boundary. System coefficients ${d_{12}},d_{11},d_{13}$, ${d_{16}},d_{17}$, $d_{15},d_{18},d_{14}$ are arbitrary constants and $d_{19}(l(t)),d_{20}(l(t))$ are nonzero. The spatially-varying coefficients ${d_1}(x)$, ${d_2}(x)$, ${d_3}(x)$, ${d_6}(x)$, ${d_7}(x)$, ${d_{8}}(x)$, ${d_{4}}(x)$, ${d_{5}}(x)$, ${d_{9}}(x)$, ${d_{10}}(x)$ are under the following two assumptions:
\begin{assum}\label{as:ab}
The spatially-varying coefficients ${d_1}(x)$, ${d_2}(x)$, ${d_3}(x)$, ${d_{4}}(x)$, ${d_{5}}(x)$, ${d_6}(x)$, ${d_7}(x)$, ${d_{8}}(x)$,  ${d_{9}}(x)$, ${d_{10}}(x)$ are bounded, $\forall x\in[0,L]$.
\end{assum}
\begin{assum}\label{as:a3}
$d_1(x),d_6(x)\in C^{1}$ are positive and $d_1(x)\neq d_6(x)$, $\forall x\in[0,L]$.
\end{assum}
The time-varying domain, i.e., moving boundary $l(t)$ is under the following two assumptions:
\begin{assum}\label{as:a0}
$l(t)$ is bounded by $0<l(t)\le L$, $\forall t\ge0$.
\end{assum}
\begin{assum}\label{as:a1}
$\dot l(t)$ is bounded by $[\underline {M},\overline {M}]$, where $\overline {M}$ satisfies
\begin{align}
\overline {M}<  \min_{0\le x\le L}\{\sqrt{d_1(x)},\sqrt{d_6(x)}\},
\end{align}
and $\underline {M}$ is arbitrary in $(-\infty,\overline {M})$.
\end{assum}
\begin{remark}
\emph{Assumptions \ref{as:ab}-\ref{as:a1} about the spatially-varying coefficients and the time-varying spatial domain of \eqref{eq:o1}-\eqref{eq:o6} are fully satisfied in the application of the DCV, which can be easily checked by the specific expressions of $d_1,\ldots,d_{20}$ \eqref{eq:q0}-\eqref{eq:q20} and parameter values in Tab. \ref{table1} of the DCV in the simulation.}
\end{remark}
\begin{remark}
\emph{The general plant \eqref{eq:o1}-\eqref{eq:o6} whose diagram  is shown in Fig.\ref{fig:wavecoupled} covers the vibration dynamics of the DCV \eqref{eq:l1}-\eqref{eq:l6} considered in this paper, namely \eqref{eq:l1}-\eqref{eq:l6} being a particular case of \eqref{eq:o1}-\eqref{eq:o6} by setting the coefficients $d_1,\ldots,d_{20}$ as the expressions \eqref{eq:q0}-\eqref{eq:q20} in simulation. Moreover, \eqref{eq:o1}-\eqref{eq:o6}  also can cover the coupled vibration dynamics of mining cable elevators or oil drilling systems.}
\end{remark}
\begin{figure}
\centering
\includegraphics[width=8cm]{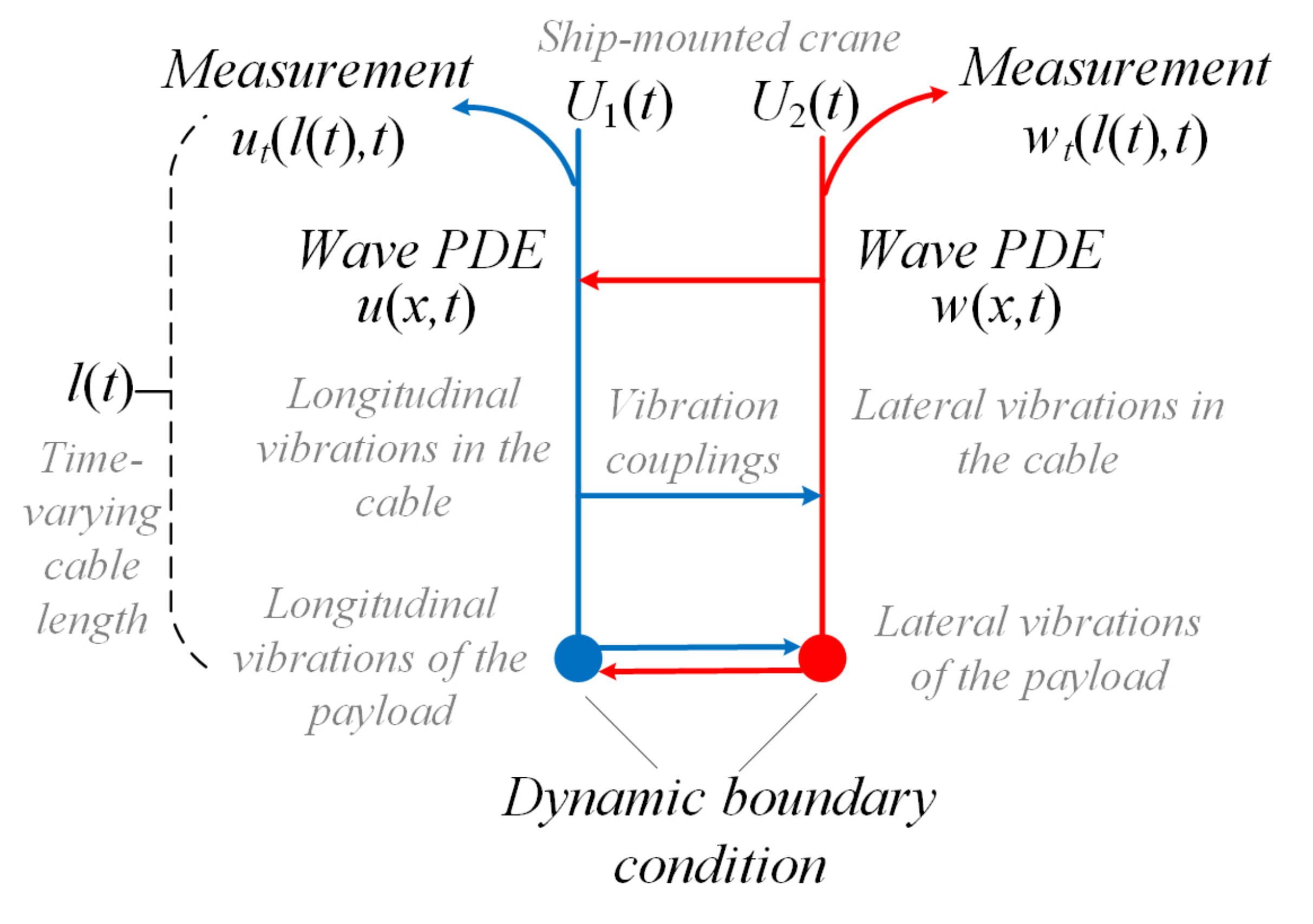}
\caption{Diagram of the plant \eqref{eq:o1}-\eqref{eq:o6}.  The italics text in gray describes the according physical meanings of \eqref{eq:o1}-\eqref{eq:o6} in the specific application of the DCV. }
\label{fig:wavecoupled}
\end{figure}
Our objective is to exponentially stabilize \eqref{eq:o1}-\eqref{eq:o6} through designing control inputs $U_1(t)$, $U_2(t)$ in \eqref{eq:o5}-\eqref{eq:o6} only using boundary values $u_t(l(t),t),w_t(l(t),t)$, i.e., a collocated type output-feedback control system. Note that well-posedness of \eqref{eq:o1}-\eqref{eq:o6} can be seen clearly based on an  equivalent  reformulated plant shown in the next subsection.
\subsection{Reformulated plant}\label{sec:pretran}
A preliminary transformation, which allows one to convert the original plant \eqref{eq:o1}-\eqref{eq:o6} to a reformulated plant based on which the control design will be conducted, will be utilized in the following.

We introduce a set of Riemann transformations:
\begin{align}
z(x,t) = {w_t}(x,t) + \sqrt {{d_1}(x)} {w_x}(x,t),\label{eq:zw}\\
v(x,t) = {w_t}(x,t) - \sqrt {{d_1}(x)} {w_x}(x,t),\label{eq:vw}\\
k(x,t) = {u_t}(x,t) + \sqrt {{d_6}(x)} {u_x}(x,t),\\
y(x,t) = {u_t}(x,t) - \sqrt {{d_6}(x)} {u_x}(x,t)\label{eq:yu}
\end{align}
and define new variables as
\begin{align}
X(t)=[w(0,t),w_t(0,t)],~~Y(t)=[u(0,t),u_t(0,t)],\label{eq:XY}
\end{align}
to reformulate \eqref{eq:o1}-\eqref{eq:o6} as
\begin{align}
&{y_t}(x,t) + \sqrt {{d_6}(x)} {y_x}(x,t) = \left(\frac{{{d_7}(x)}}{{2\sqrt {{d_1}(x)} }}+ \frac{{{d_{9}}(x)}}{2}\right)z(x,t)\notag\\
&+ \left(s_1(x)+ \frac{{{d_{10}}(x)}}{2}\right)k(x,t)+ \left(\frac{{{d_{9}}(x)}}{2}-\frac{{{d_7}(x)}}{{2\sqrt {{d_1}(x)} }}\right)v(x,t)\notag\\
& + \left(\frac{{{d_{10}}(x)}}{2}-s_1(x)\right)y(x,t),\label{eq:no1}\\
&{k_t}(x,t) - \sqrt {{d_6}(x)} {k_x}(x,t) = \left(\frac{{{d_7}(x)}}{{2\sqrt {{d_1}(x)} }}+ \frac{{{d_{9}}(x)}}{2}\right)z(x,t)\notag\\
&+ \left(s_1(x)+ \frac{{{d_{10}}(x)}}{2}\right)k(x,t)+ \left(\frac{{{d_{9}}(x)}}{2}-\frac{{{d_7}(x)}}{{2\sqrt {{d_1}(x)} }}\right)v(x,t)\notag\\
& + \left(\frac{{{d_{10}}(x)}}{2}-s_1(x)\right)y(x,t),\\
&{v_t}(x,t) + \sqrt {{d_1}(x)} {v_x}(x,t) = \left(s_2(x)+ \frac{{{d_{4}}(x)}}{2}\right)z(x,t) \notag\\
&+ \left(\frac{{{d_3}(x)}}{{2\sqrt {{d_6}(x)} }}+\frac{{{d_{5}}(x)}}{2}\right)k(x,t)+ \left(\frac{{{d_{4}}(x)}}{2}-s_2(x)\right)v(x,t)\notag\\
&  + \left(\frac{{{d_{5}}(x)}}{2}-\frac{{{d_3}(x)}}{{2\sqrt {{d_6}(x)} }}\right)y(x,t),\\
&{z_t}(x,t) - \sqrt {{d_1}(x)} {z_x}(x,t) = \left(s_2(x)+ \frac{{{d_{4}}(x)}}{2}\right)z(x,t) \notag\\
&+ \left(\frac{{{d_3}(x)}}{{2\sqrt {{d_6}(x)} }}+\frac{{{d_{5}}(x)}}{2}\right)k(x,t)+ \left(\frac{{{d_{4}}(x)}}{2}-s_2(x)\right)v(x,t)\notag\\
&  + \left(\frac{{{d_{5}}(x)}}{2}-\frac{{{d_3}(x)}}{{2\sqrt {{d_6}(x)} }}\right)y(x,t),\\
&v(0,t) = 2{C_2}X(t)-z(0,t) ,\\
&y(0,t) = 2{C_2}Y(t)-k(0,t),\\
&\dot X(t) = \left(A_1-\frac{{{B_1}{d_{12}}C_2}}{{\sqrt {{d_1}(0)} }}\right)X(t)\notag\\
 &+ \left({{d_{13}}}{B_1}{C_2}-\frac{{{B_1}{d_{14}}C_2}}{{\sqrt {{d_6}(0)} }}\right)Y(t) \notag\\
& + \frac{{{B_1}{d_{12}}}}{{\sqrt {{d_1}(0)} }}z(0,t)+ \frac{{{B_1}{d_{14}}}}{{\sqrt {{d_6}(0)} }}k(0,t),\\
&\dot Y(t) =\left( A_2-\frac{{{B_1}{d_{16}}C_2}}{{\sqrt {{d_6}(0)} }}\right)Y(t)\notag\\
&+\left(d_{17}{B_1}{C_2}-\frac{{{B_1}{d_{18}}C_2}}{{\sqrt {{d_1}(0)} }}\right)X(t)\notag\\
& + \frac{{{B_1}{d_{18}}}}{{\sqrt {{d_1}(0)} }}z(0,t)+ \frac{{{B_1}{d_{16}}}}{{\sqrt {{d_6}(0)} }}k(0,t),\\
&k(l(t),t) = 2\sqrt {{d_6}(l(t))} d_{19}(l(t)){U_1}(t) + y(l(t),t),\\
&z(l(t),t)= 2\sqrt {{d_1}(l(t))} d_{20}(l(t)){U_2}(t) + v(l(t),t) \label{eq:no10}
\end{align}
where
\begin{align}
s_1(x)=\frac{{{d_{8}}(x) - \sqrt {{d_6}(x)} {{\sqrt {{d_6}(x)} }^\prime }}}{{2\sqrt {{d_6}(x)} }},\\
s_2(x)=\frac{{{d_2}(x) - \sqrt {{d_1}(x)} {{\sqrt {{d_1}(x)} }^\prime }}}{{2\sqrt {{d_1}(x)} }},
\end{align}
and
\begin{align}
&A_1=\left[ {\begin{array}{*{20}{c}}
0&1\\
0&{{d_{11}}}
\end{array}} \right],A_2=\left[ {\begin{array}{*{20}{c}}
0&1\\
0&{{d_{15}}}
\end{array}} \right],\label{eq:AB}\\&B_1=\left(
                           \begin{array}{c}
                             0 \\
                             1 \\
                           \end{array}
                         \right),C_2=\left(
                                       \begin{array}{cc}
                                         0 & 1 \\
                                       \end{array}
                                     \right).\label{eq:C2}
\end{align}
The diagram of the system \eqref{eq:no1}-\eqref{eq:no10} is depicted in Fig. \ref{fig:4b4}.
\begin{figure}
\centering
\includegraphics[width=7.5cm]{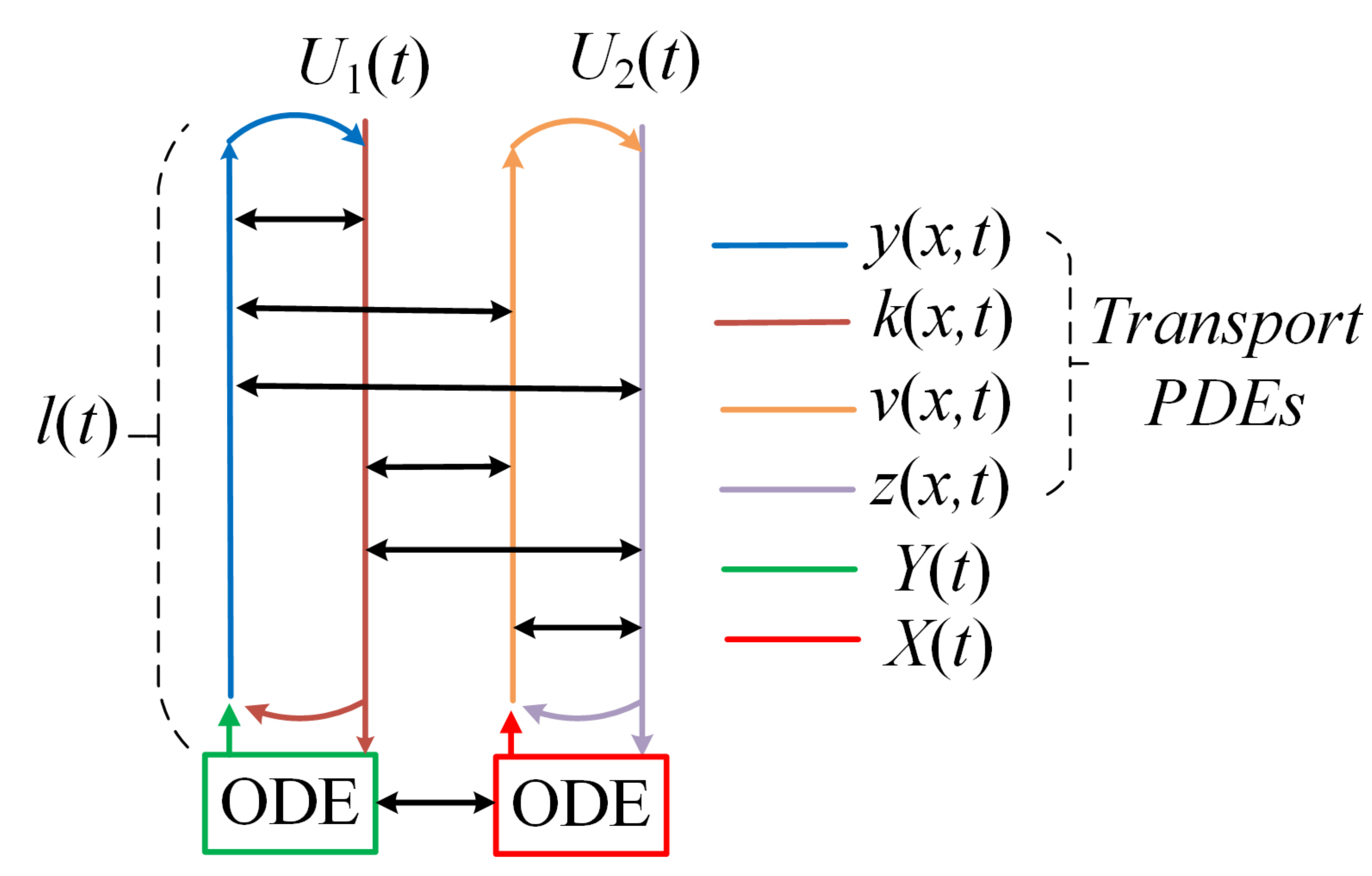}
\caption{Diagram of the system \eqref{eq:no1}-\eqref{eq:no10}.}
\label{fig:4b4}
\end{figure}
\eqref{eq:no1}-\eqref{eq:no10} is an unstable $4\times 4$ coupled linear
heterodirectional hyperbolic PDE-ODE system, where $y(x,t),v(x,t)$ and $k(x,t),z(x,t)$ propagate in opposite direction, and the four transport PDEs are coupled with each other in the time-varying spatial domain $[0,l(t)]$ and coupled with two ODEs $X(t),Y(t)$ which are coupled with each other as well at the uncontrolled boundary $x=0$. Compared with our previous result in \cite{J2018Balancing}, the two pairs of $2\times 2$ coupled linear
heterodirectional hyperbolic PDE-ODE system is extended to one $4\times 4$ system because additional in-domain couplings between original wave PDEs are considered. Moreover, the assumption that some ODE states at the uncontrolled boundary are measurable \cite{J2018Balancing} is removed in this paper.

In order to rewrite \eqref{eq:no1}-\eqref{eq:no10} in a compact form, i.e., in a matrix representation, we define new variables as
\begin{align}
p(x,t) &= {[y(x,t),v(x,t)]^T},\label{eq:dep}\\
r(x,t) &= {[k(x,t),z(x,t)]^T},\label{eq:kz}\\
W(t) &= {[X(t),Y(t)]^T}.\label{eq:deW}
\end{align}
\eqref{eq:no1}-\eqref{eq:no10} can then be rewritten as
\begin{align}
&{p_t}(x,t) + Q(x){p_x}(x,t) = {T_a}(x)r(x,t) + {T_b}(x)p(x,t),\label{eq:p1}\\
&{r_t}(x,t) - Q(x){r_x}(x,t) = {T_a}(x)r(x,t) + {T_b}(x)p(x,t),\label{eq:p2}\\
&p(0,t) = {C_3}W(t) - r(0,t),\\
&\dot W(t) = (\bar A - \bar B{C_3})W(t) + 2\bar Br(0,t),\label{eq:p4}\\
&r(l(t),t) = R(l(t))U(t) + p(l(t),t)\label{eq:p5}
\end{align}
where $U(t)=[U_1(t),U_2(t)]^T$, $R(l(t))=2{\rm{diag}}({{\sqrt {{d_6}(l(t))}d_{19}(l(t)) }},{{\sqrt {{d_1}(l(t))}d_{20}(l(t)) }})$, $Q(x)={\rm{diag}}\{Q_1(x),Q_2(x)\}={\rm{diag}}\{\sqrt {{d_6}(x)} ,\sqrt {{d_1}(x)} \}$. $\bar A ,{C_3},\bar B$ and ${T_a}(x)=\{T_{aij}(x)\}_{1\le i,j\le 2}$, ${T_b}(x)=\{T_{bij}(x)\}_{1\le i,j\le 2}$ are shown as follows
\begin{align}
&\bar A = \left[ {\begin{array}{*{20}{c}}
{{A_1}}&{{{d_{13}}}{B_1}{C_2}}\\
{{d_{17}}{B_1}{C_2}}&{{A_2}}
\end{array}} \right],{C_3} = 2\left[ {\begin{array}{*{20}{c}}
0&{{C_2}}\\
{{C_2}}&0
\end{array}} \right],\label{eq:bA}\\
&\bar B = \left[ {\begin{array}{*{20}{c}}
{\frac{{{B_1}{d_{14}}}}{{2\sqrt {{d_6}(0)} }}}&{\frac{{{B_1}{d_{12}}}}{{2\sqrt {{d_1}(0)} }}}\\
{\frac{{{B_1}{d_{16}}}}{{2\sqrt {{d_6}(0)} }}}&{\frac{{{B_1}{d_{18}}}}{{2\sqrt {{d_1}(0)} }}}
\end{array}} \right] = \left[ {\begin{array}{*{20}{c}}
0&0\\
{\frac{{{d_{14}}}}{{2\sqrt {{d_6}(0)} }}}&{\frac{{{d_{12}}}}{{2\sqrt {{d_1}(0)} }}}\\
0&0\\
{\frac{{{d_{16}}}}{{2\sqrt {{d_6}(0)} }}}&{\frac{{{d_{18}}}}{{2\sqrt {{d_1}(0)} }}}
\end{array}} \right]
\end{align}
\begin{align}
&{T_a}(x) = \left[ {\begin{array}{*{20}{c}}
s_1(x) + \frac{{d_{10}}(x)}{2}&{\frac{{{d_7}(x)}}{{2\sqrt {{d_1}(x)} }} + \frac{{d_{9}}(x)}{2}}\\
{\frac{{{d_3}(x)}}{{2\sqrt {{d_6}(x)} }} + \frac{{d_{5}}(x)}{2}}&{s_2(x)+ \frac{{d_{4}}(x)}{2}}
\end{array}} \right],\label{eq:Ta}\\
&{T_b}(x) = \left[ {\begin{array}{*{20}{c}}
{\frac{{d_{10}}(x)}{2} - s_1(x) }&{\frac{{d_{9}}(x)}{2} - \frac{{{d_7}(x)}}{{2\sqrt {{d_1}(x)} }}}\\
{\frac{{d_{5}}(x)}{2} - \frac{{{d_3}(x)}}{{2\sqrt {{d_6}(x)} }}}&{\frac{{d_{4}}(x)}{2} - s_2(x)}
\end{array}} \right].\label{eq:Tb}
\end{align}
The following assumption is required for the stabilization
of the ODE subsystem \eqref{eq:p4} in the state-feedback and
observer design, and is also satisfied in the DCV model by checking the system parameters in the simulation.
\begin{assum}\label{as:a2}
$(\bar A - \bar B{C_3},\bar B)$ is controllable and $(\bar A - \bar B{C_3}, {C_3})$ is observable .
\end{assum}
\begin{remark}
\emph{The reformulated plant  \eqref{eq:p1}-\eqref{eq:p5} obtained  from \eqref{eq:o1}-\eqref{eq:o6} via the preliminary transformation in Section \ref{sec:pretran} is well-posed, because it is analogous to the well-posed plant in \cite{Meglio2017Stabilization} with setting $m=n=2$, which indicates that the original plant \eqref{eq:o1}-\eqref{eq:o6} is also well-posed because of the invertible preliminary transformation. }
\end{remark}
As compared to \cite{Meglio2017Stabilization} where state-feedback control design needing distributed states on a constant PDE domain was presented, in the following sections, we develop an observer-based output-feedback controller upon \eqref{eq:p1}-\eqref{eq:p5} on a time-varying PDE domain, and only collocated boundary states are assumed measurable. Then the control input and the stability result will be formulated back in the original plant \eqref{eq:o1}-\eqref{eq:o6}.
\section{State-feedback control design}\label{sec:Con}
\subsection{Backstepping design}
We seek an invertible transformation that converts the system \eqref{eq:p1}-\eqref{eq:p5} $(p(x,t),r(x,t),W(t))$ into a so called target system whose exponential stability is obvious.

We postulate the backstepping transformation in the form
\begin{align}
\alpha (x,t) &= p(x,t) - \int_0^xK(x,y)p(y,t)dy\notag\\
&\quad-  \int_0^x {J(x,y)r(y,t)dy - } \gamma (x)W(t),\label{eq:t1}\\
\beta (x,t) &= r(x,t) - \int_0^x F(x,y)p(y,t)dy \notag\\
&\quad-  \int_0^x {N(x,y)r(y,t)dy - } \lambda (x)W(t),\label{eq:t2}
\end{align}
where
\begin{align}
K(x,y)=\{K_{ij}(x,y)\}_{1\le i,j\le 2},J(x,y)=\{J_{ij}(x,y)\}_{1\le i,j\le 2},\notag\\
F(x,y)=\{F_{ij}(x,y)\}_{1\le i,j\le 2},N(x,y)=\{N_{ij}(x,y)\}_{1\le i,j\le 2},\notag
\end{align}
on the triangular domain $\mathcal D=\{0\le y\le x\le l(t)\}$, and $\gamma (x)=\{\gamma_{ij}(x)\}_{1\le i\le 2,1\le j\le 4}$, $ \lambda (x)=\{ \lambda_{ij}(x)\}_{1\le i\le 2,1\le j\le 4}$ are to be determined.

The target system $(\alpha(x,t),\beta(x,t),W(t))$ is designed as
\begin{align}
&\dot W(t) = \hat AW(t) + 2\bar B\beta (0,t),\label{eq:tar1}\\
&{\alpha _t}(x,t){\rm{ = }} - Q(x){\alpha _x}(x,t)\notag\\
&\quad\quad\quad\quad + {\bar T_b}(x)\alpha (x,t)+g_1(x)\beta(0,t),\label{eq:tar2}\\
&{\beta _t}(x,t) = Q(x){\beta _x}(x,t) + {\bar T_a}(x)\beta (x,t)+g(x)\beta(0,t),\label{eq:tar3}\\
&\alpha (0,t) =  - \beta (0,t),\label{eq:tar4} \\
&\beta (l(t),t) =0\label{eq:tar5}
\end{align}
where $\hat A=\bar A - \bar B{C_3} + 2\bar B\kappa$ is a Hurwitz matrix by choosing $\kappa=\{\kappa_{ij}\}_{1\le i\le 2,1\le j\le 4}$ recalling Assumption \ref{as:a2}. Note that ${\bar T_a}(x),{\bar T_b}(x)$ are diagonal matrices consisting of the diagonal elements of  \eqref{eq:Ta}-\eqref{eq:Tb}, denoted as ${\bar T_a}(x)={\rm {diag}}({\bar T_{ai}}(x))$ and ${\bar T_b}(x)={\rm {diag}}({\bar T_{bi}}(x))$ for $i=1,2$.  Thus all coupling terms in the PDE domain \eqref{eq:p1}-\eqref{eq:p2}  are removed. $g(x),g_1(x)$ are in the form of
\begin{align}
g(x)=\left(
                             \begin{array}{cc}
                               0 & 0 \\
                               g_a(x) & 0 \\
                             \end{array}
                           \right),g_1(x)=\left(
                             \begin{array}{cc}
                               0 & 0 \\
                               g_b(x) & 0 \\
                             \end{array}
                           \right)
\end{align}
where  $g_a(x)=\sqrt {{d_6}(0)} {F_{21}}(x,0) - \frac{{{{\lambda _{22}}(x)d_{14}}}}{{\sqrt {{d_6}(0)} }} - \frac{{{\lambda _{24}}(x){d_{16}}}}{{\sqrt {{d_6}(0)} }}+ \sqrt {{d_6}(0)} {N_{21}}(x,0)$ and $g_b(x)=\sqrt {{d_6}(0)} {J_{21}}(x,0) - \frac{{{\gamma _{22}}(x){d_{14}}}}{{\sqrt {{d_6}(0)} }} - \frac{{{\gamma _{24}}(x){d_{16}}}}{{\sqrt {{d_6}(0)} }}+ \sqrt {{d_6}(0)} {K_{21}}(x,0)$.
\eqref{eq:tar1}-\eqref{eq:tar5} is exponentially stable, as we will see in the stability analysis via Lyapunov function in Theorem \ref{main}.

By matching the systems \eqref{eq:p1}-\eqref{eq:p5} and \eqref{eq:tar1}-\eqref{eq:tar5} through \eqref{eq:t1}-\eqref{eq:t2}, a lengthy but straightforward
calculation leads to the conditions on the kernels in \eqref{eq:t1}-\eqref{eq:t2} as follows. $F(x,y)$, $N(x,y)$  $\lambda(x)$ should satisfy the matrix equations:
\begin{align}
&Q(x)F(x,x) + F(x,x)Q(x)=-{T_b}(x), \label{eq:Kerc1}\\
&Q(x)N(x,x) - N(x,x)Q(x)=\bar T_a(x)-T_a(x),\label{eq:Kerc2}\\
& N(x,0)Q(0)=-F(x,0)Q(0) + 2\lambda (x)\bar B +g(x),\label{eq:Kerc3}\\
& Q(x){N_x}(x,y)+ {N_y}(x,y)Q(y)+ N(x,y)Q'(y)   \notag\\
& - N(x,y){T_a}(y)+ {\bar T_a}(x)N(x,y)- F(x,y){T_a}(y)   = 0,\label{eq:Kerc4}\\
&  Q(x){F_x}(x,y) - {F_y}(x,y)Q(y)- F(x,y)Q'(y)   \notag\\
& - F(x,y){T_b}(y)+ {\bar T_a}(x)F(x,y)- N(x,y){T_b}(y) = 0,\label{eq:Kerc5}\\
&Q(x)\lambda '(x) - \lambda (x)(\bar A - \bar B{C_3}) + {\bar T_a}(x)\lambda (x)\notag\\
& - F(x,0)Q(0){C_3}+g(x)\lambda(0) = 0,\label{eq:Kerc5a}\\
&\lambda(0)=\kappa\label{eq:Kerc6}
\end{align}
and $K(x,y)$, $J(x,y)$, $\gamma(x)$ should satisfy
\begin{align}
& - Q(x)J(x,x) - J(x,x)Q(x) =- {T_a}(x),\label{eq:Kerc1a}\\
&K(x,x)Q(x) - Q(x)K(x,x)=\bar T_b(x)-T_b(x),\\
&K(x,0)Q(0)=-J(x,0)Q(0) + 2\gamma (x)\bar B +g_1(x),\\
&  - Q(x){J_x}(x,y) + {J_y}(x,y)Q(y)+ J(x,y)Q'(y) \notag\\
& - J(x,y){T_a}(y) + {\bar T_b}(x)J(x,y)- K(x,y){T_a}(y)=0,\\
& - Q(x){K_x}(x,y) - {K_y}(x,y)Q(y)- K(x,y)Q'(y)  \notag\\
&- K(x,y){T_b}(y) + {\bar T_b}(x)K(x,y)- J(x,y){T_b}(y) =0,\\
&Q(x)\gamma '(x) +\gamma (x)(\bar A - \bar B{C_3}) +  {\bar T_b}(x)\gamma (x)\notag\\
& - K(x,0)Q(0){C_3}+g_1(x)\lambda(0)=0,\\
&\gamma (0)= {C_3}-\lambda (0). \label{eq:Kerc6a}
\end{align}
In order to ensure the existence of a unique solution of the above kernel equations, additional artificial boundary conditions should be imposed, which is shown in the proof of the following lemma about well-posedness of the  kernel equations.
\begin{lem}\label{lm:wp}
After adding two additional artificial boundary conditions for the subelements $N_{21}, K_{21}$ of kernels $N$ and $K$ respectively, the matrix equations \eqref{eq:Kerc1}-\eqref{eq:Kerc6} have a unique solution ${F},{N}\in L^{\infty}(\mathcal D)$, $\lambda\in L^{\infty}([0,l(t)])$ and \eqref{eq:Kerc1a}-\eqref{eq:Kerc6a} have a unique solution ${K},{J}\in L^{\infty}(\mathcal D)$,$\gamma \in L^{\infty}([0,l(t)])$.
\end{lem}
\begin{proof}
The proof is shown in the Appendix.
\end{proof}
With similar derivations, one can show that the inverse transformations are defined as
\begin{align}
 p(x,t) &= \alpha(x,t) - \int_0^x\bar K(x,y)\alpha(y,t)dy\notag\\
& -  \int_0^x {\bar J(x,y)\beta(y,t)dy - } \bar\gamma (x)W(t),\label{eq:It1}\\
 r(x,t) &= \beta(x,t) - \int_0^x \bar F(x,y)\alpha(y,t)dy \notag\\
&-  \int_0^x {\bar N(x,y)\beta(y,t)dy - } \bar\lambda (x)W(t),\label{eq:It2}
\end{align}
where kernels $\bar K(x,y),\bar J(x,y),\bar\gamma (x),\bar F(x,y),\bar N(x,y),\bar\lambda (x)$ can be proved well-posed as Lemma \ref{lm:wp}.

Once the equation for the kernels are solved, the control input is obtained as
\begin{align}
U(t)&=-R(l(t))^{-1}\bigg[p(l(t),t)- \int_0^{l(t)} F(l(t),y)p(y,t)dy \notag\\
&-  \int_0^{l(t)} {N({l(t)},y)r(y,t)dy - } \lambda ({l(t)})W(t)\bigg]\label{eq:U1}
\end{align}
by matching boundary conditions \eqref{eq:p5} and \eqref{eq:tar5} via \eqref{eq:t2}.

Applying \eqref{eq:zw}-\eqref{eq:yu}, \eqref{eq:XY}, \eqref{eq:dep}-\eqref{eq:deW},  then \eqref{eq:U1} can be expanded and rewritten as the original states as
\begin{align}
U_1(t)=&-\frac{1}{d_{19}(l(t))\sqrt{d_6(l(t)}}\bigg[{u_t}(l(t),t) \notag\\
 &- \int_0^{l(t)}\big[F_{11}(l(t),y)({u_t}(y,t) - \sqrt {{d_6}(y)} {u_x}(y,t))\notag\\
 &+F_{12}(l(t),y)({w_t}(y,t) - \sqrt {{d_1}(y)} {w_x}(y,t))\big]dy\notag\\
&-  \int_0^{l(t)}\big[ N_{11}(l(t),y)({u_t}(y,t) + \sqrt {{d_6}(y)} {u_x}(y,t))\notag\\
&+ N_{12}(l(t),y)({w_t}(y,t) + \sqrt {{d_1}(y)} {w_x}(y,t))\big]dy \notag\\
& -  \lambda_{11} (l(t))w(0,t)-\lambda_{12} (l(t))w_t(0,t)\notag\\
&-\lambda_{13} (l(t))u(0,t)-\lambda_{14} (l(t))u_t(0,t)\bigg],\label{eq:control1}\\
U_2(t)=&-\frac{1}{d_{20}(l(t))\sqrt{d_1(l(t)}}\bigg[{w_t}(l(t),t)\notag\\
 &- \int_0^{l(t)}\big[F_{21}(l(t),y)({u_t}(y,t) - \sqrt {{d_6}(y)} {u_x}(y,t))\notag\\
 &+F_{22}(l(t),y)({w_t}(y,t) - \sqrt {{d_1}(y)} {w_x}(y,t))\big]dy\notag\\
&-  \int_0^{l(t)}\big[ N_{21}(l(t),y)({u_t}(y,t) + \sqrt {{d_6}(y)} {u_x}(y,t))\notag\\
&+ N_{22}(l(t),y)({w_t}(y,t) + \sqrt {{d_1}(y)} {w_x}(y,t))\big]dy \notag\\
& -  \lambda_{21} (l(t))w(0,t)-\lambda_{22} (l(t))w_t(0,t)\notag\\
&-\lambda_{23} (l(t))u(0,t)-\lambda_{24} (l(t))u_t(0,t)\bigg].\label{eq:control2}
\end{align}
where inserting \eqref{eq:o5}-\eqref{eq:o6} to replace $u_x(l(t),t)$, $w_x(l(t),t)$ is used.
With the above control laws $U_1(t),U_2(t)$, we obtain stability results of the state-feedback
closed-loop system, summarized in the following two theorems.
\subsection{Stability analysis in the state-feedback closed-loop system}
\subsubsection{Stability result of closed-loop system}
The exponential stability result of the state-feedback closed-loop system is shown in the following theorem, which physically means the vibration energy of the cable, including kinetic energy and potential
energy in two directions, bounded by $\xi(\| {{w_t}(\cdot,t)} \|^2 + \| {{ w_x}(\cdot,t)} \|^2+\| {{u_t}(\cdot,t)} \|^2 + \| {{ u_x}(\cdot,t)} \|^2)$ with $\xi>0$, is exponentially convergent to zero, where the decay rate of the vibration energy is adjustable by the control parameters.
\begin{thm}\label{main}
If initial values $(w(x, 0),w_t(x, 0)) \in H^{2}(0,L)\times H^{1}(0,L)$, $(u(x, 0),u_t(x, 0)) \in H^{2}(0,L)\times H^{1}(0,L)$, the closed-loop system consisting of the plant \eqref{eq:o1}-\eqref{eq:o6} and the state-feedback control law \eqref{eq:control1}-\eqref{eq:control2} is exponentially stable in the sense that there exist positive constants $\Upsilon_{1}, \sigma_1$ such that
\begin{align}
\bigg(&\| {{w_t}(\cdot,t)} \|^2 + \| {{ w_x}(\cdot,t)} \|^2+\| {{u_t}(\cdot,t)} \|^2 + \| {{ u_x}(\cdot,t)} \|^2\notag\\
&+|w(0,t)|^2+|w_t(0,t)|^2+|u(0,t)|^2+|u_t(0,t)|^2\bigg)^{1/2}\notag\\
&\le \Upsilon_{1}\bigg(\| {{w_t}(\cdot,0)} \|^2 + \| {{ w_x}(\cdot,0)} \|^2+\| {{u_t}(\cdot,0)} \|^2 + \| {{ u_x}(\cdot,0)} \|^2\notag\\
&+|w(0,0)|^2+|w_t(0,0)|^2+|u(0,0)|^2+|u_t(0,0)|^2\bigg)^{1/2} e^{-\sigma_1 t},\label{eq:Theo2}
\end{align}
where $\|u(\cdot,t)\|^2$ is a compact notation for $\int_0^{l(t)} {{ u}(x,t)^2} dx$ and $|\cdot|$ denotes the Euclidean norm. The convergence rate $\sigma_1$ is adjustable by the control parameters.
\end{thm}
\begin{proof}
We start from studying the stability of the target system \eqref{eq:tar1}-\eqref{eq:tar5}. The equivalent stability property between the target system \eqref{eq:tar1}-\eqref{eq:tar5} and the original system \eqref{eq:o1}-\eqref{eq:o6} is ensured via the definitions \eqref{eq:zw}-\eqref{eq:yu}, \eqref{eq:dep}-\eqref{eq:deW} and the backsteppting transformations \eqref{eq:t1}-\eqref{eq:t2} and \eqref{eq:It1}-\eqref{eq:It2}.

Consider the following Lyapunov funtion for the target system \eqref{eq:tar1}-\eqref{eq:tar5},
\begin{align}
V_1 =& {W^T}(t){P_1}W(t) + \frac{1}{2}\int_0^{l(t)} {{e^{{\delta _2}x}}\beta {{(x,t)}^T}{R_a}Q{{(x)}^{ - 1}}\beta (x,t)} dx \notag\\
&+ \frac{1}{2}\int_0^{l(t)} {{e^{ - {\delta _1}x}}\alpha {{(x,t)}^T}{R_b}Q{{(x)}^{ - 1}}\alpha (x,t)} dx,\label{eq:V1}
\end{align}
where there exists a positive definite matrix $P_1 = {P_1^T} $ being the solution to the Lyapunov equation
$P_1\hat A + \hat A^T P_1 =  - \hat Q_1$,
for some $\hat Q_1 = {{\hat Q_1}^T} > 0$. $R_a,R_b$ are diagonal matrices as $R_a={\rm diag}\{r_{a1},r_{a2}\}$, $R_b={\rm diag}\{r_{b1},r_{b2}\}$.
The positive parameters $r_{a1},r_{a2},r_{b1}.r_{b2},\delta_1,\delta_2$  are to be chosen later.

According to \eqref{eq:V1}, we have
\begin{align}
\mu_1\Omega(t)\le V_1(t)\le \mu_2\Omega(t)\label{eq:norm}
\end{align}
with defining
\begin{align}
\Omega(t)=|W(t)|^2+\|{\beta}{{(x,t)}}\|^2+\|{\alpha}{{(x,t)}}\|^2
\end{align}
for some positive $\mu_1,\mu_2$.
Time derivative of $V_1(t)$ along \eqref{eq:tar1}-\eqref{eq:tar5} is obtained as
\begin{align}
\dot V_1 & ={{\dot W}^T}(t){P_1}W(t) + {W^T}(t){P_1}\dot W(t)\notag\\
 &+ \int_0^{l(t)} {{e^{{\delta _2}x}}\beta {{(x,t)}^T}{R_a}Q{{(x)}^{ - 1}}{\beta _t}(x,t)} dx\notag\\
 &+ \int_0^{l(t)} {{e^{ - {\delta _1}x}}\alpha {{(x,t)}^T}{R_b}Q{{(x)}^{ - 1}}{\alpha _t}(x,t)} dx\notag\\
 &+ \frac{{\dot l(t)}}{2}{e^{ - {\delta _1}l(t)}}\alpha {(l(t),t)^T}{R_b}Q{(l(t))^{ - 1}}\alpha (l(t),t)\notag\\
 &= W{(t)^T}({{\hat A}^T}{P_1} + {P_1}\hat A)W(t) + 4{W^T}(t){P_1}\bar B\beta (0,t)\notag\\
 &+ \frac{{\dot l(t)}}{2}{e^{ - {\delta _1}l(t)}}\alpha {(l(t),t)^T}{R_b}Q{(l(t))^{ - 1}}\alpha (l(t),t)\notag\\
 &+ \int_0^{l(t)} {{e^{{\delta _2}x}}\beta {{(x,t)}^T}{R_a}Q{{(x)}^{ - 1}}{{\bar T}_a}(x)\beta (x,t)} dx\notag\\
 &+ \int_0^{l(t)} {{e^{{\delta _2}x}}\beta {{(x,t)}^T}{R_a}Q{{(x)}^{ - 1}}g(x)\beta (0,t)} dx\notag\\
 &- \frac{1}{2}\beta {(0,t)^T}{R_a}\beta (0,t) - \frac{{{\delta _2}}}{2}\int_0^{l(t)} {{e^{{\delta _2}x}}\beta {{(x,t)}^T}{R_a}\beta (x,t)} dx\notag\\
 &- \frac{1}{2}{e^{ - {\delta _1}l(t)}}\alpha {(l(t),t)^T}{R_b}\alpha (l(t),t) + \frac{1}{2}\alpha {(0,t)^T}{R_b}\alpha (0,t)\notag\\
 &- \frac{{{\delta _1}}}{2}\int_0^{l(t)} {{e^{ - {\delta _1}x}}\alpha {{(x,t)}^T}{R_b}\alpha (x,t)} dx\notag\\
& + \int_0^{l(t)} {{e^{ - {\delta _1}x}}\alpha {{(x,t)}^T}{R_b}Q{{(x)}^{ - 1}}{{\bar T}_b}(x)\alpha (x,t)} dx\notag\\
 &+ \int_0^{l(t)} {{e^{ - {\delta _1}x}}\alpha {{(x,t)}^T}{R_b}Q{{(x)}^{ - 1}}{g_1}(x)\beta (0,t)} dx.\label{eq:dV11}
\end{align}
Applying Young's inequality and considering the boundedness of the elements $\frac{1}{\sqrt {d_1(x)}},\frac{1}{\sqrt {d_6(x)}},g_a(x),g_b(x)$ in the matrices $Q(x)^{-1},g(x),g_1(x)$, there exists $\xi>0$ such that the following inequalities holds
\begin{align}
&\quad\int_0^{l(t)} {{e^{{\delta _2}x}}\beta {{(x,t)}^T}{R_a}Q{{(x)}^{ - 1}}g(x)\beta (0,t)} dx\notag\\
 &\le \xi\int_0^{l(t)} {{e^{{\delta _2}x}}\beta {{(x,t)}^T}{R_a}\beta (x,t)} dx\notag\\
  &\quad+ \xi \int_0^{l(t)} {{e^{{\delta _2}L}}\beta {{(0,t)}^T}{\Lambda _a}\beta (0,t)} dx\label{eq:in1}\\
&\quad\int_0^{l(t)} {{e^{ - {\delta _1}x}}\alpha {{(x,t)}^T}{R_b}Q{{(x)}^{ - 1}}{g_1}(x)\beta (0,t)} dx\notag\\
& \le \xi \int_0^{l(t)} {{e^{ - {\delta _1}x}}\alpha {{(x,t)}^T}{R_b}\alpha (x,t)} dx \notag\\
&\quad+ \xi \int_0^{l(t)} {\beta {{(0,t)}^T}{\Lambda _b}\beta (0,t)} dx,\label{eq:in2}
\end{align}
where
\begin{align}
{\Lambda _a} = \left[ {\begin{array}{*{20}{c}}
{{r_{a2}}}&0\\
0&0
\end{array}} \right],{\Lambda _b} = \left[ {\begin{array}{*{20}{c}}
{{r_{b2}}}&0\\
0&0
\end{array}} \right].
\end{align}
Inserting \eqref{eq:in1}-\eqref{eq:in2} and applying Young's inequality into \eqref{eq:dV11}, one obtains
\begin{align}
&\dot V_1 (t) \le - \frac{1}{2}{\lambda _{\min }}({Q_2}){\left| {W(t)} \right|^2}- \beta {(0,t)^T}\bigg(\frac{{{R_a}}}{2} - \frac{{{R_b}}}{2} \notag\\
& - \frac{{8{{\left| {{P_1}{{\bar B}_1}} \right|}^2}}}{{{\lambda _{\min }}({Q_2})}}{I_2} - {e^{{\delta _2}L}}\xi {\Lambda _a} - \xi {\Lambda _b}\bigg)\beta (0,t)- \int_0^{l(t)} {e^{{\delta _2}x}}\beta {{(x,t)}^T}\notag\\
 &\times{R_a}\left((\frac{{{\delta _2}}}{2} - \xi ){I_2} - Q{{(x)}^{ - 1}}{{\bar T}_a}(x)\right)\beta (x,t) dx- \int_0^{l(t)} {e^{ - {\delta _1}x}}\notag\\
 &\times\alpha {{(x,t)}^T}{R_b}\left((\frac{{{\delta _1}}}{2} - \xi ){I_2} - Q{{(x)}^{ - 1}}{{\bar T}_b}(x)\right)\alpha (x,t) dx\notag\\
 &- \frac{1}{2}{e^{ - {\delta _1}l(t)}}\alpha {(l(t),t)^T}{R_b}\left({I_2} - Q{(l(t))^{ - 1}}\dot l(t)\right)\alpha (l(t),t),
\label{eq:dV12}
\end{align}
where $I_2$ is a $2\times2$ identity matrix.

The parameters $r_{a1},r_{a2},r_{b1}.r_{b2},\delta_1,\delta_2$ are chosen to satisfy
\begin{align}
&{r_{a1}}> {r_{b1}} + \frac{{16{{\left| {{P_1}{{\bar B}_1}} \right|}^2}}}{{{\lambda _{\min }}({Q_2})}} + 2{r_{a2}}{e^{{\delta _2}L}}\xi  + 2{r_{b2}}\xi,  \\
&{r_{a2}} > {r_{b2}}+ \frac{{16{{\left| {{P_1}{{\bar B}_1}} \right|}^2}}}{{{\lambda _{\min }}({Q_2})}}
\end{align}
with sufficiently large $\delta_1,\delta_2$. Note that positive constants $r_{b1},r_{b2}$ can be arbitrary.
We know the elements in the diagonal matrix $Q{(l(t))^{ - 1}}\dot l(t)$ is less than $1$ by recalling Assumption \ref{as:a1}, and the boundedness of all elements in the diagonal matrix $Q(x)^{-1}$, $\bar T_a(x),\bar T_b(x)$ by recalling Assumption \ref{as:ab}, we thus arrive at
\begin{align}
&\dot V_1(t)\le- \eta_1V_1(t),\label{eq:dV1final}
\end{align}
for some positive $\eta_1$. It follows that
\begin{align}
V_1(t)\le V_1(0)e^{-\eta_1 t}\label{eq:V0}
\end{align}
and then
\begin{align}
\Omega(t)\le \frac{\mu_{2}}{\mu_{1}}\Omega(0)e^{-\eta_1 t}
\end{align}
by recalling \eqref{eq:norm}.

Now we have obtained exponential stability in $\Omega(t)$.  Establishing the relationship between
the $\Omega(t)$ and the appropriate norm of the
$u(x,t),w(x,t)$-system, is the key to establishing exponential stability in the
original
variables.

Defining
\begin{flalign}
&\Xi(t)=\|u_x(\cdot,t)\|^2+\|u_t(\cdot,t)\|^2+|u(0,t)|^2+|u_t(0,t)|^2\notag\\
&+\|w_x(\cdot,t)\|^2+\|w_t(\cdot,t)\|^2+|w(0,t)|^2+|w_t(0,t)|^2 \label{eq:xinorm}
\end{flalign}
and recalling \eqref{eq:zw}-\eqref{eq:yu}, \eqref{eq:XY}-\eqref{eq:deW},  \eqref{eq:It1}-\eqref{eq:It2}, applying Cauchy-Schwarz inequality,
the following inequality holds
\begin{align}
{\bar\theta_{1a}}\Xi(t)\le\Omega(t)\le\bar\theta_{1b}\Xi(t)\label{eq:the21}
\end{align}
for some positive ${\bar\theta_{1a}}$ and $\bar\theta_{1b}$. Therefore, we have
\begin{align}
\Xi(t)\le\frac{\mu_{2}\bar\theta_{1b}}{\mu_{1}\bar\theta_{1a}}\Xi(0)e^{-\eta_1 t}.\label{eq:Xi}
\end{align}
Thus \eqref{eq:Theo2} is achieved with
\begin{align}
\Upsilon_1&=\sqrt{\frac{\mu_{2}\bar\theta_{1b}}{\mu_{1}\bar\theta_{1a}}},~\sigma_1=\frac{\eta_1}{2},
\end{align}
where the convergence rate $\sigma_1$ can be adjusted by the control parameter $\kappa$ through affecting ${\lambda _{\min }}({\hat Q_1})$. Then the proof of Theorem \ref{main} is completed.\end{proof}
\subsubsection{Exponential convergence of control input}
Before proving the exponential convergence of the control input, we propose a lemma first which shows the exponential stability result of the closed-loop system in the sense of $H^2$ norm.
\begin{lem}\label{lem:uxvx}
For any initial data $(w(x, 0),w_t(x, 0)) \in H^{2}(0,L)\times H^{1}(0,L)$, $(u(x, 0),u_t(x, 0)) \in H^{2}(0,L)\times H^{1}(0,L)$, the exponential stability estimate of the closed-loop system $(u(x,t),w(x,t))$ is obtained in the sense that there exist positive constants $\Upsilon_{1a}$ and $\sigma_{1a}$ such that
\begin{align}
&\quad\left(\|u_{xx}(\cdot,t)\|^2+\|w_{xx}(\cdot,t)\|^2+\|u_{tx}(\cdot,t)\|^2+\|w_{tx}(\cdot,t)\|^2\right)^{\frac{1}{2}}\notag\\
&\le\Upsilon_{1a}\bigg({\|u_{x}(\cdot,0)\|^2+\|w_{x}(\cdot,0)\|^2+\|u_{t}(\cdot,0)\|^2+\|w_{t}(\cdot,0)\|^2}\notag\\
&{+\|u_{xx}(\cdot,0)\|^2+\|w_{xx}(\cdot,0)\|^2+\|u_{tx}(\cdot,0)\|^2+\|w_{tx}(\cdot,0)\|^2}\notag\\
&{+|w(0,0)|^2+|w_t(0,0)|^2+|u(0,0)|^2+|u_t(0,0)|^2}\bigg)e^{-\sigma_{1a}t},
\end{align}
\end{lem}
\begin{proof}
Taking the spatial derivative and time derivative of \eqref{eq:tar2}-\eqref{eq:tar3} and \eqref{eq:dtar1}, \eqref{eq:dtar4}-\eqref{eq:dtar5} respectively,
\begin{align}
&\ddot W(t) = \hat A\dot W(t) + 2\bar BQ(0){\beta _x}(0,t) +2\bar B {\bar T_a}(0)\beta (0,t),\label{eq:dtar1}\\
&{\alpha _{xt}}(x,t)= - Q(x){\alpha _{xx}}(x,t) + ({\bar T_b}(x)- Q'(x))\alpha_x (x,t)\notag\\
&\quad\quad\quad\quad\quad+ {\bar T_b}'(x)\alpha (x,t),\\
&{\beta _{xt}}(x,t) = Q(x){\beta _{xx}}(x,t)+({\bar T_a}(x)+Q'(x)){\beta _{x}}(x,t) \notag\\
&\quad\quad\quad\quad\quad+ {\bar T_a}'(x)\beta(x,t),\\
&- Q(0){\alpha _x}(0,t) =  -  Q(0){\beta _x}(0,t)- {\bar T_a}(0)\beta (0,t)\notag \\
&\quad\quad\quad\quad\quad\quad\quad\quad- {\bar T_b}(0)\alpha (0,t),\label{eq:dtar4} \\
&{\beta}_x(l(t),t)=0\label{eq:dtar5}
\end{align}
where \eqref{eq:tar3}, \eqref{eq:tar5} are used. Note that \eqref{eq:dtar5} results from $(\dot l(t)+Q(l(t))){\beta _x}(l(t),t)=0$ where the elements in the diagonal matrix $\dot l(t)+Q(l(t))$ are identically nonzero by recalling Assumption \ref{as:a1}. Define new variables $\varpi(x,t)=\alpha_x(x,t)$, $\zeta(x,t)=\beta_x(x,t)$, $Z(t)=\dot W(t)$. Consider a Lyapunov function as
\begin{align}
V_2(t)& = R_1V_1(t)+{{Z}}(t)^T P_1Z(t) \notag\\
&+ \frac{1}{2}\int_0^{l(t)} {e^{\bar\delta_1 x}}{\zeta}{{(x,t)}}^T\bar R_aQ(x)^{-1}{\zeta}{{(x,t)}} dx\notag\\
&+ \frac{1}{2}\int_0^{l(t)} {e^{ - \bar\delta_2 x}}\varpi(x,t)^T\bar R_bQ(x)^{-1}\varpi(x,t) dx,\label{eq:V2}
\end{align}
where $\bar R_a,\bar R_b$ are diagonal matrices as $\bar R_a={\rm diag}\{\bar r_{a1},\bar r_{a2}\}$, $\bar R_b={\rm diag}\{\bar r_{b1},\bar r_{b2}\}$.
$\bar r_{a1},\bar r_{a2},\bar r_{b1},\bar r_{b2},\bar \delta_1,\bar \delta_1$ and $R_1$  are positive parameters.

Taking the derivative of \eqref{eq:V2} along \eqref{eq:dtar1}-\eqref{eq:dtar5}, recalling \eqref{eq:dV12}, determining $\bar r_{a1},\bar r_{a2},\bar r_{b1},\bar r_{b2},\bar \delta_1,\bar \delta_2$ through a similar process as \eqref{eq:dV11}-\eqref{eq:V0}, and choosing large enough positive constant $R_1$, we arrive at
\begin{align}
\dot V_2(t)\le -\eta_2V_2(t)
\end{align}
for some positive $\eta_2$. Recalling backstepping transformations  \eqref{eq:It1}-\eqref{eq:It2}, we have
\begin{align}
&\|p_x(\cdot,t)\|^2+\|r_x(\cdot,t)\|^2\le \Upsilon_{1b}(W(0)^2+\|p(\cdot,0)\|^2+\|r(\cdot,0)\|^2\notag\\
&+\|p_x(\cdot,0)\|^2+\|r_x(\cdot,0)\|^2)e^{-\eta_2 t}
\end{align}
for some positive $\Upsilon_{1b}$. Applying \eqref{eq:zw}-\eqref{eq:yu}, \eqref{eq:dep}-\eqref{eq:kz}, the proof of Lemma \ref{lem:uxvx} is completed.
\end{proof}
\begin{thm}\label{th:ecU}
In the closed-loop system, the state-feedback controller $U_1(t)$, $U_2(t)$ \eqref{eq:control1}-\eqref{eq:control2} are bounded and exponentially convergent to zero in the sense that there exist positive constants $\sigma_2$ and $\Upsilon_2$ such that
\begin{align}
&|U_1(t)|^2+|U_2(t)|^2\le {\Upsilon_2}\bigg(\|u_{x}(\cdot,0)\|^2+\|w_{x}(\cdot,0)\|^2\notag\\
&+\|u_{t}(\cdot,0)\|^2+\|w_{t}(\cdot,0)\|^2+\|u_{xx}(\cdot,0)\|^2+\|w_{xx}(\cdot,0)\|^2\notag\\
&+\|u_{tx}(\cdot,0)\|^2+\|w_{tx}(\cdot,0)\|^2+|w(0,0)|^2+|w_t(0,0)|^2\notag\\
&{+|u(0,0)|^2+|u_t(0,0)|^2}\bigg)e^{-{\sigma_2}t}.\label{eq:Ub}
\end{align}
\end{thm}
\begin{proof}
Considering \eqref{eq:U1} and the exponential stability result in Theorem \ref{main}, we know once $p(l(t),t)=[{u_t}(l(t),t) - \sqrt {{d_6}(l(t))} {u_x}(l(t),t)$, ${w_t}(l(t),t) - \sqrt {{d_1}(l(t))} {w_x}(l(t),t)]^T$ is made sure to be exponentially convergent zero in the sense of $|p(l(t),t)|^2$, the exponential convergence of the control input is obtained.

Applying Cauchy-Schwarz inequality and recalling \eqref{eq:o4}-\eqref{eq:o5}, one obtain
\begin{align}
|p(l(t),t)|&\le 2|p(0,t)|+2\sqrt{L}\|p_x(\cdot,t)\|\notag\\
&\le 4|r(0,t)|+4|C_3W(t)|+2\sqrt{L}\|p_x(\cdot,t)\|\notag\\
&\le 8|r(l(t),t)|+8\sqrt{L}\|r_x(\cdot,t)\|+4|C_3W(t)|\notag\\
&\quad+2\sqrt{L}\|p_x(\cdot,t)\|.\label{eq:ul}
\end{align}
Recalling \eqref{eq:It2}, \eqref{eq:tar5} and the exponential convergence of $\|\alpha(\cdot,t)\|^2, \|\beta(\cdot,t)\|^2, |W(t)|^2$ proved in Theorem \ref{main},  we have $|r(l(t),t)|$ is exponentially convergent to zero.
Recalling Lemma \ref{lem:uxvx}, we thus have that $|p(l(t),t)|$ is exponentially convergent to zero. The proof of Theorem \ref{th:ecU} is completed.
\end{proof}
\section{Observer design}\label{sec:observer}
\subsection{Observer structure}
Consider the sensors only are placed at the actuated boundary, an observer should be designed to estimate the in-domain and uncontrolled boundary states required in the state-feedback control laws \eqref{eq:control1}-\eqref{eq:control2}. The available measurements are $u_t(l(t),t)$, $w_t(l(t),t)$, i.e., $p(l(t),t)$ being known through a convertor as
\begin{align}
p(l(t),t)=&[u_t(l(t),t)-\sqrt{d_6(l(t))}d_{19}(l(t))U_1(t),\notag\\
&w_t(l(t),t)-\sqrt{d_1(l(t))}d_{20}(l(t))U_2(t)]\label{eq:convert}
\end{align}
recalling \eqref{eq:o5}-\eqref{eq:o6}, \eqref{eq:dep}, \eqref{eq:zw} and \eqref{eq:yu}.

Using the known signal $p(l(t),t)$, the observer for the coupled wave PDE plant \eqref{eq:o1}-\eqref{eq:o6} is constructed as:
\begin{align}
&\hat w_t(x,t)=\frac{1}{2}(\hat z(x,t)+\hat v(x,t)),\label{eq:wt}\\
&\hat w_x(x,t)=\frac{1}{2\sqrt {d_1(x)}}(\hat z(x,t)-\hat v(x,t)),\label{eq:wx}\\
&\hat u_t(x,t)=\frac{1}{2}(\hat k(x,t)+\hat y(x,t)),\label{eq:ut}\\
&\hat u_x(x,t)=\frac{1}{2\sqrt {d_6(x)}}(\hat k(x,t)-\hat y(x,t)),\label{eq:ux}\\
&{\hat p_t}(x,t) + Q(x){\hat p_x}(x,t) = {T_a}(x)\hat r(x,t) + {T_b}(x)\hat p(x,t)\notag\\
&+\Gamma_1(x,t)(p(l(t),t)-\hat p(l(t),t)),\label{eq:Op1}\\
&{\hat r_t}(x,t) - Q(x){\hat r_x}(x,t) = {T_a}(x)\hat r(x,t) + {T_b}(x)\hat p(x,t)\notag\\
&+\Gamma_2(x,t)(p(l(t),t)-\hat p(l(t),t)),\label{eq:Op2}\\
&\hat p(0,t) = {C_3}\hat W(t) - \hat r(0,t),\\
&\dot {\hat W}(t) = (\bar A - \bar B{C_3})\hat W(t) + 2\bar B\hat r(0,t)\notag\\
&+\Gamma_3(t)(p(l(t),t)-\hat p(l(t),t)),\label{eq:Op4}\\
&\hat r(l(t),t) = R(l(t))U(t) + p(l(t),t)\label{eq:Op5}
\end{align}
where $\hat p=[\hat y(x,t), \hat v(x,t)]^T$, $\hat r=[\hat k(x,t), \hat z(x,t)]^T$, $\hat W(t)=[\hat X(t),\hat Y(t)]^T=[\hat w(0,t),\hat w_t(0,t),\hat u(0,t),\hat u_t(0,t)]^T$.

Note that the observer consists of two parts:

1. \eqref{eq:Op1}-\eqref{eq:Op5} in the sense of a copy of plant \eqref{eq:p1}-\eqref{eq:p5} plus output injections is built to estimate $p(x,t),r(x,t)$;

2. Once $p(x,t),r(x,t)$ are estimated successfully by \eqref{eq:Op1}-\eqref{eq:Op5}, the estimations of the original plant are straightly obtained as \eqref{eq:wt}-\eqref{eq:ux} considering Riemann transformations \eqref{eq:zw}-\eqref{eq:yu}.

Next, the observer gains $\Gamma_1(x,t)$, $\Gamma_2(x,t)$ and $\Gamma_3(t)$ will be determined to achieve the exponential stability of the  observer error system which can be seen in the next subsection. A difference from the traditional ones should be noted that $\Gamma_1$, $\Gamma_2$  not only depend on the spatial variable $x$ but also depend on time $t$ because of the time-varying domain.
\subsection{Observer error system}
The observation problem is essentially to ensure
the observer errors (differences between the
estimated and real states) are reduced to zero, by defining adequate observer
gains. Denote the observer errors as
\begin{align}
\tilde w_t(x,t)&=w_t(x,t)-\hat w_t(x,t),\label{eq:twt}\\
\tilde w_x(x,t)&=w_x(x,t)-\hat w_x(x,t),\\
\tilde u_t(x,t)&=u_t(x,t)-\hat u_t(x,t),\\
\tilde u_x(x,t)&=u_x(x,t)-\hat u_x(x,t),\label{eq:tux}\\
\tilde W(x,t)&=W(x,t)-\hat W(t)\notag\\
&=[X(t),Y(t)]-[\hat X(t),\hat Y(t)]\notag\\
&=[w(0,t),w_t(0,t),u(0,t),u_t(0,t)]^T\notag\\
&\quad-[\hat w(0,t),\hat w_t(0,t),\hat u(0,t),\hat u_t(0,t)]^T\notag\\
&=[\tilde X(t),\tilde Y(t)]\notag\\
&=[\tilde w(0,t),\tilde w_t(0,t),\tilde u(0,t),\tilde u_t(0,t)]^T,\label{eq:tXY}\\
\tilde p(x,t)&=p(x,t)-\hat p(x,t)=[\tilde y(x,t), \tilde v(x,t)],\label{eq:tp}\\
\tilde r(x,t)&=r(x,t)-\hat r(x,t)=[\tilde k(x,t), \tilde z(x,t)].\label{eq:tr}
\end{align}
Recalling \eqref{eq:p1}-\eqref{eq:p5}, \eqref{eq:zw}-\eqref{eq:yu} and \eqref{eq:wt}-\eqref{eq:Op5}, the resulting observer error dynamics is given by
\begin{align}
&\tilde w_t(x,t)=\frac{1}{2}(\tilde z(x,t)+\tilde v(x,t)),\label{eq:ewt}\\
&\tilde w_x(x,t)=\frac{1}{2\sqrt {d_1(x)}}(\tilde z(x,t)-\tilde v(x,t)),\label{eq:ewx}\\
&\tilde u_t(x,t)=\frac{1}{2}(\tilde k(x,t)+\tilde y(x,t)),\label{eq:eut}\\
&\tilde u_x(x,t)=\frac{1}{2\sqrt {d_6(x)}}(\tilde k(x,t)-\tilde y(x,t)),\label{eq:eux}\\
&{\tilde p_t}(x,t) + Q(x){\tilde p_x}(x,t) = {T_a}(x)\tilde r(x,t) + {T_b}(x)\tilde p(x,t)\notag\\
&+\Gamma_1(x,t)\tilde p(l(t),t),\label{eq:Op1e}\\
&{\tilde r_t}(x,t) - Q(x){\tilde r_x}(x,t) = {T_a}(x)\tilde r(x,t) + {T_b}(x)\tilde p(x,t)\notag\\
&+\Gamma_2(x,t)\tilde p(l(t),t),\label{eq:Op2e}\\
&\tilde p(0,t) = {C_3}\tilde W(t) - \tilde r(0,t),\label{eq:Op3e}\\
&\dot {\tilde W}(t) = (\bar A - \bar B{C_3})\tilde W(t) + 2\bar B\tilde r(0,t)\notag\\
&+\Gamma_3(t)\tilde p(l(t),t),\label{eq:Op4e}\\
&\tilde r(l(t),t) = 0,\label{eq:Op5e}
\end{align}
where the subsystem \eqref{eq:Op1e}-\eqref{eq:Op5e} describing dynamics of the observer error of the system \eqref{eq:p1}-\eqref{eq:p5}, determines the observer error of the plant \eqref{eq:o1}-\eqref{eq:o6} via \eqref{eq:ewt}-\eqref{eq:eux}. Therefore, the exponential stability of \eqref{eq:Op1e}-\eqref{eq:Op5e} is the core to make sure the proposed observer can be exponentially convergent to the actual states of the original plant \eqref{eq:o1}-\eqref{eq:o6}.
\subsection{Observer backtepping design}
To find the observer gains $\Gamma_1(x,t),\Gamma_2(x,t),\Gamma_3(t)$ that guarantee that
\eqref{eq:Op1e}-\eqref{eq:Op5e} is exponentially stable, we use a transformation to map \eqref{eq:Op1e}-\eqref{eq:Op5e} to a target observer error system whose exponential stability result is straightforward to obtain.

The transformation is introduced as
\begin{align}
\tilde p(x,t)=&\tilde \alpha(x,t)-\int_x^{l(t)}\bar \varphi(x,y)\tilde\alpha(y,t)dy,\label{eq:Tp}\\
\tilde r(x,t)=&\tilde \beta(x,t)-\int_x^{l(t)}\bar \psi(x,y)\tilde\alpha(y,t)dy,\\
\tilde W(t) =& \tilde S(t)+ \int_0^{l(t)} {\bar K(y)\tilde \alpha (y,t)} dy,\label{eq:TW}
\end{align}
where kernels $\bar \varphi(x,y)=\{\bar \varphi_{ij}(x,y)\}_{1\le i, j\le 2}$, $\bar \psi(x,y)=\{\bar \psi_{ij}(x,y)\}_{1\le i, j\le 2}$ on a triangular domain $\mathcal D_1=\{0\le x\le y\le l(t)\}$ and $\bar K(y)=\{\bar K_{ij}(y)\}_{1\le i\le 4,1\le j\le 2}$  are to be determined.

The target observer error system is set up as
\begin{align}
&{{\tilde \alpha }_t}(x,t) + Q(x){{\tilde \alpha }_x}(x,t) = {{T}_a}(x)\tilde \beta (x,t) + {{\bar T}_b}(x)\tilde \alpha (x,t)\notag\\
 &+ \int_x^{l(t)} {\bar M(x,y)} \tilde \beta (y,t)dy,\label{eq:etar1}\\
&{{\tilde \beta }_t}(x,t) - Q(x){{\tilde \beta }_x}(x,t) = \int_x^{l(t)} {\bar N(x,y)} \tilde \beta (y,t)dy\notag\\
&+T_a(x)\tilde \beta (x,t),\\
&\tilde \alpha (0,t) = {C_3}\tilde S(t)  - \tilde \beta (0,t)+ \int_0^{l(t)} {H(y) \tilde \alpha (y,t)} dy,\label{eq:etar3}\\
&\tilde \beta (l(t),t) = 0,\label{eq:etar5}\\
&\dot {\tilde S}(t) = \check A\tilde S(t)+\check E\tilde\beta(0,t)+\int_0^{l(t)}G(y)\tilde \beta (y,t) dy\label{eq:etar4}
\end{align}
where $\check A=\bar A - \bar B{C_3} - {L_0}{C_3}$ is a Hurwitz matrix by choosing $L_{0}=\{L_{0ij}\}_{1\le i\le 4,1\le j\le 2}$ recalling Assumption \ref{as:a2}, and ${\bar M(x,y)}$, $\bar N(x,y)$ satisfy
\begin{align}
\bar M(x,y){\rm{ = }}\int_x^y {\bar \varphi(x,z)\bar M(z,y)} dz + \bar \varphi(x,y){T_a}(y),\\
\bar N(x,y){\rm{ = }}\int_x^y {\bar \psi (x,z)\bar M(z,y)} dz + \bar \psi (x,y){T_a}(y).
\end{align}
Note that $H(y)=\{h_{ij}(y)\}_{1\le i,j\le 2}$ in \eqref{eq:etar3} is a strict lower triangular matrix as
\begin{align}
H(y)=\left(
       \begin{array}{cc}
         0 & 0 \\
        \bar \psi_{2,1} (0,y) + \bar \varphi_{2,1} (0,y)+\bar K_{21}(y) & 0 \\
       \end{array}
     \right),\label{eq:Hy}
\end{align}
and $G(y)=\{G_{ij}(y)\}_{1\le i\le 4,1\le j\le 2}, \check E=\{\check E_{ij}\}_{1\le i\le 4,1\le j\le 2}$ in \eqref{eq:etar4} are equal to $- \bar K(0,y){{\bar T}_a}(y)-\int_0^{y}{\bar K(0,z)\bar M(z,y)dz}$ and $L_0+2\bar B$ respectively. The exponential stability of the target system \eqref{eq:etar1}-\eqref{eq:etar4} will be seen in Lemma \ref{thm:ob}.

By matching \eqref{eq:Op1e}-\eqref{eq:Op5e} and \eqref{eq:etar1}-\eqref{eq:etar4} through the transformation \eqref{eq:Tp}-\eqref{eq:TW}, the conditions on the kernels in \eqref{eq:Tp}-\eqref{eq:TW} and observer gains in \eqref{eq:Op1}, \eqref{eq:Op2}, \eqref{eq:Op4} are obtained as follows.
Kernels $\bar \varphi(x,y)$, $\bar \psi(x,y)$, $\bar K(y)$ should satisfy the matrix equations:
\begin{align}
&-{{\bar \varphi }_y}(x,y)Q(y) - Q(x){{\bar \varphi }_x}(x,y) - \bar \varphi(x,y)Q'(y) \notag\\
&+ {T_a}(x)\bar \psi (x,y) + {T_b}(x)\bar \varphi(x,y) - \bar \varphi(x,y){{\bar T}_b}(y)=0,\label{eq:ek1}\\
& - {{\bar \psi }_y}(x,y)Q(y) + Q(x){{\bar \psi }_x}(x,y) - \bar \psi (x,y)Q'(y)\notag\\
 & + {T_a}(x)\bar \psi (x,y) - \bar \psi (x,y){{\bar T}_b}(y)+ {T_b}(x)\bar \varphi(x,y)=0,\\
 &Q(x)\bar \varphi (x,x) - \bar \varphi (x,x)Q(x) ={T_b}(x)- {{\bar T}_b}(x) ,\\
&Q(x)\bar \psi (x,x) + \bar \psi (x,x)Q(x) {\rm{ = }}-{T_b}(x),\label{eq:ek4}\\
&\bar \psi (0,y) + \bar \varphi (0,y)+C_3\bar K(y)=H(y),\label{eq:ek5}\\
& - \bar K'(y)Q(y) + (\bar A - \bar B{C_3} - {L_0}{C_3})\bar K(y)\notag\\
&- \bar K(y)[Q'(y) + {{\bar T}_b}(y)]\notag\\
 & - {L_0}\bar \varphi (0,y) - (2\bar B + {L_0})\bar \psi (0,y) = 0,\label{eq:ek7}\\
&\bar K(0) = {L_0}Q{(0)^{ - 1}}\label{eq:ek8}
\end{align}
and the observe gains are obtained as
\begin{align}
{\Gamma _1}(x,t)&=\dot l(t)\bar \varphi (x,l(t)) - \bar \varphi (x,l(t))Q(l(t)),\label{eq:Ga1}\\
{\Gamma _2}(x,t)&=\dot l(t)\bar \psi (x,l(t)) - \bar \psi (x,l(t))Q(l(t)),\label{eq:Ga2}\\
{\Gamma _3}(t) &= \dot l(t)\bar K(l(t)) - \bar K(l(t))Q(l(t)).\label{eq:Ga3}
\end{align}
\begin{lem}\label{lm:wpe}
After adding an additional artificial boundary condition for the element $\bar \varphi_{21}$ in the matrix $\bar \varphi$, the matrix equations \eqref{eq:ek1}-\eqref{eq:ek8} have a unique solution $\bar \varphi ,\bar \psi\in L^{\infty}(\mathcal D_1)$, $\bar K\in L^{\infty}([0,l(t)])$.
\end{lem}
\begin{proof}
 After swapping positions of arguments as B.9-B.10 in \cite{Anfinsen2017Disturbance}, i.e., changing the domain $\mathcal D_1$ to $\mathcal D$, \eqref{eq:ek1}-\eqref{eq:ek8} has the analogous form with kernels $F(x,y),N(x,y),\lambda(y)$ \eqref{eq:Kerc1}-\eqref{eq:Kerc6}. Following the steps  in the proof of Lemma \ref{lm:wp}, including introducing the extended domain $\mathcal D_0$ and adding the additional artificial boundary condition, Lemma \ref{lm:wpe} can be obtained.
\end{proof}
Following similar steps as above, the inverse transformation of \eqref{eq:Tp}-\eqref{eq:TW} can be determined as
\begin{align}
\tilde \alpha(x,t)=&\tilde p(x,t)-\int_x^{l(t)}\check \varphi(x,y)\tilde p(y,t)dy,\label{eq:ITp}\\
\tilde \beta(x,t)=&\tilde r(x,t)-\int_x^{l(t)}\check \psi(x,y)\tilde p(y,t)dy,\\
\tilde S(t) =& \tilde W(t)+ \int_0^{l(t)} {\check K(y)\tilde r (y,t)} dy,\label{eq:ITW}
\end{align}
where $\check \varphi(x,y)\in R^{2\times 2}$, $\check \psi(x,y)\in R^{2\times 2}$ and $\check K(y)\in R^{4\times 2}$ are kernels on $\mathcal D_1$ and $0\le y\le l(t)$, respectively.
\subsection{Stability analysis of observer error system}
Before showing the performance of the proposed observer on tracking the actual states in the original plant \eqref{eq:o1}-\eqref{eq:o6} in the next theorem, the stability result of the observer error subsystem \eqref{eq:Op1e}-\eqref{eq:Op5e} which dominates the observer errors  of the original plant \eqref{eq:o1}-\eqref{eq:o6} is given in the following lemma.
\begin{lem}\label{thm:ob}
Consider the observer error subsystem \eqref{eq:Op1e}-\eqref{eq:Op5e},
there exist
positive constants $\Upsilon_3,\sigma_3$ such that
\begin{align}
&\quad\left(\|\tilde p(\cdot,t)\|^2+\|\tilde r(\cdot,t)\|^2+\left|\tilde W(t)\right|^2\right)^{\frac{1}{2}}\notag\\
&\le \Upsilon_3\bigg(\|\tilde p(\cdot,0)\|^2+\|\tilde r(\cdot,0)\|^2+\left|\tilde W(0)\right|^2\bigg)^{\frac{1}{2}}e^{-\sigma_3 t}.\label{eq:norm3}
\end{align}
\end{lem}
\begin{proof}
Expanding \eqref{eq:etar1}-\eqref{eq:etar4} as $\tilde\alpha=[\tilde\alpha_1,\tilde\alpha_2]^T$, $\tilde\beta=[\tilde\beta_1,\tilde\beta_2]^T$, one obtains
\begin{align}
&\tilde \alpha _{it}(x,t) + Q_{i}(x){\tilde \alpha }_{ix}(x,t) =\sum_{j=1}^{2} {T}_{aij}(x)\tilde \beta_j (x,t)\notag\\
 & + {\bar T}_{bi}(x)\tilde \alpha_i (x,t)+ \int_x^{l(t)} \sum_{j=1}^{2}{\bar M_{ij}(x,y)} \tilde \beta_j (y,t)dy,\label{eq:cetar1}\\
&{{\tilde \beta }_{it}}(x,t) - Q_{i}(x){{\tilde \beta }_{ix}}(x,t) = \int_x^{l(t)} \sum_{j=1}^{2}{\bar N_{ij}(x,y)} \tilde \beta_j (y,t)dy\notag\\
&+\sum_{j=1}^{2} {T}_{aij}(x)\tilde \beta_j (x,t),\\
&\tilde \alpha_i (0,t) = {C_3}\tilde S(t)  - \tilde \beta_i (0,t)+ (i-1)\int_0^{l(t)}{h_{21}}(y) \tilde \alpha_{1}(y,t) dy,\label{eq:cetar3}\\
&\tilde \beta_{i} (l(t),t) = 0\label{eq:cetar5}
\end{align}
for $i=1,2$, and $\tilde S(t)$ is governed by
\begin{align}
\dot {\tilde S}(t) =& \check A\tilde S(t)+\check E[\tilde \beta_1 (0,t),\tilde \beta_2 (0,t)]^T\notag\\
&+\int_0^{l(t)}G(y)[\tilde \beta_1 (y,t),\tilde \beta_2 (y,t)]^T dy.\label{eq:cetar6}
\end{align}
 In \eqref{eq:cetar1}-\eqref{eq:cetar6},  $\tilde \beta_i(\cdot,t)$ are independent and $\tilde \beta_i(\cdot,t)\equiv0$ after a finite time because of \eqref{eq:cetar5}. Thus $\tilde S(t)$ is exponentially convergent to zero because $\check A$ is Hurwitz.  $\tilde \alpha_{1}(\cdot,t)$ are exponentially convergent to zero because of the exponential convergence of $\tilde \alpha_1(0,t)$ considering \eqref{eq:cetar3} for $i=1$. $\tilde \alpha_{1}(\cdot,t)$ flow into $\tilde \alpha_{2}(0,t)$ through the boundary \eqref{eq:cetar3}, where  exponential convergence of  $\tilde \alpha_{2}(0,t)$ also can be obtained for $i=2$ because all signals at the right hand side of the equal sign are exponentially convergent to zero. It follows that $\tilde \alpha_2(\cdot,t)$ are exponentially convergent to zero as well.

The exponential stability result would be seen more clearly by using the following  Lyapunov function as
\begin{align}
&V_e(t)= \frac{\check r_{b1}}{2}\int_0^{l(t)} {e^{-\check\delta_1 x}}{\tilde\alpha_1}{{(x,t)}}^TQ_1(x)^{-1}{\tilde\alpha_1}{{(x,t)}} dx\notag\\
&+\frac{\check r_{a1}}{2}\int_0^{l(t)} {e^{\check\delta_2 x}}{\tilde\beta_1}{{(x,t)}}^TQ_1(x)^{-1}{\tilde\beta_1}{{(x,t)}} dx\notag\\
&+ \frac{\check r_{a2}}{2}\int_0^{l(t)} {e^{\check\delta_2 x}}{\tilde\beta_2}{{(x,t)}}^TQ_2(x)^{-1}{\tilde\beta_2}{{(x,t)}} dx+{{\tilde S}}(t)^T P_2\tilde S(t) \notag\\
&+ \frac{\check r_{b2}}{2}\int_0^{l(t)} {e^{ -\check\delta_1 x}}{\tilde\alpha_2}{{(x,t)}}^TQ_2(x)^{-1}{\tilde\alpha_2}{{(x,t)}} dx,\label{eq:Ve}
\end{align}
where a positive definite matrix $P_2 = {P_2^T} $ is the solution to the Lyapunov equation
$P_2\check A  + \check A^T P_2 =  -\hat  Q_2$, for some $\hat Q_2 = {{\hat Q_2}^T} > 0$, and $\check r_{a1},\check r_{a2},\check r_{b1},\check r_{b2},\check\delta_1,\check\delta_2$ are positive constants. The following inequality holds
\begin{align}
\mu_{e1}\Omega_e(t)\le V_e(t)\le\mu_{e2}\Omega_e(t)
\end{align}
for some positive $\mu_{e1},\mu_{e2}$, where $\Omega_e(t)=\|\tilde \alpha(\cdot,t)\|^2+\|\tilde \beta(\cdot,t)\|^2+|\tilde S(t)|^2$. Note that $\|\tilde \alpha(\cdot,t)\|^2=\int_0^{l(t)}\tilde \alpha_1(\cdot,t)^2dx+\int_0^{l(t)}\tilde \alpha_2(\cdot,t)^2dx$.

Taking the derivative of \eqref{eq:Ve} along \eqref{eq:cetar1}-\eqref{eq:cetar6}, choosing $\check r_{a1},\check r_{a2},\check r_{b1},\check r_{b2},\check\delta_1,\check\delta_2$ as a similar process in \eqref{eq:dV11}-\eqref{eq:V0}, we can obtain
\begin{align}
\dot V_e(t)\le -\eta_eV_e(t)
\end{align}
for some positive $\eta_e$ which is associated with the choice of $L_0$.
It follows that the exponential stability result in the sense of
\begin{align}
&\left(\|\tilde \alpha(x,t)\|^2+\|\tilde \beta(x,t)\|^2+|\tilde S(t)|^2\right)^{\frac{1}{2}}\notag\\
\le&\xi_{e}\left(\|\tilde \alpha(x,0)\|^2+\|\tilde \beta(x,0)\|^2+|\tilde S(0)|^2\right)^{\frac{1}{2}}e^{-\eta_e t},
\end{align}
for some positive $\xi_{e}$ and $\eta_e$.

Recalling the direct and inverse backstepping transformations \eqref{eq:Tp}-\eqref{eq:TW}, \eqref{eq:ITp}-\eqref{eq:ITW}, and applying Cauchy-Schwarz inequality,  the proof of Lemma \ref{thm:ob} is completed.

Applying the exponential stability result of the observer error subsystem \eqref{eq:Op1e}-\eqref{eq:Op5e} in Lemma \ref{thm:ob} and recalling the relationships \eqref{eq:ewt}-\eqref{eq:eux}, we obtain the following theorem about the performance of the observer on tracking the actual states in the original plant \eqref{eq:o1}-\eqref{eq:o6}.
\end{proof}
\begin{thm}\label{lem:tux}
 Considering the observer error system \eqref{eq:ewt}-\eqref{eq:Op5e} with the observer gains $\Gamma_1(x,t)$ \eqref{eq:Ga1}, $\Gamma_2(x,t)$ \eqref{eq:Ga2}, $\Gamma_3(t)$ \eqref{eq:Ga3}, for any initial data $(w(x, 0),w_t(x, 0)) \in H^{2}(0,L)\times H^{1}(0,L)$, $(u(x, 0),u_t(x, 0)) \in H^{2}(0,L)\times H^{1}(0,L)$, there exist positive constants $\Upsilon_4,\sigma_4$ such that
\begin{align}
&\quad\bigg(\|\tilde u_t(\cdot,t)\|^2+\|\tilde u_x(\cdot,t)\|^2+\|\tilde w_t(\cdot,t)\|^2+\|\tilde w_x(\cdot,t)\|^2\notag\\
&\quad+\tilde w(0,t)^2+\tilde w_t(0,t)^2+\tilde u(0,t)^2+\tilde u_t(0,t)^2\bigg)^{\frac{1}{2}}\notag\\
&\le \Upsilon_4\bigg(\|\tilde u_t(\cdot,0)\|^2+\|\tilde u_x(\cdot,0)\|^2+\|\tilde w_t(\cdot,0)\|^2+\|\tilde w_x(\cdot,0)\|^2\notag\\
&\quad+\tilde w(0,0)^2+\tilde w_t(0,0)^2+\tilde u(0,0)^2+\tilde u_t(0,0)^2\bigg)^{\frac{1}{2}}e^{-\sigma_4 t},
\end{align}
which means the observer states in \eqref{eq:wt}-\eqref{eq:Op5} can be exponentially convergent to the actual values in \eqref{eq:o1}-\eqref{eq:o6} according to \eqref{eq:twt}-\eqref{eq:tux}.
\end{thm}
\begin{proof}
Recalling Lemma \ref{thm:ob} and \eqref{eq:tXY}-\eqref{eq:tr}, the following inequality holds
\begin{align*}
&\quad\bigg(\|\tilde y(\cdot,t)\|^2+\|\tilde v(\cdot,t)\|^2+\|\tilde k(\cdot,t)\|^2+\|\tilde z(\cdot,t)\|^2\notag\\
&+\left|\tilde X(t)\right|^2+\left|\tilde Y(t)\right|^2\bigg)^{\frac{1}{2}}\le \Upsilon_{4a}\bigg(\|\tilde y(\cdot,0)\|^2+\|\tilde v(\cdot,0)\|^2\notag\\
&\quad+\|\tilde k(\cdot,0)\|^2+\|\tilde z(\cdot,0)\|^2+\left|\tilde X(0)\right|^2+\left|\tilde Y(0)\right|^2\bigg)^{\frac{1}{2}}e^{-\sigma_{4a} t},
\end{align*}
for some positive constants $\Upsilon_{4a},\sigma_{4a}$.

According to \eqref{eq:ewt}-\eqref{eq:eux}, of which the inverse transformation where $\tilde u_t(\cdot,t),\tilde u_x(\cdot,t)$, $\tilde w_t(\cdot,t),\tilde w_x(\cdot,t)$ are represented by $\tilde z(\cdot,t),\tilde v(\cdot,t)$, $\tilde k(\cdot,t),\tilde y(\cdot,t)$ is straightforward to obtain, the proof of Theorem \ref{lem:tux} is then completed recalling \eqref{eq:tXY}.
\end{proof}
\section{Output-feedback controller and stability analysis}\label{sec:output}
The output-feedback controller results from combining the state-feedback controller in Section \ref{sec:Con} and the observer in Section \ref{sec:observer}. After inserting observer states into the state-feedback controller \eqref{eq:control1}-\eqref{eq:control2} to replace the unmeasurable states, the controller $U_{1}(t),U_{2}(t)$ in \eqref{eq:o5}-\eqref{eq:o6} are taken as the output-feedback form as $U_{o1}(t),U_{o2}(t)$:
\begin{align}
U_{o1}(t)=&-\frac{1}{d_{19}(l(t))\sqrt{d_6(l(t))}}\bigg[{u_t}(l(t),t)\notag\\
 &- \int_0^{l(t)}\big[F_{11}(l(t),y)({\hat u_t}(y,t) - \sqrt {{d_6}(y)} {\hat u_x}(y,t))\notag\\
 &+F_{12}(l(t),y)({\hat w_t}(y,t) - \sqrt {{d_1}(y)} {\hat w_x}(y,t))\big]dy\notag\\
&-  \int_0^{l(t)}\big[ N_{11}(l(t),y)({\hat u_t}(y,t) + \sqrt {{d_6}(y)} {\hat u_x}(y,t))\notag\\
&+ N_{12}(l(t),y)({\hat w_t}(y,t) + \sqrt {{d_1}(y)} {\hat w_x}(y,t))\big]dy \notag\\
& -  \lambda_{11} (l(t))\hat w(0,t)-\lambda_{12} (l(t))\hat w_t(0,t)\notag\\
&-\lambda_{13} (l(t))\hat u(0,t)-\lambda_{14} (l(t))\hat u_t(0,t)\bigg],\label{eq:control1o}\\
U_{o2}(t)=&-\frac{1}{d_{20}(l(t))\sqrt{d_1(l(t)}}\bigg[{w_t}(l(t),t)\notag\\
 &- \int_0^{l(t)}\big[F_{21}(l(t),y)({\hat u_t}(y,t) - \sqrt {{d_6}(y)} {\hat u_x}(y,t))\notag\\
 &+F_{22}(l(t),y)({\hat w_t}(y,t) - \sqrt {{d_1}(y)} {\hat w_x}(y,t))\big]dy\notag\\
&-  \int_0^{l(t)}\big[ N_{21}(l(t),y)({\hat u_t}(y,t) + \sqrt {{d_6}(y)} {\hat u_x}(y,t))\notag\\
&+ N_{22}(l(t),y)({\hat w_t}(y,t) + \sqrt {{d_1}(y)} {\hat w_x}(y,t))\big]dy \notag\\
& -  \lambda_{21} (l(t))\hat w(0,t)-\lambda_{22} (l(t))\hat w_t(0,t)\notag\\
&-\lambda_{23} (l(t))\hat u(0,t)-\lambda_{24} (l(t))\hat u_t(0,t)\bigg].\label{eq:control2o}
\end{align}
The diagram of the output-feedback closed-loop system is shown in Fig. \ref{fig:closed}. It should be noted that $U_{o1}(t)$, $U_{o2}(t)$ are implemented based
on the boundary measurements ${u_t}(l(t),t)$, ${w_t}(l(t),t)$ mentioned in Section \ref{sec:problem}. To be exact, ${u_t}(l(t),t)$, ${w_t}(l(t),t)$ directly act  as the first terms of \eqref{eq:control1o}-\eqref{eq:control2o}, and also are used to obtain the solutions of the observer \eqref{eq:wt}-\eqref{eq:Op5} which are required in the remain terms in \eqref{eq:control1o}-\eqref{eq:control2o} through a convertor \eqref{eq:convert}.

The following theorem shows the exponential stability result of the output-feedback closed-loop system.
\begin{figure}
\centering
\includegraphics[width=8.8cm]{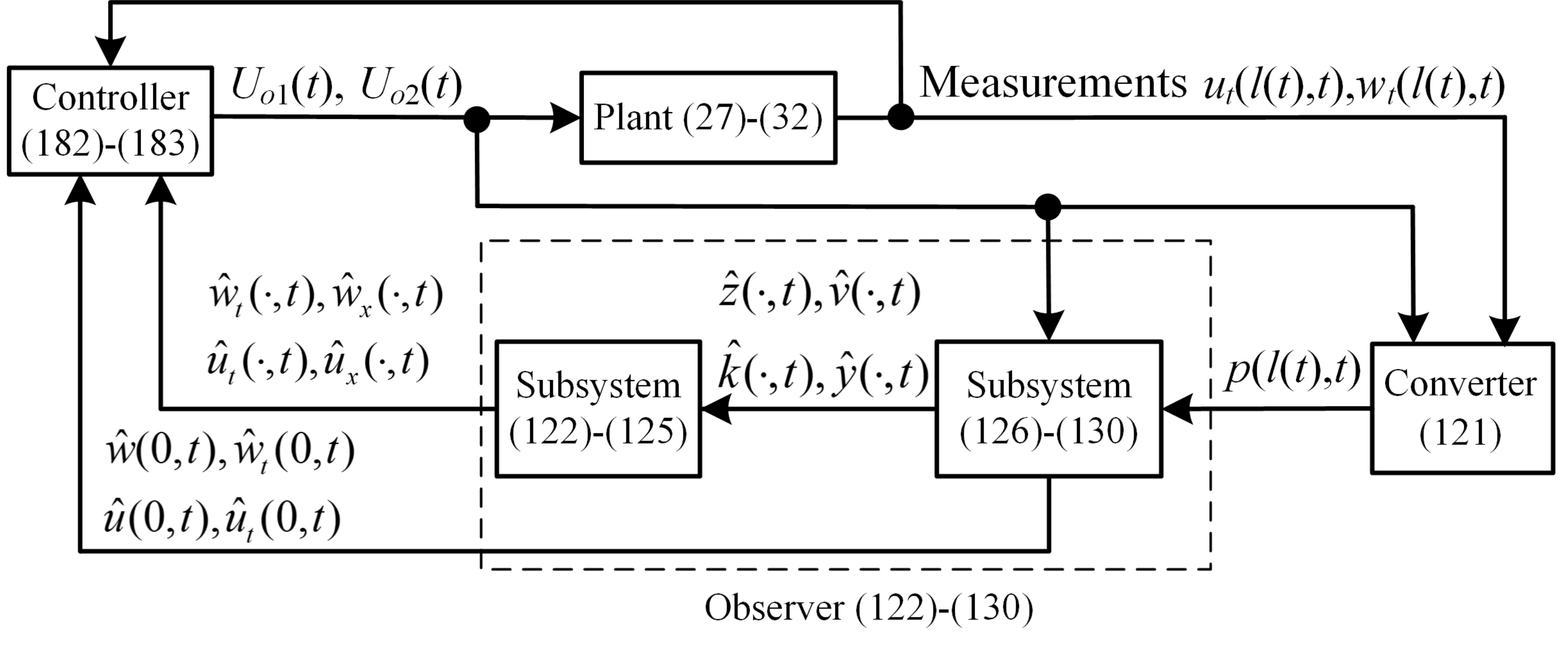}
\caption{Diagram of the output-feedback closed-loop system.}
\label{fig:closed}
\end{figure}
\begin{thm}\label{th:all}
Considering the closed-loop system consisting of the plant \eqref{eq:o1}-\eqref{eq:o6}, the observer \eqref{eq:wt}-\eqref{eq:Op5} and the output-feedback controller \eqref{eq:control1o}-\eqref{eq:control2o}, for initial values $(w(x, 0),w_t(x, 0)) \in H^{2}(0,L)\times H^{1}(0,L)$, $(u(x, 0),u_t(x, 0)) \in H^{2}(0,L)\times H^{1}(0,L)$, one obtains

1) there exist positive constants $\Upsilon_{5}$ and $\sigma_5$ such that
\begin{align}
\left(\Xi(t)+\hat\Xi(t)\right)^{1/2}\le \Upsilon_5\left(\Xi(0)+\hat\Xi(0)\right)^{1/2}e^{-\sigma_5t},\label{eq:norma}
\end{align}
where $\Xi(t)$ is given in \eqref{eq:xinorm} and $\hat\Xi(t)$ is defined as
\begin{align*}
&\hat\Xi(t)=\|\hat u_x(\cdot,t)\|^2+\|\hat u_t(\cdot,t)\|^2+\left|\hat u(0,t)\right|^2+\left|\hat u_t(0,t)\right|^2\notag\\
&+\|\hat w_x(\cdot,t)\|^2+\|\hat w_t(\cdot,t)\|^2+\left|\hat w(0,t)\right|^2+\left|\hat w_t(0,t)\right|^2.
\end{align*}
2) the output-feedback controllers \eqref{eq:control1o}-\eqref{eq:control2o} are bounded and exponentially convergent to zero.
\end{thm}
\begin{proof}
The output-feedback controller \eqref{eq:control1o}-\eqref{eq:control2o} can be written as
\begin{align}
[U_{o1}(t),U_{o2}(t)]^T=&[U_{sf1}(t),U_{sf2}(t)]^T+\tilde\delta(t)\label{eq:U12o}
\end{align}
considering \eqref{eq:twt}-\eqref{eq:tXY}, where $U_{sf1}(t),U_{sf2}(t)$ are the state-feedback form presented as \eqref{eq:control1}-\eqref{eq:control2}, and $\tilde\delta(t)\in R^{2\times 1}$ is
\begin{align}
\tilde\delta(t)=&2R(l(t))^{-1}\bigg[ \int_0^{l(t)}F(l(t),y)[{\tilde  u_t}(y,t) - \sqrt {{d_6}(y)} {\tilde  u_x}(y,t),\notag\\
&{\tilde  w_t}(y,t) - \sqrt {{d_1}(y)} {\tilde w_x}(y,t)]^Tdy\notag\\
&+  \int_0^{l(t)} N(l(t),y)[{\tilde u_t}(y,t) + \sqrt {{d_6}(y)} {\tilde u_x}(y,t),\notag\\
&{\tilde w_t}(y,t) + \sqrt {{d_1}(y)} {\tilde w_x}(y,t)]dy \notag\\
& +  \lambda (l(t))[\tilde  w(0,t),\tilde  w_t(0,t),\tilde  u(0,t),\tilde  u_t(0,t)]^T\bigg].\label{eq:delta}
\end{align}
Applying the output-feedback controller \eqref{eq:U12o} into the plant \eqref{eq:o1}-\eqref{eq:o6}, i.e., $U_1(t)=U_{1o}(t),~ U_2(t)=U_{2o}(t)$, recalling Theorem \ref{main} and Theorem \ref{lem:tux}, together with \eqref{eq:twt}-\eqref{eq:tXY}, we achieve \eqref{eq:norma}, i.e, property 1) in Theorem \ref{th:all}. Moreover, applying Theorem \ref{th:ecU} which shows the state-feedback controllers $[U_{sf1}(t),U_{sf2}(t)]^T$ are exponentially convergent to zero and Theorem \ref{lem:tux} which guarantees the exponential convergence to zero of $\tilde\delta(t)$ \eqref{eq:delta}, boundedness and exponential convergence of the output feedback control input are obtained according to \eqref{eq:U12o},  i.e., the property 2) of Theorem \ref{th:all}. Therefore, the proof of Theorem \ref{th:all} is completed.
\end{proof}
\section{Simulation test on vibration suppression of DCV}\label{sec:Sim}
The simulation is conducted based on the linear model \eqref{eq:l1}-\eqref{eq:l6} and actual nonlinear model \eqref{eq:n1}-\eqref{eq:n6} with unmodeled disturbances, respectively, where the first one is to verify the above theoretical results and the second one is to illustrate the effectiveness in the application of vibration control of DCV.
Note that the time-varying domain plant with pre-determined time-varying functions $l(t)$ and $\dot l(t)$ shown in Fig. \ref{fig:figl},  is converted to the one on the
fixed domain $\iota=[0,1]$ with time-varying coefficients related to $l(t),\dot l(t), \ddot l(t)$  via introducing
\begin{align}
\iota=\frac{x}{l(t)}, \label{eq:iota}
\end{align}
i.e., representing $u(x,t)$ by $u(\iota,t)$ as
 \begin{align}
 u_{x}(x,t)&=\frac{1}{l(t)}{u_{\iota}(\iota,t)},~u_{xx}(x,t)=\frac{1}{l(t)^2}{u_{\iota\iota}(\iota,t)},\\
 u_{t}(x,t)&=u_t(\iota,t)-\frac{\dot l(t)\iota}{l(t)}{u_{\iota}(\iota,t)},\\
 u_{tt}(x,t)&={u_{tt}(\iota,t)}-\frac{2\dot l(t)\iota}{l(t)}{u_{\iota t}(\iota,t)}-\frac{\dot l(t)^2\iota^2}{l(t)^2}{u_{\iota\iota}(\iota,t)}\notag\\
 &\quad-\frac{(l(t)\ddot l(t)-2\ddot l(t)^2)\iota}{l(t)^2}{u_{\iota}(\iota,t)},
 \end{align}
and then the simulation is conducted based
on the finite difference method with time step and space step
as 0.001 and 0.05 respectively. The observer \eqref{eq:wt}-\eqref{eq:Op5} is solved in the same way, and the following equations are used to obtain $\hat u,\hat w$ from $\hat k,\hat y,\hat z,\hat v$
\begin{align*}
\hat u(\iota,t)=\int_0^{\iota}\frac{1}{2\sqrt {d_6(\bar\iota)}}(\hat k(\bar\iota,t)-\hat y(\bar\iota,t))d{\bar\iota}+\bar C_1\hat W(t),\\
\hat w(\iota,t)=\int_0^{\iota}\frac{1}{2\sqrt {d_1(\bar\iota)}}(\hat z(\bar\iota,t)-\hat v(\bar\iota,t))d{\bar\iota}+\bar C_2\hat W(t)
\end{align*}
according to \eqref{eq:wt}-\eqref{eq:ux} and \eqref{eq:tXY}, where $\bar C_1=[0,0,1,0]$ and $\bar C_2=[1,0,0,0]$.

The initial conditions are defined according to the steady state, as $u_x(\cdot,0)=\bar\varepsilon(\cdot)$, $u_t(\cdot,0)=0$ and $w_x(\cdot,0)=-\bar\phi(\cdot)$, $w_t(\cdot,0)=0$. With defining $u(0,0)=0$ and $w(l(0),0)=0$, the initial conditions of \eqref{eq:l1}-\eqref{eq:l6} can thus be defined completely in the numerical calculation adopting the finite difference method. All initial conditions $\hat k(\cdot,0),\hat y(\cdot,0),\hat z(\cdot,0),\hat v(\cdot,0),\hat W(0)$ of the observer \eqref{eq:wt}-\eqref{eq:Op5} are set as zero.
\subsection{Linear model}\label{sec:simlin}
\subsubsection{System coefficients}
Matching \eqref{eq:o1}-\eqref{eq:o6} with \eqref{eq:l1}-\eqref{eq:l6}, we have the specific expressions of the coefficients in \eqref{eq:o1}-\eqref{eq:o6} as
\begin{align}
&d_1(x)=\frac{\frac{3}{2}EA_a\bar\phi(x)^2+T(x)}{m_c},\label{eq:q0}\\
&d_2(x)=\frac{EA_a\bar\varepsilon'(x)+\rho g}{m_c},~d_3=\frac{-EA_a\bar\phi'(x)}{m_c},\label{eq:q1}\\
&d_{4}=\frac{-c_v}{m_c},~d_{5}=0,~d_6=\frac{EA_a}{m_c},~d_7(x)=\frac{-EA_a\bar\phi'(x)}{m_c},\\
&d_{8}=d_{9}=0,~d_{10}=\frac{-c_u}{m_c},~d_{11}=\frac{-c_w}{M_L},~d_{12}=\frac{-EA_a\bar\phi(0)^2}{2M_L},\\
&d_{13}=0,~d_{14}=\frac{-EA_a\bar\phi(0)}{M_L},~d_{15}=\frac{-c_h}{M_L},d_{16}=\frac{-EA_a}{M_L},\\
&d_{17}=0,~d_{18}=\frac{EA_a\bar\phi(0)}{2M_L},~d_{19}=\frac{1}{EA_a},\\
&d_{20}(l(t))=\frac{1}{EA_a\bar\varepsilon(l(t))+ \frac{EA_a}{2}\bar\phi(l(t))^2 + T(l(t))},\label{eq:q20}
\end{align}
where $T(x)$, $\bar\varepsilon(x)$ and $\bar\phi(x)$ are given in \eqref{eq:statictension}, \eqref{eq:steady1}-\eqref{eq:steady2}, and the values of the physical parameters are shown in Table \ref{table1}. $x$ in \eqref{eq:q0}-\eqref{eq:q20} can be represented by $\iota$  via \eqref{eq:iota}.
\begin{figure}
\centering
\includegraphics[width=8cm]{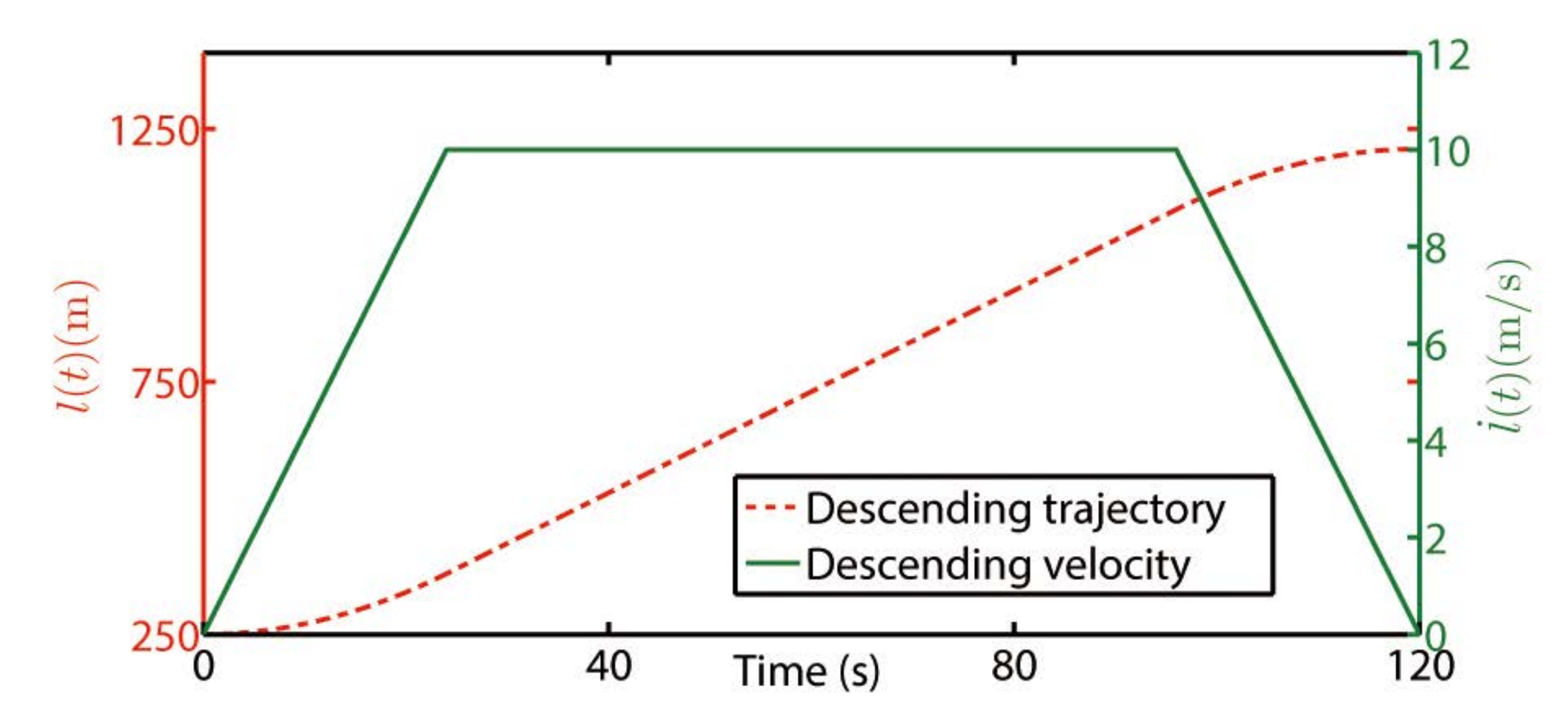}
\caption{Descending trajectory and velocity, i.e., the time-varying cable length $l(t)$ and the changing rate $\dot l(t)$.}
\label{fig:figl}
\end{figure}
\subsubsection{Controller parameters}
Apply the proposed controllers \eqref{eq:control1o}-\eqref{eq:control2o} into \eqref{eq:l1}-\eqref{eq:l6}, where the approximate solution of the kernel equations \eqref{eq:Kerc1}-\eqref{eq:Kerc6} is also solved by the finite
difference method on a fixed triangular domain ${\mathcal D_0}=\{0\le y\le x\le L\}$, and then extract $F(l(t),y),N(l(t),y)$ which would be used in the controller. The control parameters $\kappa$ are chosen as
\begin{align}
\left[ {\begin{array}{*{20}{c}}
{{\kappa _{11}}}&{{\kappa _{12}}}&{{\kappa _{13}}}&{{\kappa _{14}}}\\
{{\kappa _{21}}}&{{\kappa _{22}}}&{{\kappa _{23}}}&{{\kappa _{24}}}
\end{array}} \right]
=\left[ {\begin{array}{*{20}{c}}
{0.8}&1.2&4.5&6\\
2.5&3&1.5&2
\end{array}} \right]\times 10^3\label{eq:kappa11}
\end{align}
which determines the kernel $\lambda(x)$ used in the controllers. A same process is used to obtain $\bar\varphi(x,l(t)),\bar\psi(x,l(t))$ used in the observer gains \eqref{eq:Ga1}-\eqref{eq:Ga2} and all elements in $L_0$ are defined as 1.
\begin{figure}
\centering
\includegraphics[width=8cm]{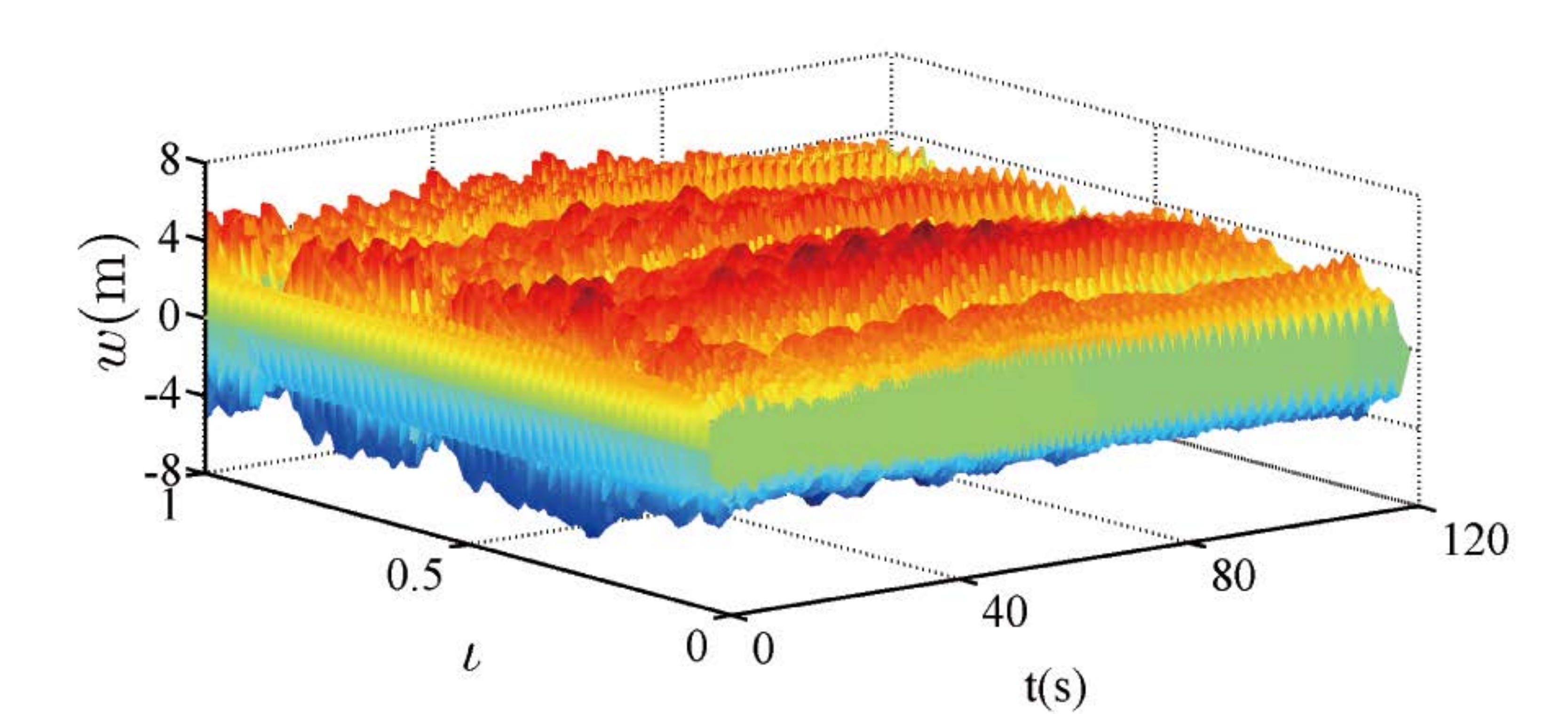}
\caption{Responses of lateral vibrations $w(x,t)$ without control.}
\label{fig:wo}
\end{figure}
\subsubsection{Simulation results}
\begin{figure}
\centering
\includegraphics[width=8cm]{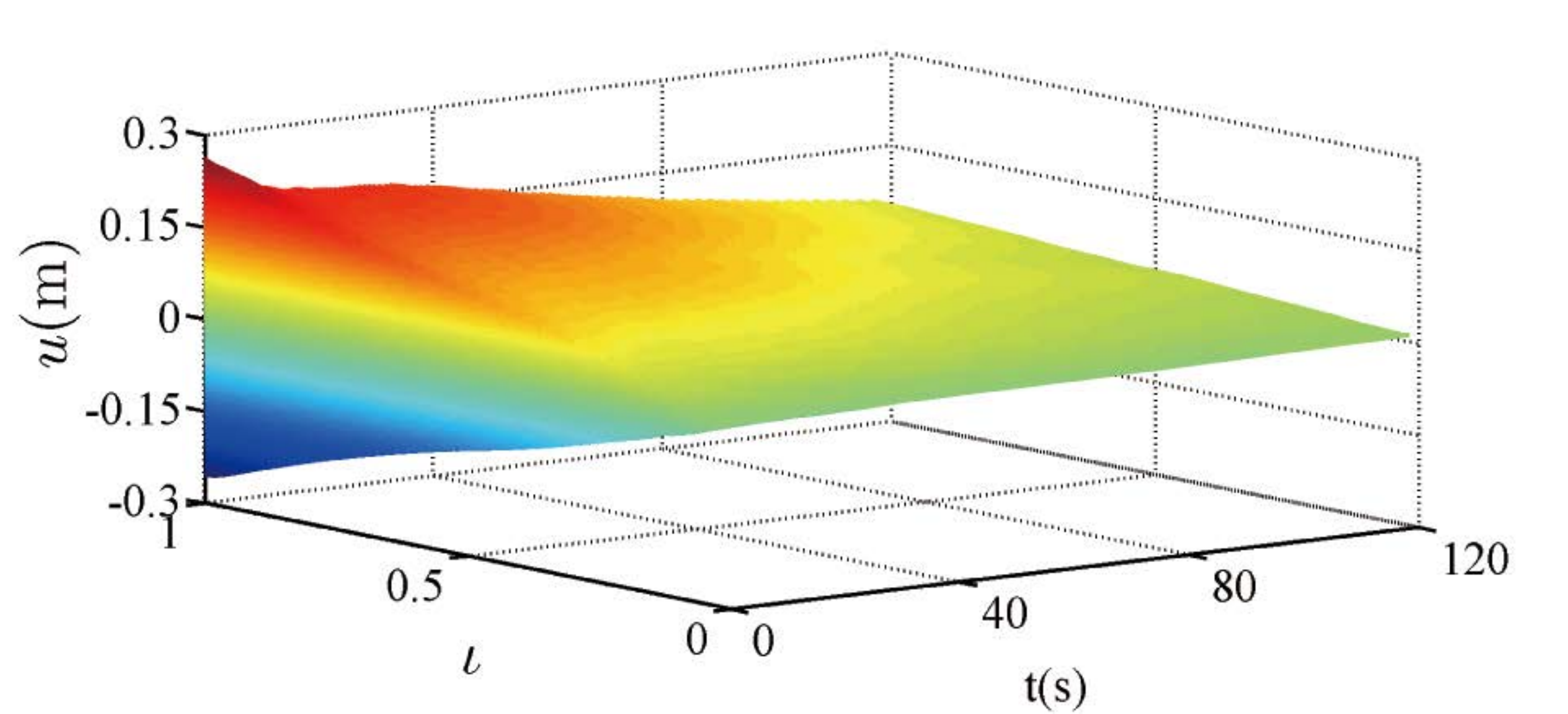}
\caption{Responses of longitudinal vibrations $u(x,t)$ without control.}
\label{fig:uo}
\end{figure}
\begin{figure}
\centering
\includegraphics[width=8cm]{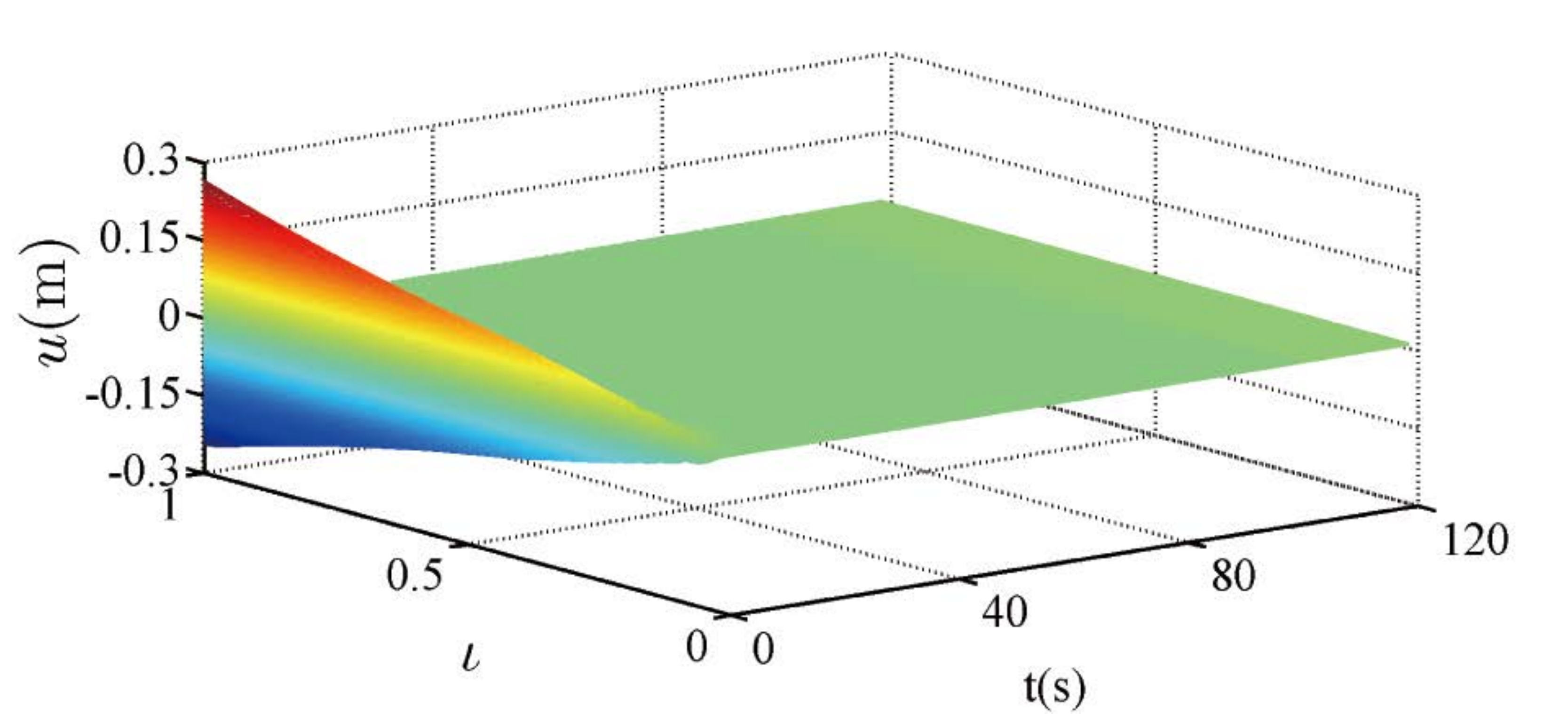}
\caption{Closed-loop responses of longitudinal vibrations $u(x,t)$.}
\label{fig:u}
\end{figure}
\begin{figure}
\centering
\includegraphics[width=8cm]{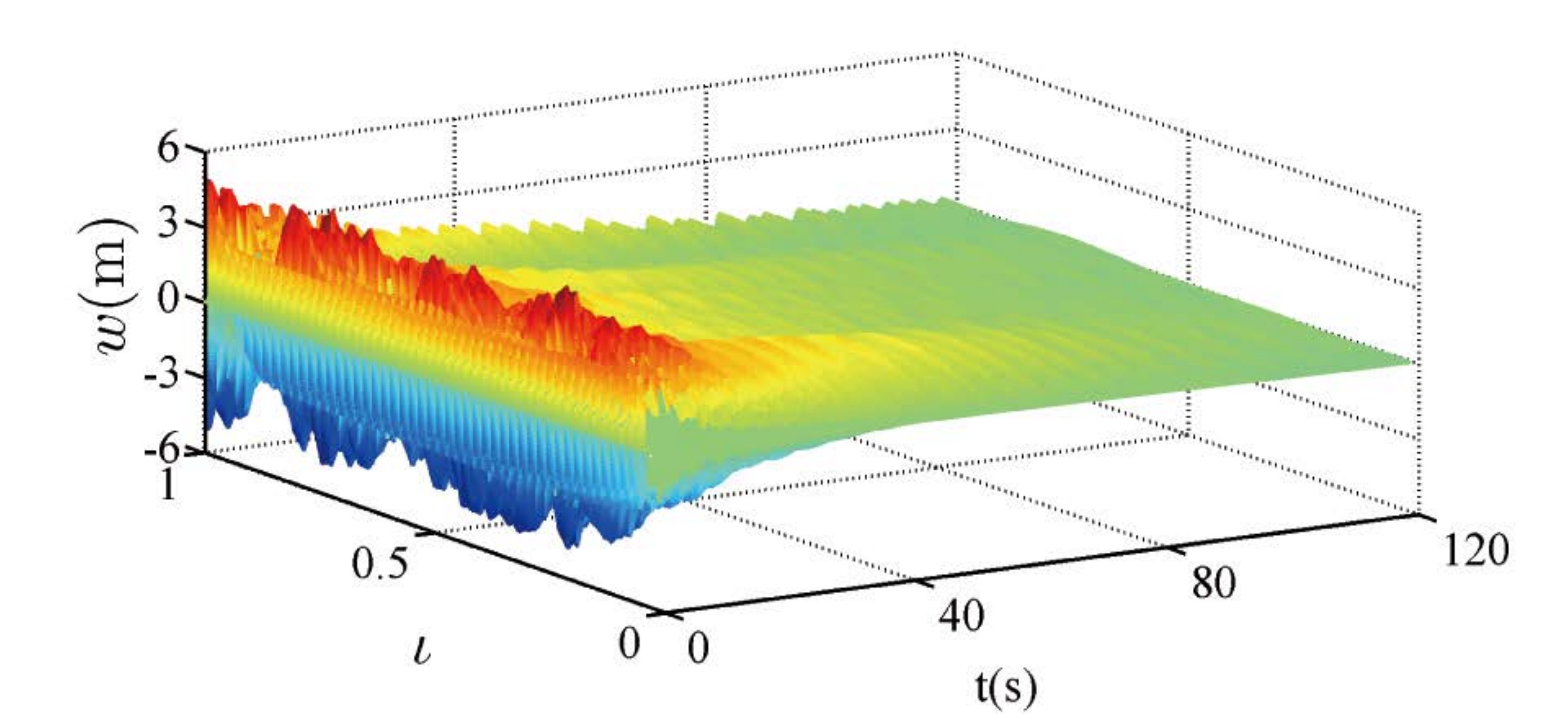}
\caption{Closed-loop responses of lateral vibrations $w(x,t)$.}
\label{fig:w}
\end{figure}
\begin{figure}
\centering
\includegraphics[width=8cm]{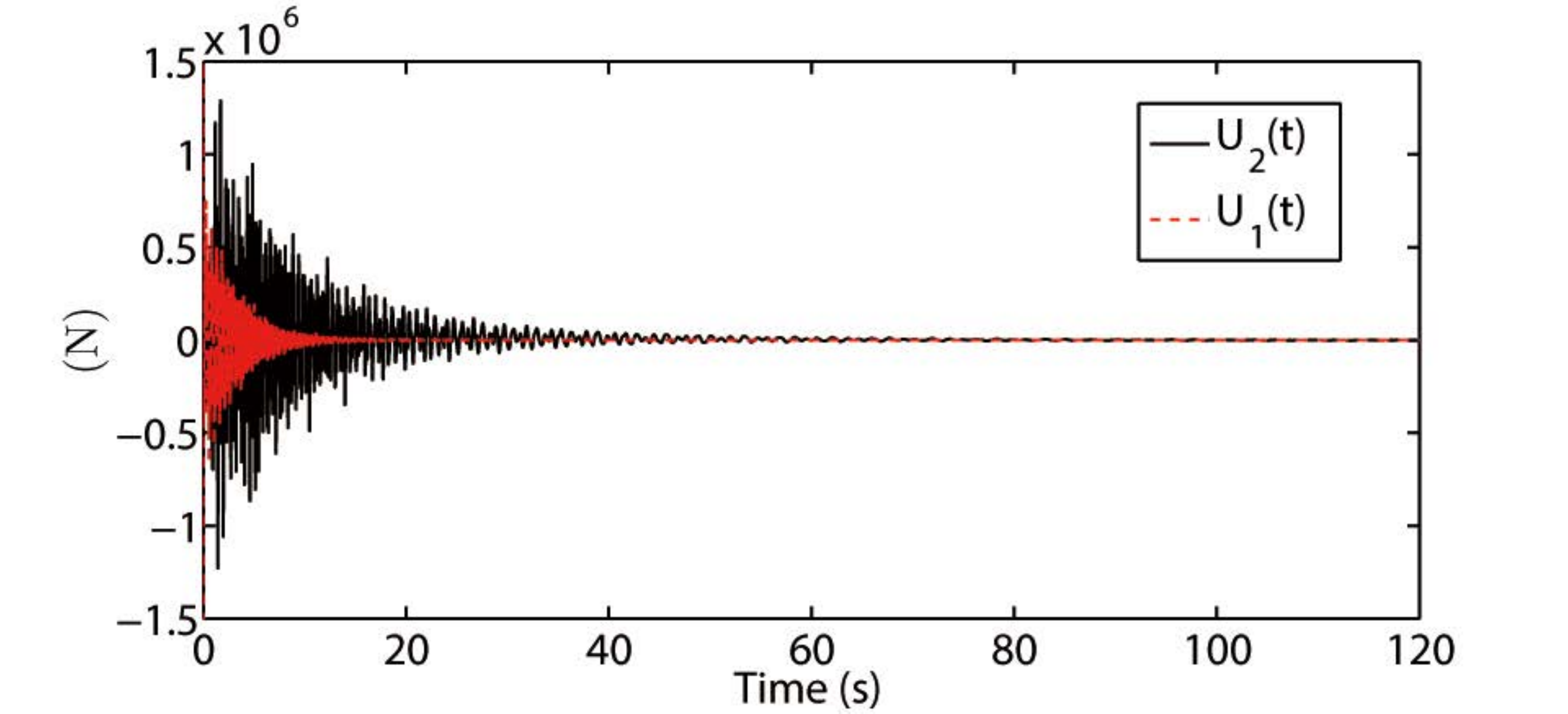}
\caption{Control forces $U_1(t)$ and $U_2(t)$.}
\label{fig:uf}
\end{figure}
\begin{figure}
\centering
\includegraphics[width=8cm]{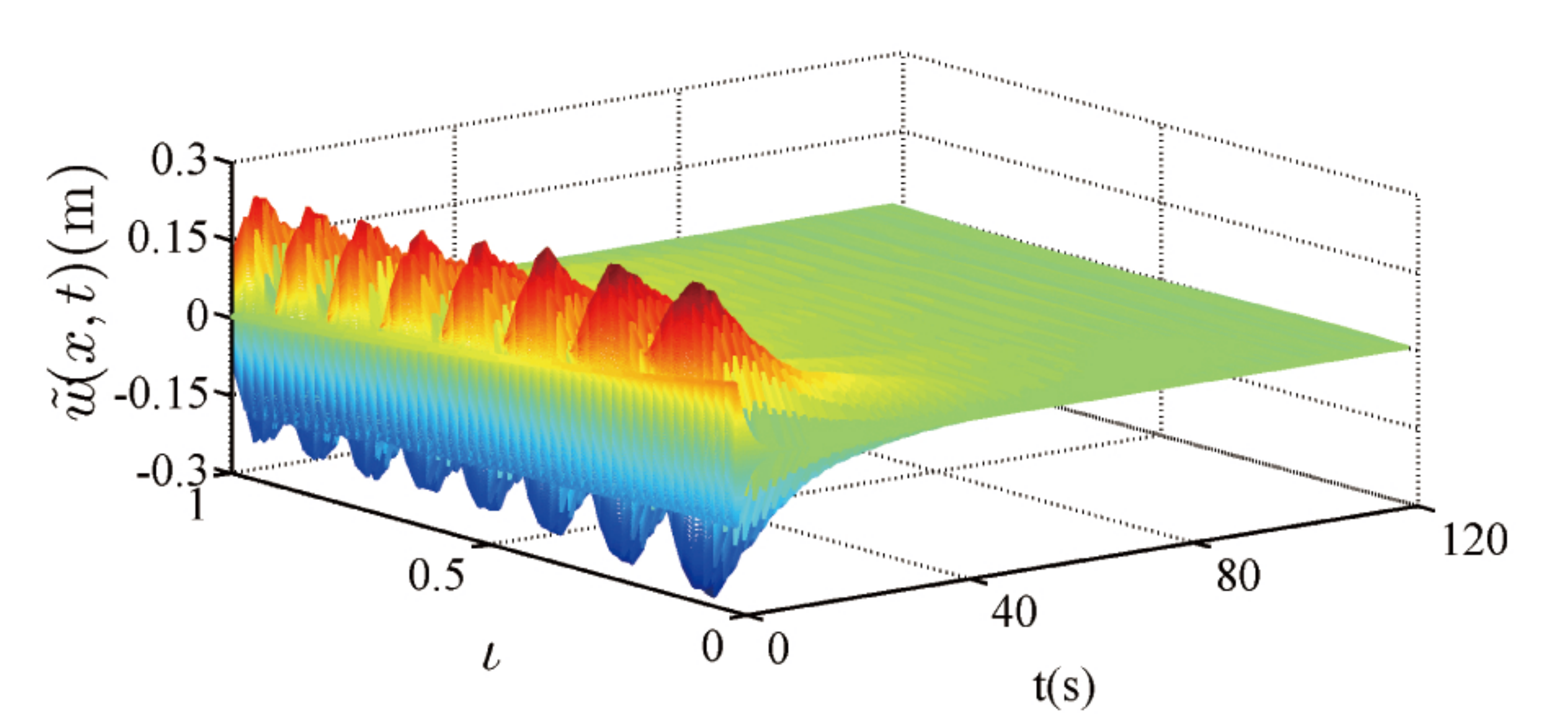}
\caption{Observer error of lateral vibrations $\tilde w(x,t)$.}
\label{fig:tw}
\end{figure}
\begin{figure}
\centering
\includegraphics[width=8cm]{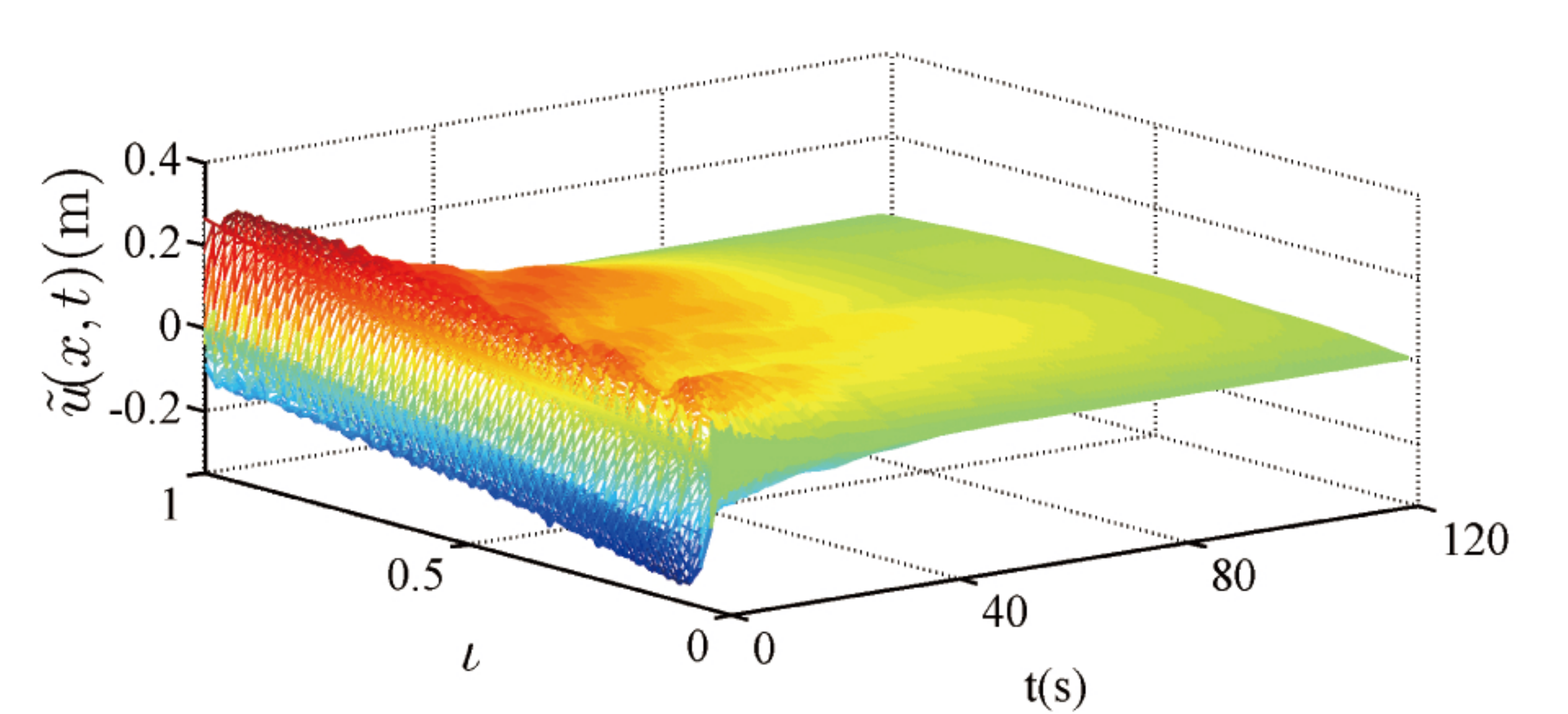}
\caption{Observer error of longitudinal vibrations $\tilde u(x,t)$.}
\label{fig:tu}
\end{figure}
It can be seen from Fig. \ref{fig:wo} that the large lateral vibrations whose oscillation range is up to $10$ m,  persist in the whole operation time $120$s in the case of without control. Even though the longitudinal vibrations in Fig. \ref{fig:uo} are decaying because of the material damping coefficient $c_u$ of cable in Tab. \ref{table1}, the longitudinal vibration at the top of the cable ($\iota=1$ in Fig. \ref{fig:wo}) is excessive and reduced slowly  because this point bears the whole mass of the cable and payload resulting in large elastic deflections. Applying the proposed output-feedback two-directional vibration control forces at the ship-mounted crane, it is shown in Fig. \ref{fig:u} that the longitudinal vibrations are suppressed very fast, and the lateral vibrations also decay with a satisfied decay rate according to Fig. \ref{fig:w}. The output-feedback control forces at the ship-mounted crane are shown in Fig. \ref{fig:uf}, where the states of the proposed observer are used. The performance of the observer on tracking the actual states can be seen in Figs. \ref{fig:tw}-\ref{fig:tu}, which show the observer errors of both lateral and longitudinal vibrations are convergent to zero.
\subsection{Actual nonlinear model with ocean current disturbances}\label{subsecnonlinear}
The nonlinear model \eqref{eq:n1}-\eqref{eq:n6} in the simulation is generated via replacing $\bar\varepsilon(\cdot),\bar\phi(\cdot)$ by $u_x(\cdot,t)$ and $-w_x(\cdot,t)$ in the linear model \eqref{eq:l1}-\eqref{eq:l6} respectively. Note that the time step and space step are changed to $0.0005$ and $0.1$ respectively in the finite difference method to ensure the existence of the numerical solution.

In practice, ocean current disturbances would act as external lateral oscillating drag forces $f(\iota,t)$ on the cable. In the simulation, $f(\iota,t)$ which is added in \eqref{eq:n2} and \eqref{eq:n4} converted to a fixed domain as above are defined as follows.
Consider the time-varying ocean surface current velocity $P(t)$ modeled by a first-order Gauss-Markov process \cite{Fossen2002}:
\begin{align}
\dot P(t)+\mu P(t)=\mathcal G(t),~~ P_{\min}\le P(t)\le P_{\max},
\end{align}
where $\mathcal G(t)$ is Gaussian white noise. Constants $P_{\min},P_{\max}$ and $\mu$ are chosen as $1.6 ms^{-1}$, $2.4 ms^{-1}$ and 0 \cite{How2011}.
$f(\iota,t)$ can then be given as  \cite{How2011}:
\begin{align}
f(\iota,t)=&\left(0.9\iota+0.1\right)\frac{1}{2}\rho_{s}C_dP(t)^2R_DA_D\notag\\
&\times\cos\left(4\pi\frac{S_tP(t)}{R_D}t+\varsigma\right)\label{eq:fxt}
\end{align}
where $0.9\iota+0.1$ means the full disturbance load is applied at the top of the cable, i.e., the ocean surface, and linearly declines to its $0.1$ at
the bottom of the cable, i.e., the payload. $C_d=1$ denotes the drag coefficient and $\varsigma=\pi$ is the phase angle. $A_D=400$ denotes the amplitude of the oscillating drag force, $S_t=0.2$ being the Strouhal number \cite{Faltinsen1993}.
The ocean disturbances $f$ used in the simulation are shown in Fig. \ref{fig:fn}.
\begin{figure}
\centering
\includegraphics[width=8cm]{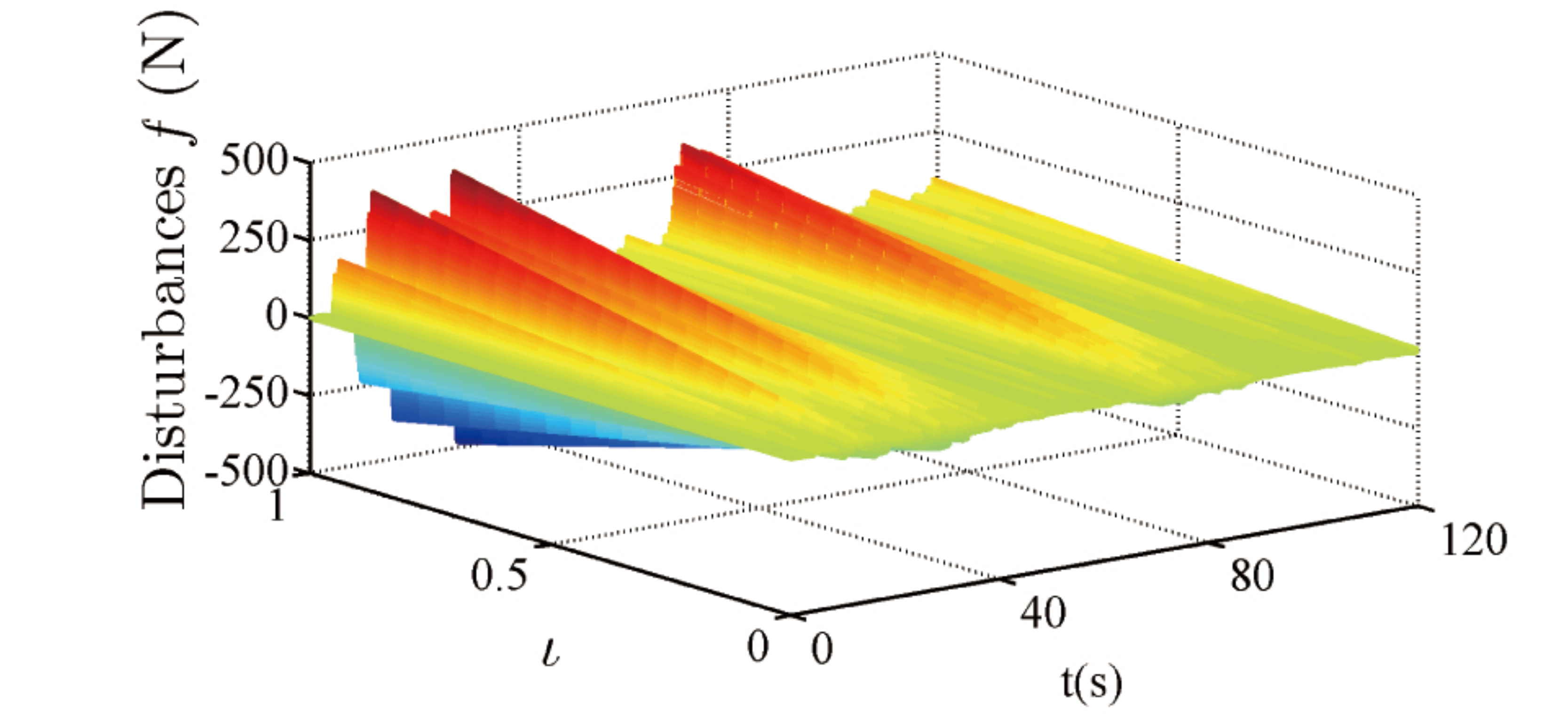}
\caption{Lateral oscillations drug forces from ocean current disturbances.}
\label{fig:fn}
\end{figure}

\begin{figure}
\centering
\includegraphics[width=8cm]{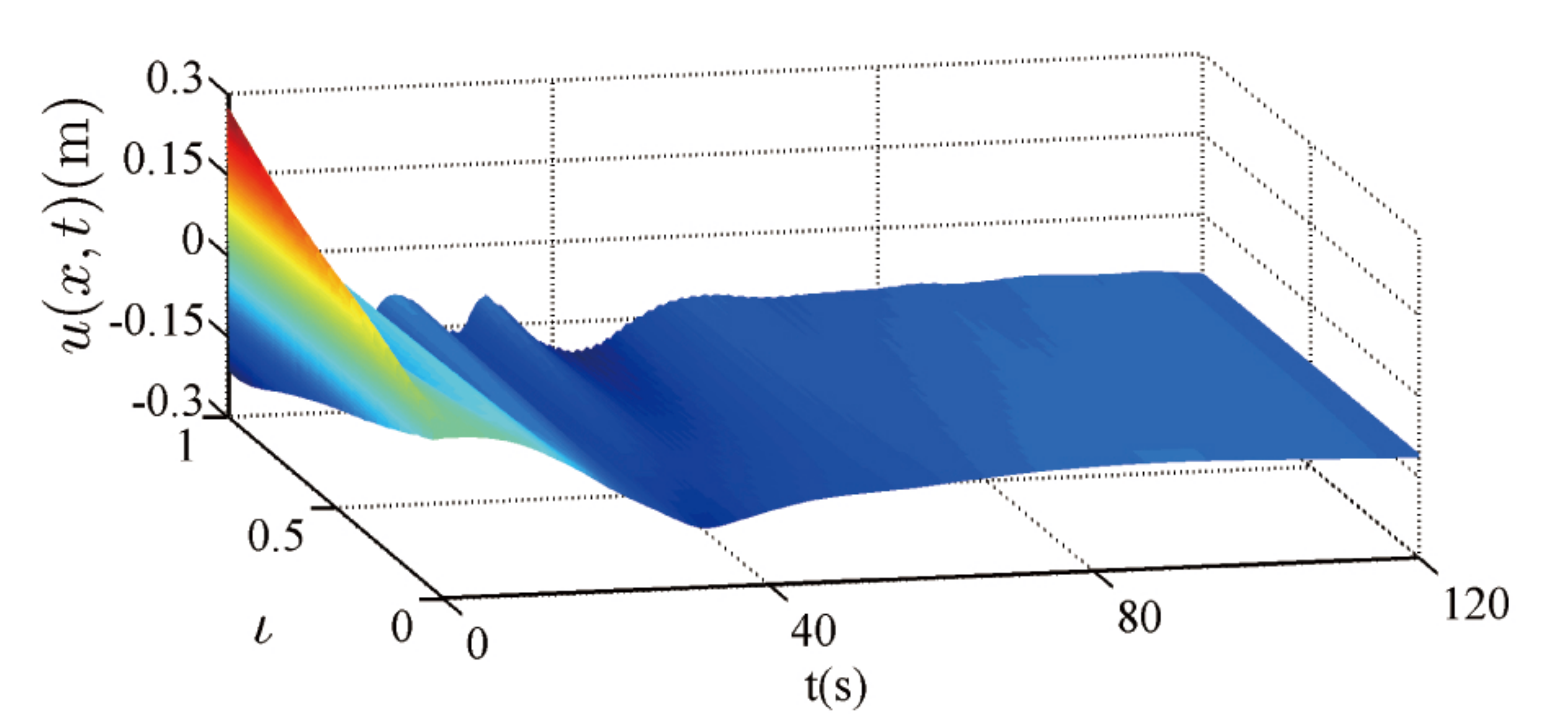}
\caption{Closed-loop responses of longitudinal vibrations $u(x,t)$ in the actual nonlinear model with unmodeled disturbances.}
\label{fig:un}
\end{figure}
\begin{figure}
\centering
\includegraphics[width=8cm]{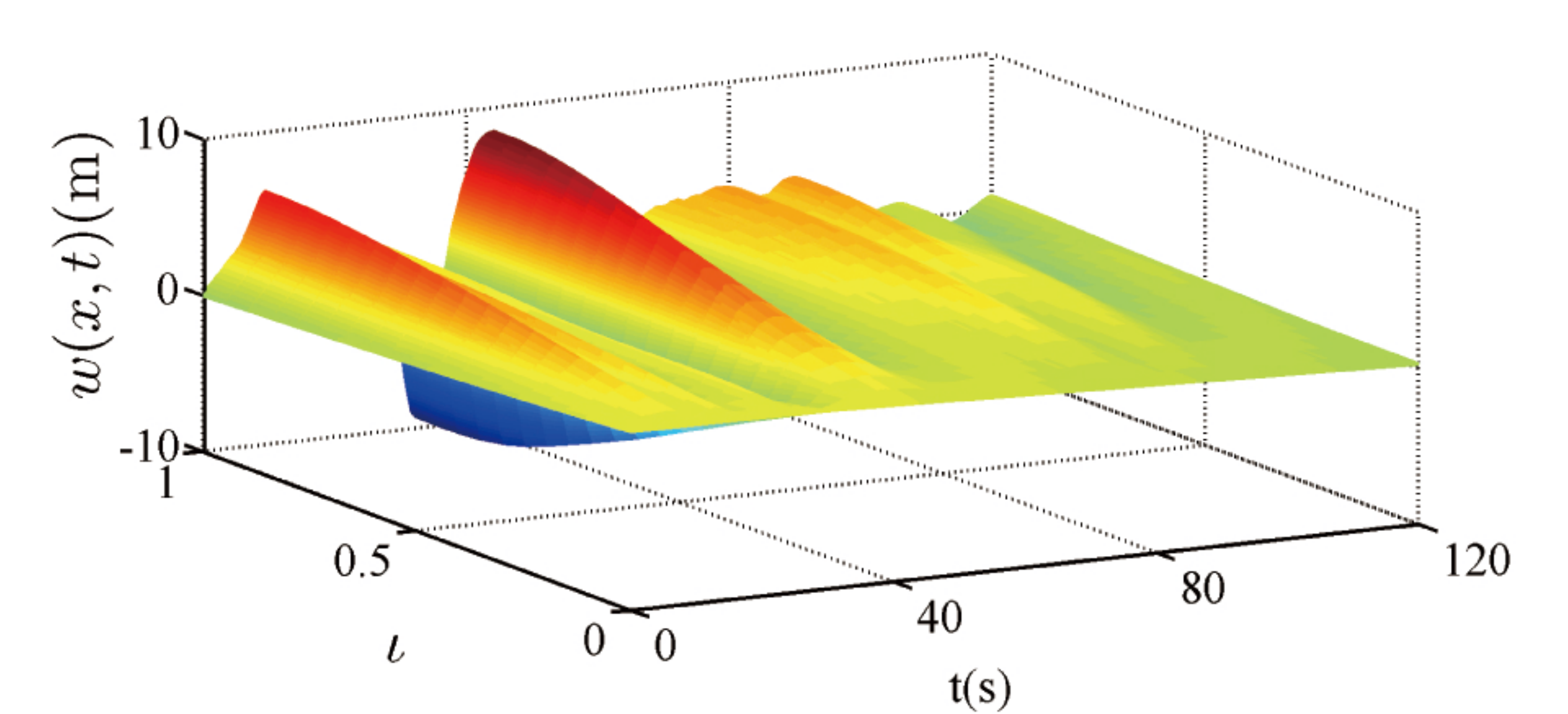}
\caption{Closed-loop responses of lateral vibrations $w(x,t)$ in the actual nonlinear model with unmodeled disturbances.}
\label{fig:wn}
\end{figure}
\begin{figure}
\centering
\includegraphics[width=8cm]{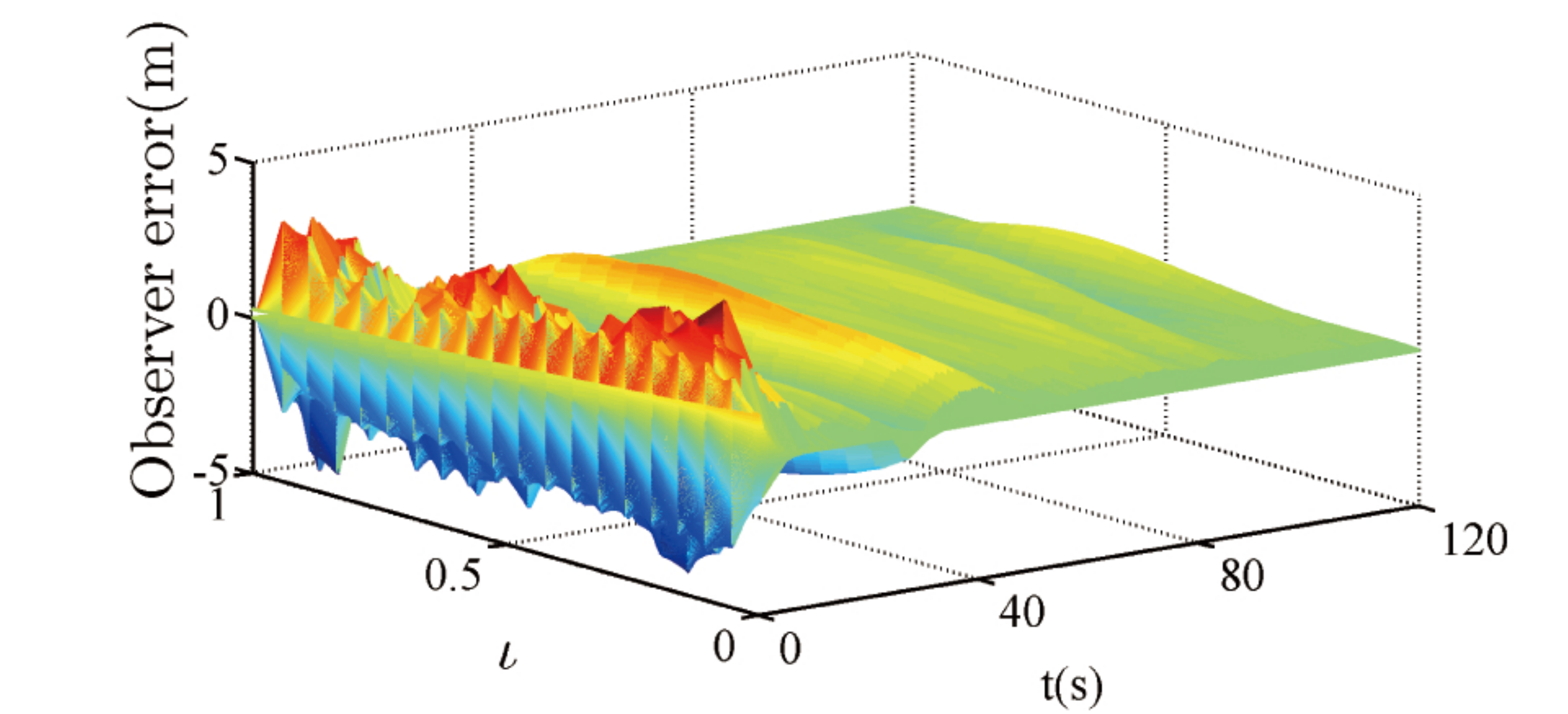}
\caption{Observer errors $\tilde u(\cdot,t)+\tilde w(\cdot,t)$ in the actual nonlinear model with unmodeled disturbances.}
\label{fig:obn}
\end{figure}
\begin{figure}
\centering
\includegraphics[width=8cm]{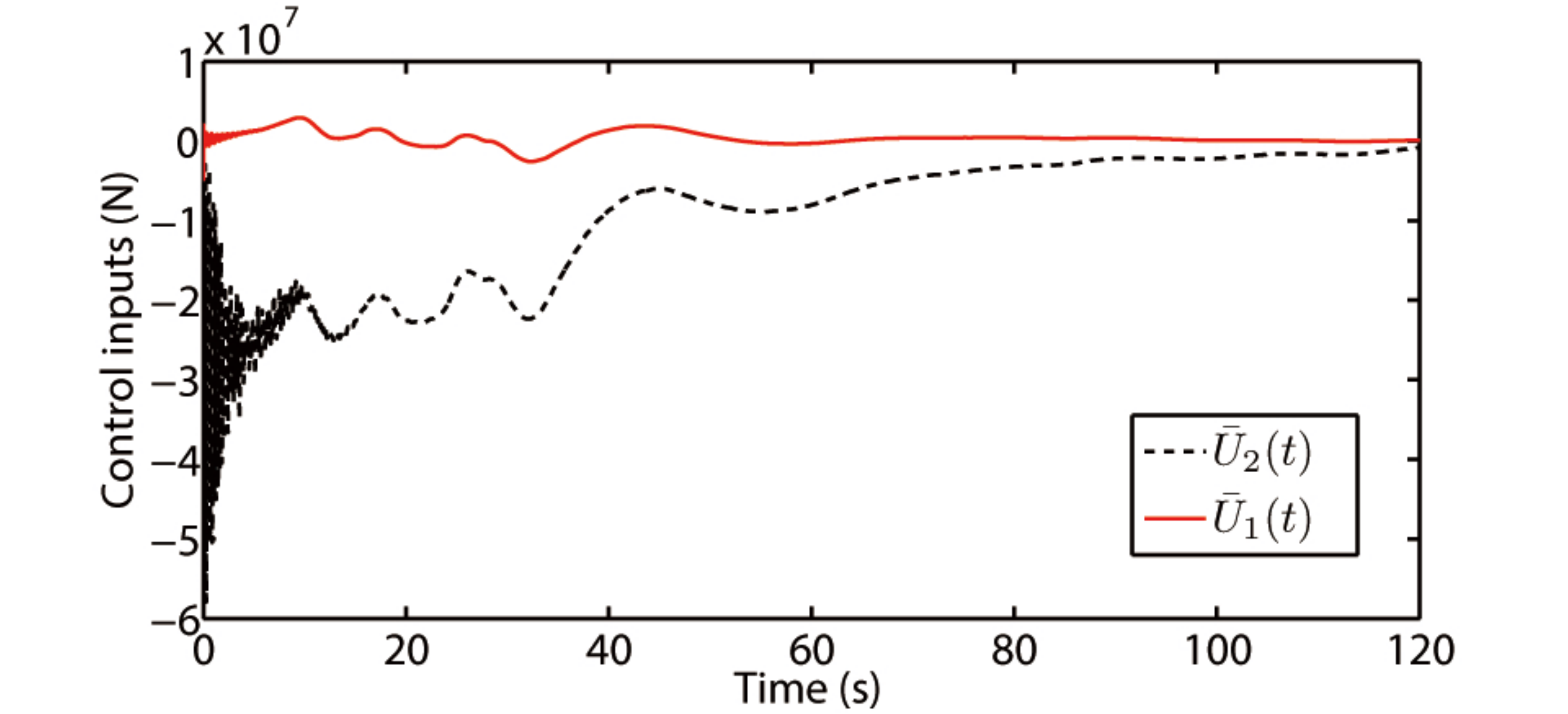}
\caption{Control forces in the actual nonlinear model.}
\label{fig:cn}
\end{figure}
The control parameters $\kappa_{11}$, $\kappa_{12}$, $\kappa_{13}$, $\kappa_{14}$ and $\kappa_{21}$, $\kappa_{22}$, $\kappa_{23}$, $\kappa_{24}$ are increased to twice and ten times of those in \eqref{eq:kappa11} respectively, considering the robust to the unmodeled disturbances. The observer parameters are kept same with those in Section \ref{sec:simlin}. Apply the proposed output-feedback controller into the actual nonlinear model with the ocean current disturbances, it is shown in Figs. \ref{fig:un}-\ref{fig:wn} that the longitudinal vibrations and lateral vibrations are reduced as time goes on. From Fig. \ref{fig:obn}, we know that the observer errors  are convergent to a small range around zero. Simulation results in Section \ref{subsecnonlinear} illustrate the effectiveness of the proposed control design applied into vibration suppression of DCV, where the output-feedback control inputs in this actual model are shown in Fig. \ref{fig:cn}.
\section{Conclusion and future work}\label{sec:Conclusion}
This work is motivated by a practical application of longitudinal-lateral vibration suppression for a deep-sea construction vessel which is used to install oil drilling equipment at the
designated locations on the seafloor. The vibration dynamics of the DCV  is derived by extended Hamilton's principle resulting in a nonlinear wave PDE model, which is then linearized around the steady state to generate a linear model used in control design. The obtained specific model of the DCV is represented as a more general coupled wave PDE plant. Through Riemann transformations, the plant is converted to a $4\times4$ coupled heterodirectional hyperbolic PDE-ODE system characterized by spatially-varying coefficients and
on a time-varying domain, based on which the controller and observer designs are conducted via the backstepping method. The exponential stability results of the observer error system and the output-feedback closed-loop system, boundedness and exponential convergence of the control inputs are proved by Lyapunov analysis. The simulation test is conducted based on the both approximated linear model and actual nonlinear model to support the obtained theoretical results and verify that the proposed ship-mounted crane control forces can effectively suppress the undesired longitudinal-lateral vibrations in DCV.

The control design is on the basis of a completely known model, but some uncertainties, such as unknown model parameters, may appear in practice. In future work, the model uncertainties will be considered and some adaptive control technologies should be incorporated into the control design.
\section*{Appendix}
\textbf{\emph{Proof of Lemma \ref{lm:wp}:}}
The boundary conditions in \eqref{eq:Kerc1}-\eqref{eq:Kerc6} are along the lines $y=x$ and $y=0$ and there are not conditions at the boundary $x=l(t)$ on the triangular domain ${\mathcal D}$. Therefore, we can extend the boundary $x=l(t)$ in $\mathcal D$ to $x=L$ considering $0<l(t)\le L$ in Assumption \ref{as:a0}, to solve $F,N$ on a fixed triangular domain ${\mathcal D_0}=\{0< y< x<L\}$ and $\lambda$ on a constant interval $0\le x\le L$ in \eqref{eq:Kerc1}-\eqref{eq:Kerc6}. Note that the lines $y=x$ and $y=0$ are overlapped boundaries of the triangular domains ${\mathcal D_0}$ and ${\mathcal D}$, and ${\mathcal D}\subseteq\mathcal D_0$. Once the solutions of \eqref{eq:Kerc1}-\eqref{eq:Kerc6} on $\mathcal D_0$ are obtained, the ones on the subset $\mathcal D$ are the required kernels $F,N$ and $\lambda$. To ensure the well-posedness of the kernel equations \eqref{eq:Kerc1}-\eqref{eq:Kerc6} on $\mathcal D_0$, we adding an additional artificial condition at $x=L$ for $N_{21}$ as done in \cite{Hu2016Control}. In the following, we will prove that there exists a unique solution of \eqref{eq:Kerc1}-\eqref{eq:Kerc6} on ${\mathcal D_0}$ by showing it is in the form of a class of well-posed equations in \cite{Meglio2017Stabilization}, which ensures there exist a unique solution $F,N\in L^{\infty}(\mathcal D)$, $\lambda\in L^{\infty}([0,l(t)])$.

Expanding \eqref{eq:Kerc1}-\eqref{eq:Kerc6} on $\mathcal D_0$, we know $F_{ij}(x,y)_{1\le i,j\le 2}$, $N_{ij}(x,y)_{1\le i,j\le 2}$ should satisfy following coupled hyperbolic PDEs:
\begin{align}
&\sqrt {{d_6}(x)} {F_{11x}}(x,y) - \sqrt {{d_6}(y)} {F_{11y}}(x,y)\notag\\
& = [\frac{{{d_{10}}(y)}}{2} - {s_1}(y)]{N_{11}}(x,y) + [\frac{{{d_{5}}(y)}}{2} - \frac{{{d_3}(y)}}{{2\sqrt {{d_6}(y)} }}]{N_{12}}(x,y)\notag\\
& + [\frac{{{d_{10}}(y)}}{2} - {s_1}(y) + {\sqrt {{d_6}(y)} ^\prime } - \frac{{{d_{10}}(x)}}{2} + {s_1}(x)]{F_{11}}(x,y)\notag\\
& + [\frac{{{d_{5}}(y)}}{2} - \frac{{{d_3}(y)}}{{2\sqrt {{d_6}(y)} }}]{F_{12}}(x,y),\label{eq:F11x}\\
&\sqrt {{d_6}(x)} {F_{12x}}(x,y) - \sqrt {{d_1}(y)} {F_{12y}}(x,y)\notag\\
& = [\frac{{{d_{9}}(y)}}{2} - \frac{{{d_7}(y)}}{{2\sqrt {{d_1}(y)} }}]{N_{11}}(x,y) + [\frac{{{d_{4}}(y)}}{2} - {s_2}(y)]{N_{12}}(x,y)\notag\\
& + [\frac{{{d_{4}}(y)}}{2} - {s_2}(y) + {\sqrt {{d_1}(y)} ^\prime } - \frac{{{d_{10}}(x)}}{2} + {s_1}(x)]{F_{12}}(x,y)\notag\\
& + [\frac{{{d_{9}}(y)}}{2} - \frac{{{d_7}(y)}}{{2\sqrt {{d_1}(y)} }}]{F_{11}}(x,y),\label{eq:F12x}\\
&\sqrt {{d_1}(x)} {F_{21x}}(x,y) - \sqrt {{d_6}(y)} {F_{21y}}(x,y)\notag\\
& = [\frac{{{d_{10}}(y)}}{2} - {s_1}(y)]{N_{21}}(x,y) + [\frac{{{d_{5}}(y)}}{2} - \frac{{{d_3}(y)}}{{2\sqrt {{d_6}(y)} }}]{N_{22}}(x,y)\notag\\
& + [\frac{{{d_{10}}(y)}}{2} - {s_1}(y) + {\sqrt {{d_6}(y)} ^\prime } - \frac{{{d_{4}}(x)}}{2} + {s_2}(x)]{F_{21}}(x,y)\notag\\
& + [\frac{{{d_{5}}(y)}}{2} - \frac{{{d_3}(y)}}{{2\sqrt {{d_6}(y)} }}]{F_{22}}(x,y),\label{eq:F21x}\\
&\sqrt {{d_1}(x)} {F_{22x}}(x,y) - \sqrt {{d_1}(y)} {F_{22y}}(x,y)\notag\\
& = [\frac{{{d_{9}}(y)}}{2} - \frac{{{d_7}(y)}}{{2\sqrt {{d_1}(y)} }}]{N_{21}}(x,y) + [\frac{{{d_{4}}(y)}}{2} - {s_2}(y)]{N_{22}}(x,y)\notag\\
& + [\frac{{{d_{4}}(y)}}{2} - {s_2}(y) + {\sqrt {{d_1}(y)} ^\prime } - \frac{{{d_{4}}(x)}}{2} + {s_2}(x)]{F_{22}}(x,y)\notag\\
& + [\frac{{{d_{9}}(y)}}{2} - \frac{{{d_7}(y)}}{{2\sqrt {{d_1}(y)} }}]{F_{21}}(x,y),\label{eq:F22x}\\
&\sqrt {{d_6}(x)} {N_{11x}}(x,y) + \sqrt {{d_6}(y)} {N_{11y}}(x,y)\notag\\
& = [{s_1}(y) + \frac{{{d_{10}}(y)}}{2}]{F_{11}}(x,y)+ [\frac{{{d_3}(y)}}{{2\sqrt {{d_6}(y)} }} + \frac{{{d_{5}}(y)}}{2}]{F_{12}}(x,y)\notag\\
& + [{s_1}(y) + \frac{{{d_{10}}(y)}}{2} - {\sqrt {{d_6}(y)} ^\prime } - {s_1}(x) - \frac{{{d_{10}}(x)}}{2}]{N_{11}}(x,y)\notag\\
&  + [\frac{{{d_3}(y)}}{{2\sqrt {{d_6}(y)} }} + \frac{{{d_{5}}(y)}}{2}]{N_{12}}(x,y),\label{eq:N11x}\\
&\sqrt {{d_6}(x)} {N_{12x}}(x,y) + \sqrt {{d_1}(y)} {N_{12y}}(x,y)\notag\\
 &= [\frac{{{d_7}(y)}}{{2\sqrt {{d_1}(y)} }} + \frac{{{d_{9}}(y)}}{2}]{F_{11}}(x,y)+ [{s_2}(y) + \frac{{{d_{4}}(y)}}{2}]{F_{12}}(x,y)\notag\\
& + [{s_2}(y) + \frac{{{d_{4}}(y)}}{2} - {\sqrt {{d_1}(y)} ^\prime } - {s_1}(x) - \frac{{{d_{10}}(x)}}{2}]{N_{12}}(x,y)\notag\\
&  + [\frac{{{d_7}(y)}}{{2\sqrt {{d_1}(y)} }} + \frac{{{d_{9}}(y)}}{2}]{N_{11}}(x,y),\label{eq:N12x}\\
&\sqrt {{d_1}(x)} {N_{21x}}(x,y) + \sqrt {{d_6}(y)} {N_{21y}}(x,y) \notag\\
&= [{s_1}(y) + \frac{{{d_{10}}(y)}}{2}]{F_{21}}(x,y)+ [\frac{{{d_3}(y)}}{{2\sqrt {{d_6}(y)} }} + \frac{{{d_{5}}(y)}}{2}]{F_{22}}(x,y)\notag\\
& + [{s_1}(y) + \frac{{{d_{10}}(y)}}{2} - {\sqrt {{d_6}(y)} ^\prime } - {s_2}(x) - \frac{{{d_{4}}(x)}}{2}]{N_{21}}(x,y)\notag\\
&  + [\frac{{{d_3}(y)}}{{2\sqrt {{d_6}(y)} }} + \frac{{{d_{5}}(y)}}{2}]{N_{22}}(x,y),\label{eq:N21x}\\
&\sqrt {{d_1}(x)} {N_{22x}}(x,y) + \sqrt {{d_1}(y)} {N_{22y}}(x,y) \notag\\
&= [\frac{{{d_7}(y)}}{{2\sqrt {{d_1}(y)} }} + \frac{{{d_{9}}(y)}}{2}]{F_{21}}(x,y)+ [{s_2}(y) + \frac{{{d_{4}}(y)}}{2}]{F_{22}}(x,y)\notag\\
& + [{s_2}(y) + \frac{{{d_{4}}(y)}}{2} - {\sqrt {{d_1}(y)} ^\prime } - {s_2}(x) - \frac{{{d_{4}}(x)}}{2}]{N_{22}}(x,y)\notag\\
&  + [\frac{{{d_7}(y)}}{{2\sqrt {{d_1}(x)} }} + \frac{{{d_{9}}(y)}}{2}]{N_{21}}(x,y)\label{eq:N22x}
\end{align}
along with the following set of boundary conditions:
 \begin{align}
&{F_{11}}(x,x) = \frac{{ - {d_{10}}(x)+2{s_1}(x)}}{{4\sqrt {{d_6}(x)} }},\label{eq:F11}\\
&{F_{12}}(x,x) = \frac{{ - {d_{9}}(x)}\sqrt {{d_1}(x)}+{d_7}(x)}{{2(\sqrt {{d_6}(x)}  + \sqrt {{d_1}(x)} )\sqrt {{d_1}(x)} }},\\
&{F_{21}}(x,x) = \frac{{ - {d_{5}}(x)}\sqrt {{d_6}(x)}+{d_3}(x)}{{2(\sqrt {{d_6}(x)}  + \sqrt {{d_1}(x)} )\sqrt {{d_6}(x)} }},\\
&{F_{22}}(x,x) = \frac{{ - {d_{4}}(x)+2s_2(x)}}{{4\sqrt {{d_1}(x)} }} ,\label{eq:F22}\\
&{N_{12}}(x,x){\rm{ = }}\frac{{\frac{{ - {d_7}(x)}}{{2\sqrt {{d_1}(x)} }} - \frac{{{d_{9}}(x)}}{2}}}{{\sqrt {{d_6}(x)}  - \sqrt {{d_1}(x)} }},\\
&{N_{21}}(x,x){\rm{ = }}\frac{{\frac{{ - {d_3}(x)}}{{2\sqrt {{d_6}(x)} }} - \frac{{{d_{5}}(x)}}{2}}}{{\sqrt {{d_1}(x)}  - \sqrt {{d_6}(x)} }},\label{eq:N21}\\
&{N_{11}}(x,0) =  - {F_{11}}(x,0) + \frac{{{d_{14}}{\lambda _{12}}(x)}}{{{d_6}(0)}} + \frac{{{d_{16}}{\lambda _{14}}(x)}}{{{d_6}(0)}},\label{eq:N110}\\
&{N_{12}}(x,0) =  - {F_{12}}(x,0) + \frac{{{d_{12}}{\lambda _{12}}(x)}}{{{d_1}(0)}} + \frac{{{d_{18}}{\lambda _{14}}(x)}}{{{d_1}(0)}},\\
&{N_{21}}(x,0) =  - {F_{21}}(x,0)\notag\\
&\quad\quad\quad\quad + \frac{{g_0}(x)\sqrt {{d_6}(0)}+{{d_{14}}{\lambda _{22}}(x)}+{{{d_{16}}{\lambda _{24}}(x)}}}{{{d_6}(0)}},\\
&{N_{22}}(x,0) =  - {F_{22}}(x,0) + \frac{{{d_{12}}{\lambda _{22}}(x)}}{{{d_1}(0)}} + \frac{{{d_{18}}{\lambda _{24}}(x)}}{{{d_1}(0)}},\label{eq:N220}\\
&N_{21}(L,y)=0,\label{eq:N21L}
\end{align}
where \eqref{eq:N21L} is the artificial boundary condition. Besides, ${\lambda _{ij}}(x)_{1\le i\le 2,1\le j\le 4}$ should satisfy the following ODEs:
\begin{align}
&\sqrt {{d_6}(x)} {\lambda _{11}}^\prime (x) + [{s_1}(x) + \frac{{{d_{10}}(x)}}{2}]{\lambda _{11}}(x) = 0,\label{eq:lambda11}\\
&\sqrt {{d_6}(x)} {\lambda _{12}}^\prime (x) + [{s_1}(x) + \frac{{{d_{10}}(x)}}{2} - {d_{11}} + \frac{{{d_{12}}}}{{\sqrt {{d_1}(0)} }}]{\lambda _{12}}(x)\notag\\
& - {\lambda _{11}}(x) - ({d_{17}} - \frac{{{d_{18}}}}{{\sqrt {{d_1}(0)} }}){\lambda _{14}}(x)\notag\\
& - 2\sqrt {{d_1}(0)} {F_{12}}(x,0) = 0,\\
&\sqrt {{d_6}(x)} {\lambda _{13}}^\prime (x) + [{s_1}(x) + \frac{{{d_{10}}(x)}}{2}]{\lambda _{13}}(x) = 0,\\
&\sqrt {{d_6}(x)} {\lambda _{14}}^\prime (x) + [{s_1}(x) + \frac{{{d_{10}}(x)}}{2} - {d_{15}} + \frac{{{d_{16}}}}{{\sqrt {{d_6}(0)} }}]{\lambda _{14}}(x)\notag\\
& - {\lambda _{13}}(x) - ({d_{13}} - \frac{{{d_{14}}}}{{\sqrt {{d_6}(0)} }}){\lambda _{12}}(x)\notag\\
& - 2\sqrt {{d_6}(0)} {F_{11}}(x,0) = 0,\\
&\sqrt {{d_1}(x)} {\lambda _{21}}^\prime (x) + [{s_2}(x) + \frac{{{d_{4}}(x)}}{2}]{\lambda _{21}}(x) + {g_0}(x){\lambda _{11}}(0){\rm{ = }}0,\\
&\sqrt {{d_1}(x)} {\lambda _{22}}^\prime (x) + [{s_2}(x) + \frac{{{d_{4}}(x)}}{2} - {d_{11}} + \frac{{{d_{12}}}}{{\sqrt {{d_1}(0)} }}]{\lambda _{22}}(x)\notag\\
& - {\lambda _{21}}(x) - ({d_{17}} - \frac{{{d_{18}}}}{{\sqrt {{d_1}(0)} }}){\lambda _{24}}(x)\notag\\
& - 2\sqrt {{d_1}(0)} {F_{22}}(x,0) + {g_0}(x){\lambda _{12}}(0) = 0,\\
&\sqrt {{d_1}(x)} {\lambda _{23}}^\prime (x) + [{s_2}(x) + \frac{{{d_{4}}(x)}}{2}]{\lambda _{23}}(x) + {g_0}(x){\lambda _{13}}(0) = 0,\\
&\sqrt {{d_1}(x)} {\lambda _{24}}^\prime (x) + [{s_2}(x) + \frac{{{d_{4}}(x)}}{2} - {d_{15}} + \frac{{{d_{16}}}}{{\sqrt {{d_6}(0)} }}]{\lambda _{24}}(x)\notag\\
& - ({d_{13}} - \frac{{{d_{14}}}}{{\sqrt {{d_6}(0)} }}){\lambda _{22}}(x) - {\lambda _{23}}(x)\notag\\
 &- 2\sqrt {{d_6}(0)} {F_{21}}(x,0) + {g_0}(x){\lambda _{14}}(0) = 0,\label{eq:lambda24}
 \end{align}
with initial conditions:
 \begin{align}
&\left[ {\begin{array}{*{20}{c}}
{{\lambda _{11}}(0)}&{{\lambda _{12}}(0)}&{{\lambda _{13}}(0)}&{{\lambda _{14}}(0)}\\
{{\lambda _{21}}(0)}&{{\lambda _{22}}(0)}&{{\lambda _{23}}(0)}&{{\lambda _{24}}(0)}
\end{array}} \right]\notag\\
&= \left[ {\begin{array}{*{20}{c}}
{{\kappa _{11}}}&{{\kappa _{12}}}&{{\kappa _{13}}}&{{\kappa _{14}}}\\
{{\kappa _{21}}}&{{\kappa _{22}}}&{{\kappa _{23}}}&{{\kappa _{24}}}
\end{array}} \right].\label{eq:lambda0}
\end{align}
\eqref{eq:F11x}-\eqref{eq:lambda0} has the same structure as the kernel equations (17)-(24) in \cite{Meglio2017Stabilization} with setting $m=n=2$. More precisely,  \eqref{eq:F11x}-\eqref{eq:F22x} corresponding to (17); \eqref{eq:N11x}-\eqref{eq:N22x} corresponding to (18); \eqref{eq:F11}-\eqref{eq:F22} corresponding to (19); \eqref{eq:N21} corresponding to (20); \eqref{eq:N110}-\eqref{eq:N220} corresponding to (21); the additional artificial boundary condition \eqref{eq:N21L} corresponding to (24);  the initial value problem \eqref{eq:lambda11}-\eqref{eq:lambda0} corresponding to (22)-(23). Please note that even thought (17)-(24) in \cite{Meglio2017Stabilization} are with constant coefficients, the well-posedness result still holds in the  case of spatially-varying coefficients as \eqref{eq:F11x}-\eqref{eq:lambda0} in this paper, because \eqref{eq:F11x}-\eqref{eq:lambda0} are in the form of a general class
of hyperbolic PDEs (30)-(31) of which the well-posedness is proved in \cite{Meglio2017Stabilization}. Therefore, \eqref{eq:F11x}-\eqref{eq:lambda0}, i.e., \eqref{eq:Kerc1}-\eqref{eq:Kerc6} are well-posed on the domain $\mathcal D_0$, which straightforwardly yields to a unique solution $F,N\in L^{\infty}(\mathcal D)$, $\lambda\in L^{\infty}([0,l(t)])$ of \eqref{eq:Kerc1}-\eqref{eq:Kerc6} because of ${\mathcal D}\subseteq\mathcal D_0$, $l(t)\le L$. Kernel equations \eqref{eq:Kerc1a}-\eqref{eq:Kerc6a} of $J(x,y),K(x,y),\gamma(x)$ have a same structure with \eqref{eq:Kerc1}-\eqref{eq:Kerc6}. Through a similar process as above with adding an additional artificial boundary condition $K_{21}(L,y)=0$, one obtain that there is a unique solution $J,K\in L^{\infty}(\mathcal D)$, $\gamma\in L^{\infty}([0,l(t)])$ of \eqref{eq:Kerc1a}-\eqref{eq:Kerc6a}. The proof of Lemma \ref{lm:wp} is completed.


\begin{thebibliography}{10}
\bibitem{Aamo2013Disturbance}
O. M. Aamo,
\newblock `` Disturbance rejection in $2\times2$ linear hyperbolic systems'',
\newblock{\em IEEE Trans. Autom. Control}, 58(5), pp.1095-1106, 2013.


\bibitem{Anfinsen2017Disturbance}
H. Anfinsen and O. M. Aamo,
\newblock `` Disturbance rejection in general heterodirectional 1-D linear hyperbolic systems using collocated sensing and control'',
\newblock{\em Automatica}, 76, pp.230-242, 2017.



\bibitem{Anfinsen2017Adaptive}
H. Anfinsen and O. M. Aamo,
\newblock `` Adaptive output-feedback stabilization of linear $2\times 2$ hyperbolic systems using anti-collocated sensing and control'',
\newblock{\em Systems $\&$ Control Letters}, 104, pp.86-94, 2017.


\bibitem{Anfinsen2018}
H. Anfinsen and O.M. Aamo,
\newblock ``Adaptive control of linear $2\times2$ hyperbolic systems'',
\newblock{\em Automatica}, 87, pp.69-82, 2018.

\bibitem{Bohm2014}
M. Bohm, M. Krstic, S. Kuchler and O. Sawodny,
\newblock ``Modeling and boundary control
of a hanging cable immersed
in water'',\newblock{\em Journal of Dynamic Systems, Measurement, and Control},136, pp. 011006, 2014.


\bibitem{Jalon1994}
J.G. de Jalon and E. Bayo, \newblock{\em Kinematic and Dynamic Simulation of Multibody Systems}, Springer, New York, 1994.

\bibitem{Geradin1994}
M. Geradin and D. Rixen, \newblock{\em Mechanical Vibrations: Theory and Application to Structural Dynamics}, Wiley, New
York, 1994.

\bibitem{bekiaris2014compensation}
N. Bekiaris-Liberis and M. Krstic,
\newblock ``Compensation of wave actuator dynamics for nonlinear systems'',
\newblock {\em IEEE Trans. Autom. Control}, 59(6), pp. 1555-1570,


\bibitem{Coron2013Local}
J.M. Coron, R. Vazquez, M. Krstic and G. Bastin,
\newblock `` Local exponential $H^2$ stabilization of a $2\times2$ quasilinear hyperbolic system using backstepping'',
\newblock {\em SIAM Journal on Control and Optimization}, 51(3), pp.2005-2035, 2013.


\bibitem{cai2016nonlinear}
X. Cai and M. Krstic,
\newblock ``Nonlinear stabilization through wave {PDE} dynamics with a moving
  uncontrolled boundary'',
\newblock {\em Automatica}, vol. 68, pp. 27-38, 2016.


\bibitem{Deutscher2017Finite-time}
J. Deutscher,
\newblock ``Finite-time output regulation for linear $2\times2$ hyperbolic systems using backstepping''.
\newblock {\em Automatica}, 75, pp.54-62, 2017.

\bibitem{Deutscher2017Output}
J. Deutscher,
\newblock ``Output regulation for general linear heterodirectional hyperbolic
systems with spatially-varying coefficients''.
\newblock {\em Automatica}, 85, pp.34-42, 2017.

\bibitem{Deutscher2018Output1}
J. Deutscher, N. Gehring and R. Kern
\newblock ``Output feedback control of general linear heterodirectional hyperbolic PDE-ODE
systems with spatially-varying coefficients'',
\newblock {\em Int. J. Control}, 92, pp.2274-2290, 2019.

\bibitem{Faltinsen1993}
Faltinsen, Odd. \newblock ``Sea loads on ships and offshore structures'',  Cambridge university press, 1993.

\bibitem{Fossen2002}
T. I. Fossen, \newblock ``Marine Control Systems: Guidance, Navigation and Control
of Ships, Rigs and Underwater Vehicles'', Norway: Marine Cybernetics
AS, 2002.

\bibitem{How2011}
B. How, S.S. Ge and Y. S. Choo
\newblock ``Control of Coupled Vessel, Crane, Cable, and Payload
Dynamics for Subsea Installation Operations''.
\newblock {\em IEEE Transactions on Control Systems Technology} 19, pp. 208-220, 2011.

\bibitem{Hu2016Control}
L. Hu, F. Di Meglio, R. Vazquez and M. Krstic,
\newblock ``Control of homodirectional and general heterodirectional linear coupled hyperbolic PDEs'',
\newblock{\em IEEE Trans. Autom. Control}, 61(11), pp.3301-3314, 2016.

\bibitem{He2014Modeling}
W. He, F. S.S. Ge and D. Huang,
\newblock ``Modeling and Vibration Control for a Nonlinear
Moving String With Output Constraint'',
\newblock{\em IEEE/ASME Trans. Mechatron.}, 20(4), pp.1886-1897, 2014.

\bibitem{He2017Boundary}
X.Y. He, W. He,  J. Shi and C.Y. Sun, \newblock ``Boundary vibration control of variable length
crane systems in two-dimensional space with
output constraints'',\newblock{\em IEEE/ASME Trans. Mechatron.}, 22, pp.1952-1962, 2017.

\bibitem{2009compensating}
M. Krstic,
\newblock ``Compensating a string {PDE} in the actuation or sensing path of an unstable {ODE}'',
\newblock {\em IEEE Trans. Autom. Control}, 54(6), pp. 1362-1368,
  2009.


\bibitem{Liu2017Modeling}
Z. Liu, J. Liu and W. He, \newblock ``Modeling and vibration control of a flexible aerial refueling hose with
variable lengths and input constraint'', \newblock {\em Automatica}, 77, pp. 302-310, 2017.


\bibitem{Meglio2013Stabilization}
F. Di Meglio, R. Vazquez and M. Krstic,
\newblock `` Stabilization of a system of $n+1$ coupled first-order hyperbolic linear PDEs with a single boundary input'',
\newblock {\em IEEE Trans. Autom. Control}, 58(12), pp.3097-3111, 2013.

\bibitem{Meglio2017Stabilization}
F. Di Meglio, F. Bribiesca, L. Hu and M. Krstic,
\newblock `` Stabilization of coupled linear heterodirectional hyperbolic PDE-ODE systems'',
\newblock {\em Automatica}, 87, pp.281-289, 2018.

\bibitem{Kaczmarczyk2003}
S. Kaczmarczyk and W. Ostachowicz,
\newblock `` Transient Vibration Phenomena
in Deep Mine Hoisting Cables-Part 1: Mathematical Model'', \newblock {\em J. Sound Vib.},
262(2), pp. 219-244, 2003.



%



\bibitem{Stensgaard2010Subsea}
 T. Stensgaard, C. White and K. Schiffer, \newblock `` Subsea Hardware Installation
from a FDPSO''. In Offshore Technology Conference, pp. 1-6, 2010.

\bibitem{Standing2002Enhancing}
R.G. Standing, B.G. Mackenzie and R.O. Snell, \newblock `` Enhancing the
technology for deepwater installation of subsea hardware''. In Offshore
Technology Conference, pp. 1-6, 2002.




\bibitem{Susto2010Control}
G.A. Susto and M. Krstic,
\newblock ``Control of {PDE}-{ODE} cascades with neumann interconnections'',
\newblock {\em J. Franklin Inst.}, 347(1), pp. 284-314, 2010.



\bibitem{Vazquez2011Backstepping}
R. Vazquez, M. Krstic and J.M. Coron,
\newblock ``Backstepping boundary stabilization and state estimation of a $2\times2$ linear hyperbolic system'',
In Decision and Control and European Control Conference (CDC-ECC), 50th IEEE Conference on (pp. 4937-4942), 2011.

\bibitem{J2020delay}
J. Wang and M. Krstic,
\newblock ``Delay-compensated control of sandwiched ODE-PDE-ODE
hyperbolic systems for oil drilling and disaster relief'', {\em Automatica}, vol. 120, pp. 109131, 2020.

\bibitem{wang2017Modeling}
J. Wang, Y. Pi, Y. Hu, and X. Gong,
\newblock ``Modeling and dynamic behavior analysis of a coupled multi-cable double drum winding hoister with flexible guides'',
\newblock {\em Mech. Mach. Theory}, vol. 108,  pp. 191-208, 2017.



\bibitem{wang2019Adaptive}
J. Wang, S.-X. Tang and M. Krstic,
\newblock ``Adaptive output-feedback control of torsional vibration in
off-shore rotary oil drilling systems'',
\newblock {\em Automatica}, 111, pp. 108640, 2020.

\bibitem{J2017Axial}
J. Wang, S. Koga, Y. Pi and M. Krstic,
\newblock ``Axial vibration suppression in a PDE Model of ascending mining cable elevator'',
\newblock {\em J. Dyn. Sys., Meas., Control.}, 140, pp. 111003, 2018.

\bibitem{J2017Exponential}
J. Wang, S.-X. Tang, Y. Pi and M. Krstic,
\newblock ``Exponential regulation of the anti-collocatedly disturbed cage in a wave PDE-modeled ascending cable elevator'',
\newblock {\em Automatica}, 95(2018), pp. 122-136.

\bibitem{J2018Balancing}
J. Wang, Y. Pi and M. Krstic,
\newblock ``Balancing and suppression of oscillations of tension and cage in dual-cable mining elevators'',
\newblock {\em Automatica}, 98, pp. 223-238, 2018.

\bibitem{J2017Control}
J. Wang, M. Krstic and Y. Pi,
\newblock ``Control of a $2\times2$ coupled linear hyperbolic system sandwiched between two ODEs'',
\newblock {\em Int. J. Robust Nonlin.}, 28, pp. 3987-4016, 2018.



\bibitem{zhu2001}
W. D. Zhu, J. Ni and J. Huang, \newblock ``Active Control of Translating
Media With Arbitrarily Varying
Length'', {\em ASME Journal of Vibration and Acoustics}, 123, pp. 347-358.

\bibitem{zhang2012modeling}
S. Zhang, W. He, and S.S. Ge,
\newblock ``Modeling and control of a nonuniform vibrating string under
  spatiotemporally varying tension and disturbance,''
\newblock {\em IEEE/ASME Trans. Mechatronics}, vol. 17, no. 6, pp. 1196-1203,
  2012.
\end{thebibliography}
\end{document}